\documentclass[11pt]{amsproc}

\usepackage{amscd, amsmath, amsthm, amssymb, mathtools, mathrsfs, stmaryrd}
\usepackage{ascmac}
\usepackage{setspace}
\usepackage{comment}
\usepackage{bm, bbm}
\allowdisplaybreaks

\usepackage{fullpage}

\usepackage{hyperref}

\usepackage{tikz}
\usetikzlibrary{intersections, calc, arrows.meta}

\usepackage{here}
\usepackage{time}
\usepackage[abbrev]{amsrefs}

\usepackage{xcolor}
\usepackage[capitalize,nameinlink,noabbrev,nosort]{cleveref}
\hypersetup{
	colorlinks=true,       % false: boxed links; true: colored links
	linkcolor=brown,          % color of internal links
	citecolor=brown,        % color of links to bibliography
	filecolor=brown,      % color of file links
	urlcolor=brown,           % color of external links
}

\makeatletter
\@namedef{subjclassname@2020}{\textup{2020} Mathematics Subject Classification}
\makeatother

\newtheorem{theoremcounter}{Theorem Counter}[section]

\theoremstyle{remark}

\theoremstyle{definition}
\newtheorem{definition}[theoremcounter]{Definition}

\theoremstyle{plain}
\newtheorem{lemma}[theoremcounter]{Lemma}
\newtheorem{proposition}[theoremcounter]{Proposition}
\newtheorem{corollary}[theoremcounter]{Corollary}

\newtheorem{theorem}[theoremcounter]{Theorem}

\numberwithin{equation}{section}

\newcommand{\Z}{\mathbb{Z}}

\newcommand{\R}{\mathbb{R}}
\newcommand{\C}{\mathbb{C}}

\newcommand{\calC}{\mathcal{C}}
\newcommand{\calD}{\mathcal{D}}
\newcommand{\calE}{\mathcal{E}}

\newcommand{\dd}{\mathrm{d}}
\newcommand{\bbH}{\mathbb{H}}

\newcommand{\pmat}[1]{\begin{pmatrix}#1\end{pmatrix}}

\newcommand{\smat}[1]{\bigl(\begin{smallmatrix}#1\end{smallmatrix}\bigr)}

\def\M{\mathrm{M}}% matrix space
\def\gl{\mathfrak{gl}}
\def\sl{\mathfrak{sl}}
\def\spo{\mathfrak{spo}}

\def\sp{\mathfrak{sp}}
\def\so{\mathfrak{so}}

\def\t{{}^t\hspace{-.2em}}% transpose
\def\st{{}^{st}}% super-transpose
\def\largest{{\begin{matrix}\phantom{a} \\ \phantom{a}\end{matrix}}^{st\hspace{-.4em}}}% large super-transpose
\def\1{\mathbbm{1}}
\DeclareMathOperator{\diag}{diag} % diagonal matrix
\DeclareMathOperator{\vol}{vol}% volume
\DeclareMathOperator{\tr}{tr}% trace
\DeclareMathOperator{\str}{str}% super-trace
\DeclareMathOperator{\Aut}{Aut}

\DeclareMathOperator{\sgn}{sgn}% the sign character
\newcommand{\fg}{\mathfrak{g}}
\newcommand{\fh}{\mathfrak{h}}

\newcommand{\fq}{\mathfrak{q}}

\newcommand{\fS}{\mathfrak{S}}

\newcommand{\cC}{\mathcal{C}}

\newcommand{\cS}{\mathcal{S}}

\begin{document}
% --------------------------------------------------------------------------

\title[]{ Indefinite theta functions arising from affine Lie superalgebras and sums of triangular numbers}

\author{Toshiki Matsusaka}
\address{Faculty of Mathematics, Kyushu University, Motooka 744, Nishi-ku, Fukuoka 819-0395, Japan}
\email{matsusaka@math.kyushu-u.ac.jp}

\author[]{Miyu Suzuki}
\address{Department of Mathematics, Kyoto University, Kitashirakawa Oiwake-cho, Sakyo-ku,  Kyoto 606-8502, Japan}
\email{suzuki.miyu.4c@kyoto-u.ac.jp}

\subjclass[2020]{11F27,  17B10}

%11F27  	Theta series; Weil representation; theta correspondences
%17B10  	Representations of Lie algebras and Lie superalgebras, algebraic theory (weights)

% --------------------------------------------------------------------------

\maketitle

\begin{abstract}
We extend the recently developed theory of Roehrig and Zwegers on indefinite theta functions to prove certain power series are modular forms.
As a consequence,  we obtain several power series identities for powers of the generating function of triangular numbers.
We also show that these identities arise as specializations of denominator identities of affine Lie superalgebras.
\end{abstract}
% --------------------------------------------------------------------------

% --------------------------------------------------------------------------
\section{Introduction}
% --------------------------------------------------------------------------

In 1828, Legendre~\cite[p.134]{Legendre1828} discovered the $q$-series identity
    \[
	\left(\sum_{n=0}^\infty q^{\frac{n(n+1)}{2}} \right)^4 = \sum_{n=0}^\infty \frac{(2n+1) q^n}{1-q^{2n+1}}.
    \]
The right-hand side coincides with the generating function $\sum_{n=0}^\infty \sigma_1(2n+1)q^n$, where $\sigma_1(m)$ denotes the sum of positive divisors of $m$. This identity therefore implies that the number of representations of a positive integer $n$ as a sum of four triangular numbers is given by $\sigma_1(2n+1)$. Although this result is weaker than Gauss' Eureka theorem, which asserts that every positive integer is the sum of three triangular numbers, Legendre's identity offers a remarkably simple and explicit formula for counting such representations. The study of triangular number representations has a long history, and a collection of classical results on this topic can be found in Dickson's book~\cite[Chapter 1]{Dickson1966}.

The central topic of this article is to reinterpret Legendre's identity and its generalizations arising from affine Lie superalgebras not merely as $q$-series identities, but as identities between modular forms. To this end, let us define
     \[
	\triangle(q) \coloneqq \sum_{n=0}^\infty q^{\frac{n(n+1)}{2}}.
    \]
As noted by Ono--Robins--Wahl~\cite[Theorem 3]{OnoRobinsWahl1995}, the function $q \triangle(q^2)^4$ is a modular form of weight $2$ on $\Gamma_0(4)$, and the corresponding right-hand side, $\sum_{n=0}^\infty \sigma_1(2n+1)q^{2n+1}$, is the Eisenstein series of the same weight and level. From this perspective, Legendre's identity 
    \[
	q \triangle(q^2)^4 = \sum_{n=0}^\infty \sigma_1(2n+1) q^{2n+1}
    \]
follows naturally from standard properties of modular forms. This viewpoint also applies to any higher powers of $\triangle(q)$. For instance, Ono--Robins--Wahl~\cite[Theorem 8]{OnoRobinsWahl1995} considered the modular forms of weight $12$ and obtained the identity
\begin{align}\label{eq:ORW-12}
	q^3 \triangle(q)^{24} = \frac{1}{176896} \sum_{n=3}^\infty \left(\sigma_{11}^{(2)}(n) - \tau(n) - 2072 \tau \left(\frac{n}{2}\right)\right) q^n,
\end{align}
where $\sigma_k^{(2)}(n) = \sum_{\substack{d \mid n \\ n/d \equiv 1\ (2)}} d^k$ and $\tau(n)$ is the Ramanujan tau function.

Another method to obtain expressions for higher powers of $\triangle(q)$ involves the denominator identities for affine Lie superalgebras. The following notable examples, conjectured by Kac--Wakimoto~\cite[Section 7]{KacWakimoto1994} in connection with the affine Lie superalgebra $\widehat{\fq}(N)$,  were later proved by Zagier~\cite{Zagier2000} using a similar modular approach.

\begin{theorem}[{Zegier~\cite{Zagier2000} and Milne~\cite{Milne2002}, independently}]\label{thm:Zagier-Q}
	For every integer $m \ge 1$, we have
	\[
		q^{\frac{m^2}{2}} \triangle(q)^{4m^2} = \frac{4^m (2m)!}{\prod_{j=1}^m (2j)!^2} \sum_{\substack{x_j, y_j \in 1/2 + \Z_{\ge 0} \\ 1 \le j \le m}} \left(\prod_{j=1}^m x_j \cdot \prod_{1 \le i < j \le m}(x_i^2 - x_j^2)^2 \right) q^{2\sum_{j=1}^m x_j y_j}
	\]
	and
	\[
		q^{\frac{m(m+1)}{2}} \triangle(q)^{4m(m+1)} = \frac{4^m}{\prod_{j=1}^m (2j)!^2} \sum_{\substack{x_j \in \Z_{\ge 0} \\ y_j \in 1/2 + \Z_{\ge 0}}} \left(\prod_{j=1}^m x_j^3 \cdot \prod_{1 \le i < j \le m}(x_i^2 - x_j^2)^2 \right) q^{2\sum_{j=1}^m x_j y_j}.
	\]
	In particular, the case $m=1$ of the first identity recovers Legendre's identity.
\end{theorem}

At first glance, the appearance of intricate polynomials in $x_j$ might obscure the modular nature of these identities. However,  an important observation is that these polynomials depend only on $x_j$'s,  not on $y_j$'s,  and each $x_j$ appears with an odd degree of at least $3$. 
This allows us to rewrite the right-hand side as a sum of products of Eisenstein series. 
For instance,  the second identity with $m=2$ yields 
\begin{align}\label{eq:Zagier-second-m2}
\begin{split}
	q^3 \triangle(q)^{24} &= \frac{1}{144} \sum_{\substack{x_1, x_2 \in \Z_{\ge 0} \\ y_1, y_2 \in 1/2 + \Z_{\ge 0}}} x_1^3 x_2^3 (x_1^2 - x_2^2)^2 q^{2(x_1 y_1 + x_2 y_2)}\\
		&= \frac{1}{72} \bigg(E_4^{(2)}(\tau) E_8^{(2)}(\tau) - E_6^{(2)}(\tau)^2 \bigg),
\end{split}
\end{align}
where
\begin{align}\label{eq:Eis-level-2}
	E_k^{(2)}(\tau) = \sum_{\substack{x \in \Z_{\ge 0} \\ y \in 1/2 + \Z_{\ge 0}}} x^{k-1} q^{2xy} = \sum_{n=1}^\infty \sigma_{k-1}^{(2)}(n) q^n
\end{align}
is the Eisenstein series of weight $k$ on $\Gamma_0(2)$. 
Once both sides are seen to share the same modular properties, the identity~\eqref{eq:Zagier-second-m2} follows naturally. 
This is the key insight underlying Zagier's approach. 
Moreover, one can also recover the equation~\eqref{eq:ORW-12} from \eqref{eq:Zagier-second-m2} by expressing the products of Eisenstein series as a linear combination of modular forms of weight $12$.

This approach via Eisenstein series does not readily generalize to other analogous identities. 
For instance, the following identity is strongly supported by numerical evidence, but the polynomial factor on the right-hand side involves both $x_j$'s and $y_j$'s,  which makes it difficult to apply Zagier's argument. 
As a result, a similar proof becomes considerably more challenging.

\begin{theorem}\label{thm:Dmm-Eisen}
	For every integer $m \ge 1$, we have
	\[
		q^{\frac{m^2}{2}} \triangle(q)^{4m^2} = \frac{2^m (2m)!}{\prod_{j=1}^m (2j)!^2} \sum_{x_j, y_j \in 1/2+\Z_{\ge 0}} f(\vec{\bm{x}}) q^{2\sum_{j=1}^m x_j y_j}
	\]
	where we put
	\[
		f(\vec{\bm{x}}) \coloneqq \prod_{j=1}^m (x_j + y_j) \cdot \prod_{1 \le i < j \le m} \bigg((x_i-y_i)^2 - (x_j-y_j)^2 \bigg)\bigg((x_i+y_i)^2 - (x_j+y_j)^2\bigg).
	\]
\end{theorem}

To overcome this difficulty,  we propose a more flexible and general approach. Instead of rewriting the right-hand side in terms of Eisenstein series,  we view it as a theta function associated with the quadratic form $2\sum_{j=1}^m x_j y_j$,  and study its modular transformation. 
Note that classical theta functions (as modular forms) are attached to positive definite quadratic forms,  while the above one is indefinite.

The central idea here is to broaden the framework of modular forms by employing theta functions associated with indefinite quadratic forms. 
The broader perspective has given rise to a powerful theory, notably through Zwegers' foundational work~\cite{Zwegers2002} on Ramanujan's mock theta functions. 
More recently,  Roehrig--Zwegers~\cite{RoehrigZwegers2022, RoehrigZwegers2022-pre} have shown that even Eisenstein series \eqref{eq:Eis-level-2} can be interpreted as indefinite theta functions. 
Motivated by this viewpoint,  we aim to further extend the theory to include series of the type appearing in \cref{thm:Dmm-Eisen}.

We remark that, as with \cref{thm:Zagier-Q} (originating in the work of Kac--Wakimoto~\cite{KacWakimoto1994} on the affine Lie superalgebra $\widehat{\mathfrak{q}}(N)$), the identity in \cref{thm:Dmm-Eisen} is also related to a denominator identity for an affine Lie superalgebra. 
Specifically,  it corresponds to $\widehat{\spo}(2m,2m)$,  as we will explain in \cref{sec:spo^(2m_2m)}. 
A similar identity corresponding to $\widehat{\spo}(2m,2m+1)$,  predicted by Kac--Wakimoto~\cite[Conjecture 5.1]{KacWakimoto1994},  was proved in our previous work~\cite{MatsusakaSuzuki2025-pre} following the modular approach based on indefinite theta functions.

In this article (together with our previous work \cite{MatsusakaSuzuki2025-pre}),  we treat all the affine Lie superalgebras whose denominator identities provide power series identities for some powers of $\triangle(q)$.
As Kac and Wakimoto explained in the introduction to \cite{KacWakimoto1994},  such an affine Lie superalgebra $\widehat{\fg}$ with non-zero dual Coxeter number are listed as $\widehat{\fg}=\widehat{\sl}(m+1,  m)$,  $\widehat{\spo}(2m,  2m+1)$,  $\widehat{\spo}(2m,  2m)$. 
We also treat the case of zero dual Coxeter numbers such as $\widehat{\fg}=\widehat{\gl}(m,  m)$,  $\widehat{\spo}(2m,  2m+2)$,  and $\widehat{\fq}(N)$.

\begin{table}[htb]
\begin{tabular}{|c||c|c|c|c|} \hline  \rule{0mm}{5mm}
& $\widehat{\gl}(m,  m)$ & $\widehat{\sl}(m+1,  m)$  & $\widehat{\spo}(2m,  2m+1)$  & $\widehat{\spo}(2m,  2m)$   \\ \hline\hline
 $\begin{array}{c}
 \text{indefinite}  \\ 
 \text{theta function}  
 \end{array}$ & ?  & ? &  \cite{MatsusakaSuzuki2025-pre} 
&  \cref{sec:third_identity} \\  \hline
$\begin{array}{c}
 \text{denominator} \\ 
 \text{identity}
\end{array}$  & \cref{sec:gl^(m_m)}  & \cref{sec:sl^(m+1_m)} & \cite{MatsusakaSuzuki2025-pre} 
&  \cref{sec:spo^(2m_2m)}  \\  \hline   
\end{tabular} \vspace{5pt}

\begin{tabular}{|c||c|c|c|} \hline \rule{0mm}{5mm}
& $\widehat{\spo}(2m,  2m+2)$ & $\widehat{\fq}(2m-1)$  & $\widehat{\fq}(2m)$  \\ \hline\hline
 $\begin{array}{c}
 \text{indefinite}  \\ 
 \text{theta function}  
 \end{array}$ & 
 $\begin{array}{c}
 \text{the same }  \\ 
 \text{as $\widehat{\fq}(2m)$}  
 \end{array}$ & \cref{sec:first_identity} & \cref{sec:second_identity} \\  \hline
 $\begin{array}{c}
 \text{denominator} \\ 
 \text{identity}
\end{array}$ & \cref{sec:spo^(2m_2m+2)} & omitted & omitted \\  \hline   
  \end{tabular} 
 \end{table}

Our first goal is to extend this framework developed in our previous work~\cite{MatsusakaSuzuki2025-pre} to provide uniform proofs of \cref{thm:Zagier-Q} and \cref{thm:Dmm-Eisen} in the realm of indefinite theta functions. 
We then show how these identities are derived from the denominator identities of affine Lie superalgebras. 
That is to say,  we provide two different proofs for each identity as in our previous work~\cite{MatsusakaSuzuki2025-pre}.

However,  we were not able to give a proof based on the theory of indefinite theta functions to the identities obtained from $\widehat{\sl}(m+1,m)$ and $\widehat{\gl}(m,m)$. A detailed explanation is given in \cref{sec:Final-remark}.
In addition,  we omit the proof that we can deduce \cref{thm:Zagier-Q} from the denominator identities for $\widehat{q}(N)$ since the argument is the same as the other cases and these identities are not new.
The status of each case and the corresponding sections are summarized in the table above.

This article is organized as follows. In the first part, we study indefinite theta functions. In \cref{sec:indefinite-theta}, we review Vign\'{e}ras' criterion for the construction of indefinite theta functions and present an extension of Roehrig--Zwegers' indefinite theta functions in the special case of signature $(m,m)$. 
In \cref{sec:proof_via_theta},  we prove \cref{thm:Zagier-Q} and \cref{thm:Dmm-Eisen} using this extended framework.

In the latter part of this article, we reinterpret the $q$-series identities established here from the perspective of affine Lie superalgebras. 
In \cref{sec:affile-Lie-super}, we prepare necessary notations for affine Lie superalgebras and root systems which we use in the next section.
\cref{sec:denom-identities} is devoted to deduce certain power seires identities for powers of $\triangle(q)$ from the denominator identities of $\widehat{\gl}(m,  m)$,  $\widehat{\sl}(m+1,  m)$,  $\widehat{\spo}(2m,  2m)$,  and $\widehat{\spo}(2m,  2m+2)$.

% --------------------------------------------------------------------------
\section*{Acknowledgment}
% --------------------------------------------------------------------------

The first author was supported by JSPS KAKENHI (JP21K18141 and JP24K16901) and the MEXT Initiative through Kyushu University's Diversity and Super Global Training Program for Female and Young Faculty (SENTAN-Q). The second author was supported by JSPS KAKENHI (JP22K13891 and JP23K20785).

% --------------------------------------------------------------------------
\section{Indefinite theta functions}\label{sec:indefinite-theta}
% --------------------------------------------------------------------------

The aim of this section is to present a general framework for directly establishing the modularity of the right-hand sides of \cref{thm:Zagier-Q} and \cref{thm:Dmm-Eisen} as theta functions associated with the indefinite quadratic form $Q_m(\vec{\bm{x}}) = \sum_{j=1}^m x_j y_j$. 
We begin by recalling Vign\'{e}ras' criterion, a standard tool for constructing such theta functions.

% --------------------------------------------------------------------------
\subsection{Vign\'{e}ras' criterion}
% --------------------------------------------------------------------------

For a symmetric matrix $A$ of size $n = r+s$ and signature $(r, s)$, we define the quadratic form $Q(\bm{x}) = \frac{1}{2} {}^t \bm{x} A \bm{x}$ and the associated bilinear form $B(\bm{x}, \bm{y}) = {}^t \bm{x} A \bm{y} = Q(\bm{x} + \bm{y}) - Q(\bm{x}) - Q(\bm{y})$. Let $L \subset \R^n$ be a lattice such that $B(\bm{x}, \bm{y}) \in \Z$ for all $\bm{x}, \bm{y} \in L$. We define its dual lattice $L'$ by
     \[
	L' = \{\bm{x} \in \R^n \mid B(\bm{x}, \bm{y}) \in \Z \text{ for all $\bm{y} \in L$}\}.
     \]

Let $\partial_{\bm{x}} = {}^t(\frac{\partial}{\partial x_1}, \dots, \frac{\partial}{\partial x_n})$. 
Using the \emph{Euler operator} $\calE$ and the \emph{Laplace operator} $\Delta_A$,  defined respectively by
    \[
	\calE \coloneqq {}^t \bm{x} \partial_{\bm{x}}, \quad \Delta_A \coloneqq {}^t \partial_{\bm{x}} A^{-1} \partial_{\bm{x}},
    \]
we introduce \emph{Vign\'{e}ras' operator} $\calD_A$ by
    \[
	\calD_A \coloneqq \calE - \frac{\Delta_A}{4\pi}.
    \]

\begin{theorem}[Vign\'{e}ras~\cite{Vigneras1977-proc}]\label{fact:Vig}
	Let $p\,\colon \R^n \to \C$ be a function satisfying the following assumptions:
	\begin{itemize}
		\item[$(1)$] $p(\bm{x}) e^{-2\pi Q(\bm{x})}$ is a Schwartz function.
		\item[$(2)$] $p(\bm{x})$ is an eigenfunction of the operator $\calD_A$ with the eigenvalue $\lambda \in \Z$, that is, $\calD_A p(\bm{x}) = \lambda p(\bm{x})$.
	\end{itemize}
	Then, for any $\bm{a}, \bm{b} \in \R^n$, the theta function defined by
	\[
		\theta_{\bm{a}, \bm{b}}(\tau) \coloneqq v^{-\lambda/2} \sum_{\bm{x} \in \bm{a} + L} p (\bm{x} \sqrt{v}) q^{Q(\bm{x})} e^{2\pi i B(\bm{x}, \bm{b})}, \qquad (\tau = u + iv)
	\]
	converges absolutely and satisfies the transformation formula
	\[
		\theta_{\bm{a}, \bm{b}} \left(-\frac{1}{\tau}\right) = \frac{i^{-s-\lambda} (-i\tau)^{\lambda+n/2} e^{2\pi i B(\bm{a}, \bm{b})}}{\vol(\R^n/L) |\det A|^{1/2}} \sum_{\bm{\nu} \in L'/L} \theta_{\bm{\nu}-\bm{b}, \bm{a}} (\tau).
	\]
\end{theorem}

% --------------------------------------------------------------------------
\subsection{Indefinite theta functions of signature $(m,m)$}\label{indefinite-theta}
% --------------------------------------------------------------------------

In the following sections, we apply Vign\'{e}ras' criterion in the case $L = \Z^{2m}$ and 
    \[
	A = \pmat{0 & \1_m \\ \1_m & 0}.
   \]
to realize $q$-series identities arising from the denominator identities of affine Lie superalgebras as indefinite theta functions. 
Our main task is to find a suitable function $p(\bm{x})$. We begin by introducing some necessary notation. 

As in the previous article~\cite{MatsusakaSuzuki2025-pre}, we write $(x_1, \dots, x_m, y_1, \dots, y_m) \in \R^{2m}$ as $\vec{\bm{x}} = (\bm{x}_1, \dots, \bm{x}_m)$, where for each $1 \le j \le m$, $\bm{x}_j = \smat{x_j \\ y_j}$. 
For a vector $\vec{x} = (x_1, \dots, x_m) \in \R^m$, we set $|\vec{x}| \coloneqq x_1 + \cdots + x_m$. 
To indicate the dimension explicitly, we attach the subscript $m$ (instead of $A$) to symbols. 
For instance,
    \[
	Q_m(\vec{\bm{x}}) = \sum_{j=1}^m x_j y_j, \qquad \Delta_m = 2 \sum_{j=1}^m \frac{\partial^2}{\partial x_j \partial y_j}.
    \]
When $m =1$, the region $\{\bm{x} \in \R^2 \mid Q_1(\bm{x}) < 0\}$ consists of two components. 
We choose one of them, for instance,
    \[
	\calC \coloneqq \{\bm{x} = \smat{x \\ y} \in \R^2 \mid Q_1(\bm{x}) < 0, y > 0\}.
    \]
We also define the set of cusps $\cS$ by
\begin{align*}
	\cS &\coloneqq \{\bm{x} = \smat{x \\ y} \in \Z^2 \mid (x,y) = 1, Q_1(\bm{x}) = 0, x < y\},
\end{align*}
and set $\overline{\cC} \coloneqq \cC \cup \cS$. 

The \emph{error function} is
    \[
	E(z) \coloneqq 2\int_0^z e^{-\pi u^2} \dd u.
    \]
and the \emph{sign function} is
   \[
	\sgn(x) \coloneqq \begin{cases}
		1 &\text{if } x > 0,\\
		-1 &\text{if } x \le 0.
	\end{cases}
   \]
For $\bm{c} \in \overline{\cC}$ and $k \ge 0$, we define
    \[
	\rho_{\bm{c}, k}(\bm{x}) \coloneqq \begin{cases}
		E^{(k)} \left(\dfrac{B_1(\bm{c}, \bm{x})}{\sqrt{-Q_1(\bm{c})}} \right) &\text{if } \bm{c} \in \cC,\\
		\sgn(B_1(\bm{c}, \bm{x})) &\text{if } \bm{c} \in \cS \text{ and } k=0,\\
		0 &\text{if } \bm{c} \in \cS \text{ and } k>0,
	\end{cases}
    \]
where $E^{(k)}(z)$ is the $k$-th derivative of $E(z)$. Strongly inspired by the work of Roehrig--Zwegers~\cite{RoehrigZwegers2022, RoehrigZwegers2022-pre} in the case of signature $(n-1,1)$, we are now ready to define the function $p^{\vec{\bm{c}}_0, \vec{\bm{c}}_1}[f](\vec{\bm{x}})$.

\begin{definition}\label{def:deg-m-p}
	For $\vec{\bm{c}}_0 = (\bm{c}_1^{(0)}, \dots, \bm{c}_m^{(0)}), \vec{\bm{c}}_1 = (\bm{c}_1^{(1)}, \dots, \bm{c}_m^{(1)}) \in \overline{\cC}^m$ and a polynomial $f(\vec{\bm{x}})$, we define
	\begin{align*}
		p^{\vec{\bm{c}}_0, \vec{\bm{c}}_1}[f](\vec{\bm{x}}) \coloneqq \sum_{\vec{\sigma} = (\sigma_1, \dots, \sigma_m) \in \{0,1\}^m} (-1)^{|\vec{\sigma}|} \sum_{k=0}^d \frac{(-1)^k}{(4\pi)^k} \sum_{\substack{\vec{k} = (k_1, \dots, k_m) \in \Z_{\ge 0}^m \\ |\vec{k}| = k}}F_{\vec{\sigma}, \vec{k}}(\vec{\bm{x}}) G_{\vec{\sigma}, \vec{k}}(\vec{\bm{x}}),
	\end{align*}
	where we put
	\begin{align*}
		F_{\vec{\sigma}, \vec{k}}(\vec{\bm{x}}) \coloneqq \prod_{j=1}^m \frac{\rho_{\bm{c}_j^{(\sigma_j)}, k_j} (\bm{x}_j)}{k_j!}, \qquad G_{\vec{\sigma}, \vec{k}}(\vec{\bm{x}}) \coloneqq \left(\prod_{j=1}^m \partial_{\bm{c}_j^{(\sigma_j)}}(\bm{x}_j)^{k_j}\right) f(\vec{\bm{x}}),
	\end{align*}
	and
	\begin{align*}
		\partial_{\bm{c}}(\bm{x}) &\coloneqq \begin{cases}
			\displaystyle{\frac{1}{\sqrt{-Q_1(\bm{c})}} \left(c_1 \frac{\partial}{\partial x} + c_2 \frac{\partial}{\partial y}\right)} &\text{if } \bm{c} \in \cC,\\
			1 &\text{if } \bm{c} \in \cS,
		\end{cases}
	\end{align*}
	for $\bm{c} = \smat{c_1 \\ c_2}, \bm{x} = \smat{x \\ y}$.
\end{definition}

A polynomial $f(\vec{\bm{x}})$ is called a \emph{spherical polynomial} of degree $d$ if it is homogeneous of degree $d$ and annihilated by the Laplace operator $\Delta_m$. For each $\bm{c} \in \overline{\cC}$, we define
    \[
	R(\bm{c}) \coloneqq \begin{cases}
		\R^2 &\text{if } \bm{c} \in \cC,\\
		\{\bm{a} \in \R^2 \mid B_1(\bm{c}, \bm{a}) \not\in \Z\} &\text{if } \bm{c} \in \cS.
	\end{cases}
    \] 
In this setting, we define the theta functions as follows.

\begin{definition}\label{def:indef-theta}
	Let $f: \R^{2m} \to \C$ be a spherical polynomial of degree $d$ and let $\vec{\bm{c}}_0, \vec{\bm{c}}_1 \in \overline{\cC}^m$. For $\vec{\bm{a}} = (\bm{a}_1, \dots, \bm{a}_m) \in \R^{2m}$, we assume that $\bm{a}_j \in R(\bm{c}_j^{(0)}) \cap R(\bm{c}_j^{(1)})$ for any $1 \le j \le m$. Then, for any $\vec{\bm{b}} \in \R^{2m}$, we define the theta function by
	\begin{align*}
		\theta_{\vec{\bm{a}}, \vec{\bm{b}}}^{\vec{\bm{c}}_0, \vec{\bm{c}}_1}[f] (\tau) \coloneqq v^{-d/2} \sum_{\vec{\bm{x}} \in \vec{\bm{a}} + \Z^{2m}} p^{\vec{\bm{c}}_0, \vec{\bm{c}}_1}[f](\vec{\bm{x}} \sqrt{v}) q^{Q_m(\vec{\bm{x}})} e^{2\pi i B_m(\vec{\bm{x}}, \vec{\bm{b}})}.
	\end{align*}
	In particular, if $\vec{\bm{c}}_0, \vec{\bm{c}}_1 \in \cS^m$, then we have
	\[
		\theta_{\vec{\bm{a}}, \vec{\bm{b}}}^{\vec{\bm{c}}_0, \vec{\bm{c}}_1}[f] (\tau) = \sum_{\vec{\bm{x}} \in \vec{\bm{a}} + \Z^{2m}} \prod_{j=1}^m \bigg(\sgn(B_1(\bm{c}_j^{(0)}, \bm{x}_j)) - \sgn(B_1(\bm{c}_j^{(1)}, \bm{x}_j)) \bigg) \cdot f(\vec{\bm{x}}) q^{Q_m(\vec{\bm{x}})} e^{2\pi i B_m(\vec{\bm{x}}, \vec{\bm{b}})}.
	\]
\end{definition}

The initial goal is to verify that the function $p^{\vec{\bm{c}}_0, \vec{\bm{c}}_1}[f](\vec{\bm{x}})$ satisfies the assumptions of Vign\'{e}ras' criterion, particularly when $\vec{\bm{c}}_0, \vec{\bm{c}}_1 \in \cC^m$, and to show that it defines a (generally non-holomorphic) modular form. 
First, as for the Schwartz property, the following has already been proved. 
Although the quadratic form under consideration is different, we note that the same proof works without any modification.

\begin{proposition}[{\cite[Proposition 5.2]{MatsusakaSuzuki2025-pre}}]\label{prop:degm-Vig1}
	If $\vec{\bm{c}}_0, \vec{\bm{c}}_1 \in \cC^m$, then the function $p^{\vec{\bm{c}}_0, \vec{\bm{c}}_1}[f](\vec{\bm{x}}) e^{-2\pi Q_m(\vec{\bm{x}})}$ is a Schwartz function.
\end{proposition}

Next, we verify that the second condition in Vign\'{e}ras' criterion is satisfied.

\begin{proposition}\label{prop:degm-Vig2}
	If $\vec{\bm{c}}_0, \vec{\bm{c}}_1 \in \cC^m$, then, for a spherical polynomial $f(\vec{\bm{x}})$ of degree $d$, we have $\mathcal{D}_m p^{\vec{\bm{c}}_0, \vec{\bm{c}}_1}[f](\vec{\bm{x}}) = d p^{\vec{\bm{c}}_0, \vec{\bm{c}}_1}[f](\vec{\bm{x}})$.
\end{proposition}

When $A = \diag(1, \dots, 1, -1, \dots, -1)$,  this was proved in \cite[Proposition 5.4]{MatsusakaSuzuki2025-pre}.
Here we sketch the proof focusing on necessary modifications for current $A$.

\begin{proof}
For each $\vec{k}$ with $|\vec{k}| = k$, a direct calculation yields
\begin{align*}
	\calD_m \left(F_{\vec{\sigma}, \vec{k}}(\vec{\bm{x}}) G_{\vec{\sigma}, \vec{k}}(\vec{\bm{x}}) \right) &= \calD_m F_{\vec{\sigma}, \vec{k}}(\vec{\bm{x}}) \cdot G_{\vec{\sigma}, \vec{k}}(\vec{\bm{x}}) + F_{\vec{\sigma}, \vec{k}}(\vec{\bm{x}}) \cdot \calD_m G_{\vec{\sigma}, \vec{k}}(\vec{\bm{x}})\\
		&\quad -\frac{1}{2\pi} \sum_{j=1}^m \left( \frac{\partial}{\partial x_j} F_{\vec{\sigma}, \vec{k}}(\vec{\bm{x}}) \cdot \frac{\partial}{\partial y_j} G_{\vec{\sigma}, \vec{k}}(\vec{\bm{x}}) + \frac{\partial}{\partial y_j} F_{\vec{\sigma}, \vec{k}}(\vec{\bm{x}}) \cdot \frac{\partial}{\partial x_j} G_{\vec{\sigma}, \vec{k}}(\vec{\bm{x}}) \right).
\end{align*}
The treatment of the first two terms is exactly the same as in the proof of \cite[Proposition 5.4]{MatsusakaSuzuki2025-pre}, and we obtain
\[
	\calD_m F_{\vec{\sigma}, \vec{k}}(\vec{\bm{x}}) = - k F_{\vec{\sigma}, \vec{k}}(\vec{\bm{x}}), \qquad \calD_m G_{\vec{\sigma}, \vec{k}}(\vec{\bm{x}}) = (d - k) G_{\vec{\sigma}, \vec{k}}(\vec{\bm{x}}).
\]
The sum in the third term changes its form due to the replacement of $A$, but the conclusion remains unchanged, and we still obtain
\begin{align*}
	-\frac{1}{2\pi} \sum_{j=1}^m (k_j+1) F_{\vec{\sigma}, \vec{k}+\vec{1}_j}(\vec{\bm{x}}) G_{\vec{\sigma}, \vec{k} + \vec{1}_j}(\vec{\bm{x}}),
\end{align*}
where $\vec{1}_j \in \Z^m$ be the vector whose $j$-th component is $1$ and all other components are $0$. The remaining computations proceed in the same way.
\end{proof}

Thus, Vign\'{e}ras' criterion (\cref{fact:Vig}) implies the following result, providing non-holomorphic modular forms.

\begin{theorem}\label{thm:indefinite-theta}
	Let $\vec{\bm{c}}_0, \vec{\bm{c}}_1 \in \cC^m$, and let $f(\vec{\bm{x}})$ be a spherical polynomial of degree $d$. For $\vec{\bm{a}}, \vec{\bm{b}} \in \R^{2m}$, the theta function $\theta_{\vec{\bm{a}}, \vec{\bm{b}}}^{\vec{\bm{c}}_0, \vec{\bm{c}}_1}[f] (\tau)$ defined in \cref{def:indef-theta} converges absolutely and satisfies
	\[
		\theta_{\vec{\bm{a}}, \vec{\bm{b}}}^{\vec{\bm{c}}_0, \vec{\bm{c}}_1}[f] \left(-\frac{1}{\tau}\right) = (-\tau)^{m+d} e^{2\pi iB_m(\vec{\bm{a}}, \vec{\bm{b}})} \theta_{-\vec{\bm{b}}, \vec{\bm{a}}}^{\vec{\bm{c}}_0, \vec{\bm{c}}_1}[f] (\tau).
	\]
\end{theorem}

To be complete, we also derive the transformation formula under the shift $\tau \mapsto \tau +1$.

\begin{theorem}\label{thm:shift+1}
	Under the same conditions as in \cref{thm:indefinite-theta}, we have
	\[
		\theta_{\vec{\bm{a}}, \vec{\bm{b}}}^{\vec{\bm{c}}_0, \vec{\bm{c}}_1}[f] (\tau+ 1) = e^{-2\pi i Q_m(\vec{\bm{a}})} \theta_{\vec{\bm{a}}, \vec{\bm{a}} + \vec{\bm{b}}}^{\vec{\bm{c}}_0, \vec{\bm{c}}_1}[f](\tau).
	\]
\end{theorem}

\begin{proof}
	By definition, 
	\[
		\theta_{\vec{\bm{a}}, \vec{\bm{b}}}^{\vec{\bm{c}}_0, \vec{\bm{c}}_1}[f] (\tau+ 1) = v^{-d/2} \sum_{\vec{\bm{x}} \in \vec{\bm{a}} + \Z^{2m}} p^{\vec{\bm{c}}_0, \vec{\bm{c}}_1}[f](\vec{\bm{x}} \sqrt{v}) q^{Q_m(\vec{\bm{x}})} e^{2\pi i Q_m(\vec{\bm{x}})} e^{2\pi i B_m(\vec{\bm{x}}, \vec{\bm{b}})}.
	\]
	For $\vec{\bm{x}} = \vec{\bm{a}} + \vec{\bm{n}}$ with $\vec{\bm{n}} \in \Z^{2m}$, we have
	\begin{align*}
		e^{2\pi i Q_m(\vec{\bm{x}})} &= e^{2\pi i (Q_m(\vec{\bm{a}}) + Q_m(\vec{\bm{n}}) + B_m(\vec{\bm{a}} + \vec{\bm{n}}, \vec{\bm{a}}) - B_m(\vec{\bm{a}}, \vec{\bm{a}}))}\\
			&= e^{-2\pi i Q_m(\vec{\bm{a}})} e^{2\pi iB_m(\vec{\bm{x}}, \vec{\bm{a}})},
	\end{align*}
	where we use the fact that $Q_m(\vec{\bm{n}}) \in \Z$. 
\end{proof}

% --------------------------------------------------------------------------
\subsection{The case $\vec{\bm{c}}_0, \vec{\bm{c}}_1 \in \overline{\cC}^m$}
% --------------------------------------------------------------------------

Next, we consider the case where at least one of the components of $\vec{\bm{c}}_0, \vec{\bm{c}}_1$ lies in $\cS$. 
For later use, we define the \emph{complementary error function} $\beta(x)$ by
\[
	\beta(x) \coloneqq \int_x^\infty u^{-1/2} e^{-\pi u} \dd u.
\]
Then, we have
\begin{align}\label{eq:Error-beta}
	E(x) = \sgn(x) - \sgn(x) \beta(x^2),
\end{align}
and $\beta(x^2) \le e^{-\pi x^2}$ for all $x \in \R$, (see \cite[Lemma 1.7]{Zwegers2002}).

\begin{lemma}\label{lem:theta-abs-conv-S}
	Under the same conditions as in \cref{def:indef-theta}, the theta function $\theta_{\vec{\bm{a}}, \vec{\bm{b}}}^{\vec{\bm{c}}_0, \vec{\bm{c}}_1}[f](\tau)$ converges absolutely.
\end{lemma}

\begin{proof}
	For each $\vec{k} \in \Z_{\ge 0}^m$, we define $I(\vec{k}) \coloneqq \{1 \le j \le m \mid k_j = 0\}$ and $I(\vec{k})^c = \{1, 2, \dots, m\} \setminus I(\vec{k})$. Then, one can verify that the summand of the theta function can be expressed as
	\begin{align}\label{eq:theta-summand-decomp}
		&v^{-d/2} p^{\vec{\bm{c}}_0, \vec{\bm{c}}_1}[f](\vec{\bm{x}} \sqrt{v}) q^{Q_m(\vec{\bm{x}})} e^{2\pi i B_m(\vec{\bm{x}}, \vec{\bm{b}})} \nonumber\\
		&= v^{-d/2} \sum_{k=0}^d \frac{(-1)^k}{(4\pi)^k} \sum_{\substack{\vec{k} \in \Z_{\ge 0}^m \\ |\vec{k}| = k}} \left(\sum_{\vec{\sigma} \in \{0,1\}^m} (-1)^{|\vec{\sigma}|} F_{\vec{\sigma}, \vec{k}}(\vec{\bm{x}} \sqrt{v}) G_{\vec{\sigma}, \vec{k}}(\vec{\bm{x}} \sqrt{v}) \right) q^{Q_m(\vec{\bm{x}})} e^{2\pi i B_m(\vec{\bm{x}}, \vec{\bm{b}})} \nonumber\\
		&= v^{-d/2} \sum_{k=0}^d \frac{(-1)^k}{(4\pi)^k} \sum_{\substack{\vec{k} \in \Z_{\ge 0}^m \\ |\vec{k}| = k}} \left(\sum_{\vec{\sigma} \in \{0,1\}^{I(\vec{k})^c}} (-1)^{|\vec{\sigma}|} \prod_{j=0}^m H_j^{\vec{\sigma}} (\vec{\bm{x}} \sqrt{v}) \right) q^{Q_m(\vec{\bm{x}})} e^{2\pi iB_m(\vec{\bm{x}}, \vec{\bm{b}})},
	\end{align}
	where
	\begin{align}\label{eq:def-Hx}
		H_j^{\vec{\sigma}} (\vec{\bm{x}}) \coloneqq \begin{cases}
			\displaystyle{\bigg(\prod_{l \in I(\vec{k})^c} \partial_{\bm{c}_l^{(\sigma_l)}}(\bm{x}_l)^{k_l} \bigg) f(\vec{\bm{x}})} &\text{if } j = 0,\\ \vspace{2mm}
			\displaystyle{\rho_{\bm{c}_j^{(0)},0}(\bm{x}_j) - \rho_{\bm{c}_j^{(1)},0}(\bm{x}_j)} &\text{if } j \in I(\vec{k}),\\
			\displaystyle{\frac{1}{k_j!} \rho_{\bm{c}_j^{(\sigma_j)}, k_j}(\bm{x}_j)} &\text{if } j \in I(\vec{k})^c.
		\end{cases}
	\end{align}
	The problem then reduces to showing that, for each $\vec{k}$, $\vec{\sigma}$, and $1 \le j \le m$, the series
	\begin{align}\label{eq:abs-conv-proof}
		\sum_{\bm{x}_j \in \bm{a}_j + \Z^2} P(\bm{x}_j) H_j^{\vec{\sigma}}(\vec{\bm{x}} \sqrt{v}) e^{-2\pi Q_1(\bm{x}_j)v}
	\end{align}
	converges absolutely for any polynomial $P(\bm{x}_j)$. In the discussion below, we fix $\vec{\sigma}$, and, to simplify notation, write $\bm{c}_j^{(\sigma_j)}$ as $\bm{c}_j$ when $j \in I(\vec{k})^c$.
	
	For $j \in I(\vec{k})^c$, 
	if $\bm{c}_j \in \cS$,  we have $\rho_{\bm{c}_j, k_j}(\bm{x}_j) = 0$ since $k_j > 0$. 
	Hence it suffices to consider the case $\bm{c}_j \in \cC$. By definition, we have
	\[
		H_j^{\vec{\sigma}}(\vec{\bm{x}})e^{-2\pi Q_1(\bm{x}_j)} = \frac{1}{k_j!} E^{(k_j)} \left(\frac{B_1(\bm{c}_j, \bm{x}_j)}{\sqrt{-Q_1(\bm{c}_j)}} \right)e^{-2\pi Q_1(\bm{x}_j)} = P_1(\bm{x}_j) e^{-2\pi \left(Q_1(\bm{x}_j) - \frac{B_1(\bm{c}_j, \bm{x}_j)^2}{2Q_1(\bm{c}_j)}\right)}
	\]
	for some polynomial $P_1(\bm{x}_j)$.
	By \cite[Lemma 2.5]{Zwegers2002}, the quadratic form $Q_1(\bm{x}) - B_1(\bm{c}, \bm{x})^2/2Q_1(\bm{c})$ is positive definite for $\bm{c} \in \cC$. It follows that the series in \eqref{eq:abs-conv-proof} converges absolutely.
	
	For $j \in I(\vec{k})$, if $\bm{c}_j \in \cC$ ($\bm{c}_j \in \{\bm{c}_j^{(0)}, \bm{c}_j^{(1)}\}$), we write the error function in $\rho_{\bm{c}_j, 0}(\bm{x}_j)$ as in \eqref{eq:Error-beta}, decomposing it into the sign function and a term involving $\beta(x^2)$. Since $\beta(x^2) \le e^{-\pi x^2}$, the terms containing $\beta(x^2)$ form Schwartz functions in $\bm{x}_j$, for the same reason as above. 
Thus, whether $\bm{c}_j \in \cC$ or $\in \cS$, the problem reduces to proving the absolute convergence of the sum
	\begin{align*}
		\sum_{\bm{x}_j \in \bm{a}_j + \Z^2} \bigg(\sgn(B_1(\bm{c}_j^{(0)}, \bm{x}_j)) - \sgn(B_1(\bm{c}_j^{(1)}, \bm{x}_j)) \bigg) P_2(\bm{x}_j) e^{-2\pi Q_1(\bm{x}_j) v}
	\end{align*}
	for some polynomial $P_2(\bm{x}_j)$. 
The absolute convergence of this sum for $\bm{a}_j \in R(\bm{c}_j^{(0)}) \cap R(\bm{c}_j^{(1)})$ was proven by Roehrig--Zwegers~\cite[Lemma 3.1]{RoehrigZwegers2022-pre}, which completes the proof.
\end{proof}

Next, we show that the modular transformation law of the theta functions in the general case $\vec{\bm{c}}_0, \vec{\bm{c}}_1 \in \overline{\cC}^m$ can be obtained as a limit of the case of $\vec{\bm{c}}_0, \vec{\bm{c}}_1 \in \cC^m$.
For each $\bm{c} \in \cS$, we choose any $\bm{c}' \in \cC$ and define $\bm{c}(t) \coloneqq \bm{c} + t \bm{c}'$ with $t \ge 0$. For $t > 0$, since
    \[
	Q_1(\bm{c}(t)) = B_1(\bm{c}, t \bm{c}') + Q_1(\bm{c}) + Q_1(t \bm{c}') = t B_1(\bm{c}, \bm{c}') + t^2 Q_1(\bm{c}') < 0,
    \]
we have $\bm{c}(t) \in \cC$. For any $\vec{\bm{c}} \in \overline{\cC}^m$, we define $\vec{\bm{c}}(t) \in \cC^m$ by replacing each component $\bm{c}_j \in \cS$ with $\bm{c}_j(t) \in \cC$ as above. If $\bm{c}_j \in \cC$, we simply set $\bm{c}_j(t) = \bm{c}_j$ withoug modification. When defining $\bm{c}_j(t)$, the choice of $\bm{c}'$ can either be the same for all $\bm{c}_j \in \cS$ or vary depending on each $\bm{c}_j$. 

\begin{theorem}\label{thm:indef-theta-S}
	Under the same conditions as in \cref{def:indef-theta}, we have
	\[
		\lim_{t \to 0} \theta_{\vec{\bm{a}}, \vec{\bm{b}}}^{\vec{\bm{c}}_0(t), \vec{\bm{c}}_1(t)}[f](\tau) = \theta_{\vec{\bm{a}}, \vec{\bm{b}}}^{\vec{\bm{c}}_0, \vec{\bm{c}}_1}[f](\tau).
	\]
\end{theorem}

\begin{proof}
	Let $H_j^{\vec{\sigma}}(\vec{\bm{x}}, t)$ denote the function defined by the same formula as in \eqref{eq:def-Hx}, but with the vectors $\vec{\bm{c}}_0(t)$ and $\vec{\bm{c}}_1(t)$ in place of $\vec{\bm{c}}_0$ and $\vec{\bm{c}}_1$. We then express the summand of the theta function on both sides in the form of~\eqref{eq:theta-summand-decomp}.
	A direct calculation shows that
	\begin{align*}
		&\prod_{j=0}^{m} H_j^{\vec{\sigma}} (\vec{\bm{x}}, t) - \prod_{j=0}^{m} H_j^{\vec{\sigma}}(\vec{\bm{x}}, 0)\\
		&= \sum_{\ell=0}^{m} \bigg(\prod_{0 \le j < \ell} H_j^{\vec{\sigma}}(\vec{\bm{x}}, 0) \cdot (H_\ell^{\vec{\sigma}}(\vec{\bm{x}}, t) - H_\ell^{\vec{\sigma}}(\vec{\bm{x}}, 0)) \cdot \prod_{\ell < j \le m} H_j^{\vec{\sigma}}(\vec{\bm{x}}, t) \bigg),
	\end{align*}
	so it suffices to show that, for each fixed $\vec{k}$, $\vec{\sigma}$, and $0 \le \ell \le m$, the following expression converges to $0$:
	\begin{align}\label{eq:contribution}
		\sum_{\vec{\bm{x}} \in \vec{\bm{a}} + \Z^{2m}} \left(\prod_{0 \le j < \ell} H_j^{\vec{\sigma}}(\vec{\bm{x}} \sqrt{v}, 0) \cdot (H_\ell^{\vec{\sigma}}(\vec{\bm{x}} \sqrt{v}, t) - H_\ell^{\vec{\sigma}}(\vec{\bm{x}} \sqrt{v}, 0)) \cdot \prod_{\ell < j \le m} H_j^{\vec{\sigma}}(\vec{\bm{x}} \sqrt{v}, t)\right) q^{Q_m(\vec{\bm{x}})} e^{2\pi i B_m(\vec{\bm{x}}, \vec{\bm{b}})}.
	\end{align}
	Throughout the following discussion, we fix $\vec{\sigma}$, and, for notational convenience, we simply write $\bm{c}_j^{(\sigma_j)}$ as $\bm{c}_j$.
	
	First, we consider the case $\ell = 0$. We expand the polynomial $f(\vec{\bm{x}})$ as
	\[
		f(\vec{\bm{x}}) = \sum_{\vec{e}, \vec{f} \in \Z_{\ge 0}^m} a_{\vec{e}, \vec{f}} \prod_{j=1}^m x_j^{e_j} y_j^{f_j}.
	\]
	Then, the function $H_0^{\vec{\sigma}}(\vec{\bm{x}}, t)$ is expressed as
	\begin{align}\label{eq:H0xt-expand}
		H_0^{\vec{\sigma}}(\vec{\bm{x}}, t) &= \sum_{\vec{e}, \vec{f} \in \Z_{\ge 0}^m} \bigg( a_{\vec{e}, \vec{f}} \prod_{j \in I(\vec{k})} x_j^{e_j} y_j^{f_j} \cdot \prod_{\substack{j \in I(\vec{k})^c \\ \bm{c}_j \in \cC}} \partial_{\bm{c}_j}(\bm{x}_j)^{k_j} x_j^{e_j} y_j^{f_j} \cdot \prod_{\substack{j \in I(\vec{k})^c \\ \bm{c}_j \in \cS}} \partial_{\bm{c}_j(t)}(\bm{x}_j)^{k_j} x_j^{e_j} y_j^{f_j} \bigg).
	\end{align}
	If $\bm{c}_j \in \cC$ for all $j \in I(\vec{k})^c$, then $H_0^{\vec{\sigma}}(\vec{\bm{x}} \sqrt{v},t) - H_0^{\vec{\sigma}}(\vec{\bm{x}} \sqrt{v}, 0) = 0$, which implies \eqref{eq:contribution} equals $0$. On the other hand, if there exists $j \in I(\vec{k})^c$ such that $\bm{c}_j \in \cS$, then we show that 
	\begin{align*}
		\lim_{t \to 0} \sum_{\vec{\bm{x}} \in \vec{\bm{a}} + \Z^{2m}} \left((H_0^{\vec{\sigma}}(\vec{\bm{x}} \sqrt{v}, t) - H_0^{\vec{\sigma}}(\vec{\bm{x}} \sqrt{v}, 0)) \cdot \prod_{1 \le j \le m} H_j^{\vec{\sigma}}(\vec{\bm{x}} \sqrt{v}, t)\right) q^{Q_m(\vec{\bm{x}})} e^{2\pi i B_m(\vec{\bm{x}}, \vec{\bm{b}})} = 0.
	\end{align*}
	We verify a stronger statement: each of the sums involving $H_0^{\vec{\sigma}}(\vec{\bm{x}} \sqrt{v}, t)$ and $H_0^{\vec{\sigma}}(\vec{\bm{x}} \sqrt{v}, 0)$ converges to $0$ individually. 
First,  we analyze the sum involving $H_0^{\vec{\sigma}}(\vec{\bm{x}} \sqrt{v}, t)$.
Substitute \eqref{eq:H0xt-expand} into the above sum to obtain
	\begin{align}\label{eq:H-prod-000}
		&\sum_{\vec{\bm{x}} \in \vec{\bm{a}} + \Z^{2m}} H_0^{\vec{\sigma}}(\vec{\bm{x}} \sqrt{v}, t) \cdot \prod_{1 \le j \le m} H_j^{\vec{\sigma}}(\vec{\bm{x}} \sqrt{v}, t) q^{Q_m(\vec{\bm{x}})} e^{2\pi i B_m(\vec{\bm{x}}, \vec{\bm{b}})}\\
		&= \sum_{\vec{e}, \vec{f} \in \Z_{\ge 0}^m} a_{\vec{e}, \vec{f}} \prod_{j \in I(\vec{k})} \left(\sum_{\bm{x}_j \in \bm{a}_j + \Z^2} x_j^{e_j} y_j^{f_j} H_j^{\vec{\sigma}}(\vec{\bm{x}} \sqrt{v}, t) q^{Q_1(\bm{x}_j)} e^{2\pi i B_1(\bm{x}_j, \bm{b}_j)} \right) \nonumber\\
			&\quad \times \prod_{\substack{j \in I(\vec{k})^c \\ \bm{c}_j \in \cC}} \left(\sum_{\bm{x}_j \in \bm{a}_j + \Z^2} \bigg(\partial_{\bm{c}_j}(\bm{x}_j)^{k_j} x_j^{e_j} y_j^{f_j} \bigg) H_j^{\vec{\sigma}}(\vec{\bm{x}} \sqrt{v}, 0) q^{Q_1(\bm{x}_j)} e^{2\pi i B_1(\bm{x}_j, \bm{b}_j)} \right) \nonumber\\
			&\quad \times \prod_{\substack{j \in I(\vec{k})^c \\ \bm{c}_j \in \cS}} \left(\sum_{\bm{x}_j \in \bm{a}_j + \Z^2} \bigg(\partial_{\bm{c}_j(t)}(\bm{x}_j)^{k_j} x_j^{e_j} y_j^{f_j} \bigg) H_j^{\vec{\sigma}}(\vec{\bm{x}} \sqrt{v}, t) q^{Q_1(\bm{x}_j)} e^{2\pi i B_1(\bm{x}_j, \bm{b}_j)} \right). \nonumber
	\end{align}
	For the first sum with $j \in I(\vec{k})$, we know that
	\begin{align}\label{eq:sum-H-uniform-lim}
		\lim_{t \to 0} \sum_{\bm{x}_j \in \bm{a}_j + \Z^2} x_j^{e_j} y_j^{f_j} H_j^{\vec{\sigma}}(\vec{\bm{x}} \sqrt{v}, t) q^{Q_1(\bm{x}_j)} e^{2\pi i B_1(\bm{x}_j, \bm{b}_j)} = \sum_{\bm{x}_j \in \bm{a}_j + \Z^2} x_j^{e_j} y_j^{f_j} H_j^{\vec{\sigma}}(\vec{\bm{x}} \sqrt{v}, 0) q^{Q_1(\bm{x}_j)} e^{2\pi i B_1(\bm{x}_j, \bm{b}_j)},
	\end{align}
	as established in the proof of \cite[Theorem 2.4]{RoehrigZwegers2022-pre}. 
The second sum, corresponding to $j \in I(\vec{k})^c$ with $\bm{c}_j \in \cC$, is independent of $t$ and converges absolutely. 
The third sum is the one that tends to $0$. Since $j \in I(\vec{k})^c$, we have $k_j > 0$ and
\[
	H_j^{\vec{\sigma}}(\vec{\bm{x}} \sqrt{v}, t) = \frac{1}{k_j!} E^{(k_j)} \left(\frac{B_1(\bm{c}_j(t), \bm{x}_j \sqrt{v})}{\sqrt{-Q_1(\bm{c}_j(t))}} \right).
\]
Here, since $E^{(k)}(z)$ can be written as a polynomial times $e^{-\pi z^2}$, there exist a polynomial $P_j(\bm{c}_j(t), \bm{x}_j)$ whose variables are the components of $\bm{c}_j(t)/\sqrt{-Q_1(\bm{c}_j(t))}$ and $\bm{x}_j$ such that
	\begin{align*}
		\bigg(\partial_{\bm{c}_j(t)}(\bm{x}_j)^{k_j} x_j^{e_j} y_j^{f_j} \bigg) H_j^{\vec{\sigma}}(\vec{\bm{x}} \sqrt{v}, t) e^{-2\pi Q_1(\bm{x}_j)v} = P_j(\bm{c}_j(t), \bm{x}_j) e^{-2\pi Q_1(\bm{x}_j)v + \pi \frac{B_1(\bm{c}_j(t), \bm{x}_j)^2}{Q_1(\bm{c}_j(t))}v}.
	\end{align*}
	Recalling that Roehrig--Zwegers~\cite[Proof of Theorem 2.4]{RoehrigZwegers2022-pre} showed that for any such polynomial $P_j$, there exists a majorant that is independent of $t$ and absolutely convergent, we conclude that the limit over the sum on $\bm{x}_j \in \bm{a}_j + \Z^2$ can be interchanged with the summation order. 
This yields the desired identity
	\[
		\lim_{t \to 0} \sum_{\bm{x}_j \in \bm{a}_j + \Z^2} \bigg(\partial_{\bm{c}_j(t)}(\bm{x}_j)^{k_j} x_j^{e_j} y_j^{f_j} \bigg) H_j^{\vec{\sigma}}(\vec{\bm{x}} \sqrt{v}, t) q^{Q_1(\bm{x}_j)} e^{2\pi i B_1(\bm{x}_j, \bm{b}_j)} = 0.
	\]
	Therefore, the limit of \eqref{eq:H-prod-000} as $t \to 0$ equals $0$. 
Similarly, we can verify that
	\[
		\lim_{t \to 0} \sum_{\vec{\bm{x}} \in \vec{\bm{a}} + \Z^{2m}} H_0^{\vec{\sigma}}(\vec{\bm{x}} \sqrt{v}, 0) \cdot \prod_{1 \le j \le m} H_j^{\vec{\sigma}}(\vec{\bm{x}} \sqrt{v}, t) q^{Q_m(\vec{\bm{x}})} e^{2\pi i B_m(\vec{\bm{x}}, \vec{\bm{b}})} = 0.
	\]
	Hence, \eqref{eq:contribution} converges to $0$ when $\ell = 0$.
	
	Next, we assume that $1 \le \ell \le m$. Note that $H_0^{\vec{\sigma}}(\vec{\bm{x}}\sqrt{v}, 0)$ is a polynomial independent of $t$. We now rewrite it as
	\[
		H_0^{\vec{\sigma}}(\vec{\bm{x}}, 0) = \sum_{\vec{e}, \vec{f} \in \Z_{\ge 0}^m} b_{\vec{e}, \vec{f}} \prod_{j=1}^m x_j^{e_j} y_j^{f_j}.
	\]
	Then, as in \eqref{eq:H-prod-000}, we are led to evaluate the following sum.
	\begin{align*}
		\eqref{eq:contribution} &= \sum_{\vec{e}, \vec{f} \in \Z_{\ge 0}^m} b_{\vec{e}, \vec{f}} \prod_{1 \le j < \ell} \left(\sum_{\bm{x}_j \in \bm{a}_j + \Z^2} x_j^{e_j} y_j^{f_j} H_j^{\vec{\sigma}}(\vec{\bm{x}}\sqrt{v}, 0) q^{Q_1(\bm{x}_j)} e^{2\pi i B_1(\bm{x}_j, \bm{b}_j)} \right)\\
		&\quad \times \left(\sum_{\bm{x}_\ell \in \bm{a}_\ell + \Z^2} x_\ell^{e_\ell} y_\ell^{f_\ell} (H_\ell^{\vec{\sigma}}(\vec{\bm{x}} \sqrt{v}, t) - H_\ell^{\vec{\sigma}}(\vec{\bm{x}} \sqrt{v}, 0)) q^{Q_1(\bm{x}_\ell)} e^{2\pi i B_1(\bm{x}_\ell, \bm{b}_\ell)}\right)\\
		&\quad \times \prod_{\ell < j \le m} \left(\sum_{\bm{x}_j \in \bm{a}_j + \Z^2} x_j^{e_j} y_j^{f_j} H_j^{\vec{\sigma}}(\vec{\bm{x}}\sqrt{v}, t) q^{Q_1(\bm{x}_j)} e^{2\pi i B_1(\bm{x}_j, \bm{b}_j)} \right).
	\end{align*}
	The first sum converges absolutely and is independent of $t$. 
Moreover, as in \eqref{eq:sum-H-uniform-lim}, the third sum converges as $t \to 0$. 
It therefore remains to show that the second sum converges to $0$. 
This also follows from the proof of \cite[Theorem 2.4]{RoehrigZwegers2022-pre}.
\end{proof}

% --------------------------------------------------------------------------
\section{Proofs of \cref{thm:Zagier-Q} and \cref{thm:Dmm-Eisen} via indefinite theta functions}\label{sec:proof_via_theta}
% --------------------------------------------------------------------------

We first establish the modularity of the left-hand side $\triangle(q)$. 
Then, applying the theory of indefinite theta functions developed in the previous chapter, we show the modularity of the right-hand side. 
By comparing both sides, we verify the three identities.

% --------------------------------------------------------------------------
\subsection{Theta function $\theta_\triangle(\tau)$}
% --------------------------------------------------------------------------

As in our previous article~\cite{MatsusakaSuzuki2025-pre}, we define
    \[
	\theta_\triangle(\tau) \coloneqq q^{1/16} \triangle(q^{1/2}) = \frac{1}{2} \sum_{n \in \frac{1}{2} + \Z} q^{n^2/4}.
    \]
Then, the following holds.

\begin{lemma}[{\cite[Lemma 4.1]{MatsusakaSuzuki2025-pre}}]\label{lem:theta-triangle}
	The function $\theta_\triangle(\tau)$ is a modular form of weight $1/2$ on $\Gamma(2)$ with no zeros or poles on $\bbH$. Specifically, for generators $\smat{1 & 2 \\ 0 & 1}$ and $\smat{1 & 0 \\ 2 & 1}$ of $\Gamma(2)/\{\pm I\}$, we have
	\begin{align*}
		\theta_\triangle(\tau+2) &= e^{\frac{\pi i}{4}} \theta_\triangle(\tau),\\
		\theta_\triangle \left(\frac{\tau}{2\tau+1}\right) &= (2\tau+1)^{1/2} \theta_\triangle(\tau).
	\end{align*}
	Its behavior at the cusps $i\infty, 0, 1$ of $\Gamma(2)$ are given by
	\begin{align*}
		\theta_\triangle(\tau) &= q^{1/16} + O(q^{9/16}),\\
		(-i\tau)^{-1/2} \theta_\triangle\left(-\frac{1}{\tau}\right) &= \frac{1}{\sqrt{2}} + O(q),\\
		(-i\tau)^{-1/2} \theta_\triangle \left(\frac{\tau-1}{\tau}\right) &= e^{\frac{\pi i}{8}} q^{1/16} + O(q^{9/16}),
	\end{align*}
	respectively.
\end{lemma}

% --------------------------------------------------------------------------
\subsection{Three spherical polynomials}
% --------------------------------------------------------------------------

As a first step, we verify that the following three homogeneous polynomials are spherical, that is, they are annihilated by $\Delta_m$:
\begin{align*}
	f_1(\vec{\bm{x}}) &= \prod_{j=1}^m x_j \cdot \prod_{1 \le i < j \le m}(x_i^2 - x_j^2)^2,\\
	f_2(\vec{\bm{x}}) &= \prod_{j=1}^m x_j^3 \cdot \prod_{1 \le i < j \le m}(x_i^2 - x_j^2)^2,\\
	f_3(\vec{\bm{x}}) &= \prod_{j=1}^m (x_j + y_j) \cdot \prod_{1 \le i < j \le m} \bigg((x_i-y_i)^2 - (x_j-y_j)^2 \bigg)\bigg((x_i+y_i)^2 - (x_j+y_j)^2\bigg).
\end{align*}
Since $f_1(\vec{\bm{x}})$ and $f_2(\vec{\bm{x}})$ do not involve the variables $y_j$, it is clear that they are spherical polynomials. We now prove that $f_3(\vec{\bm{x}})$ is also spherical.

\begin{lemma}
	For every integer $m \ge 1$, we have $\Delta_m f_3(\vec{\bm{x}}) = 0$.
\end{lemma}

\begin{proof}
	We define the polynomials $A(\vec{\bm{x}})$ and $B(\vec{\bm{x}})$ by
	\begin{align*}
		A(\vec{\bm{x}}) &\coloneqq \prod_{j=1}^m (x_j + y_j),\\
		B(\vec{\bm{x}}) &\coloneqq \prod_{1 \le i < j \le m} \bigg((x_i-y_i)^2 - (x_j-y_j)^2 \bigg)\bigg((x_i+y_i)^2 - (x_j+y_j)^2\bigg).
	\end{align*}
	For each $1 \le j \le m$, we compute
	\begin{align}\label{eq:f3-spherical}
	\begin{split}
		\frac{\partial^2}{\partial x_j \partial y_j} f_3(\vec{\bm{x}}) &= \frac{\partial^2}{\partial x_j \partial y_j} A(\vec{\bm{x}}) \cdot B(\vec{\bm{x}}) + A(\vec{\bm{x}}) \cdot \frac{\partial^2}{\partial x_j \partial y_j} B(\vec{\bm{x}})\\
			&\quad + \frac{\partial}{\partial x_j} A(\vec{\bm{x}}) \cdot \frac{\partial}{\partial y_j} B(\vec{\bm{x}}) + \frac{\partial}{\partial y_j} A(\vec{\bm{x}}) \cdot \frac{\partial}{\partial x_j} B(\vec{\bm{x}}).
	\end{split}
	\end{align}
	The first term is clearly $0$. To evaluate the second term, we set
	\begin{align*}
		p_{i,j}(\vec{\bm{x}}) &\coloneqq (x_i - y_i)^2 - (x_j - y_j)^2,\\
		q_{i,j}(\vec{\bm{x}}) &\coloneqq (x_i + y_i)^2 - (x_j + y_j)^2.
	\end{align*}
	Then, we have
	\begin{align*}
		\frac{\partial^2}{\partial x_j \partial y_j} B(\vec{\bm{x}}) &= \sum_{1 \le k < l \le m} \left(\frac{B(\vec{\bm{x}})}{p_{k,l}(\vec{\bm{x}})} 	\frac{\partial^2}{\partial x_j \partial y_j} p_{k,l}(\vec{\bm{x}}) + \frac{B(\vec{\bm{x}})}{q_{k,l}(\vec{\bm{x}})} 	\frac{\partial^2}{\partial x_j \partial y_j} q_{k,l}(\vec{\bm{x}}) \right)\\
			&+ \sum_{\substack{1 \le k < l \le m \\ 1 \le s < t \le m}} \left(\frac{B(\vec{\bm{x}})}{p_{k,l}(\vec{\bm{x}}) q_{s,t}(\vec{\bm{x}})} \frac{\partial}{\partial x_j} p_{k,l}(\vec{\bm{x}}) \frac{\partial}{\partial y_j} q_{s,t}(\vec{\bm{x}}) + \frac{B(\vec{\bm{x}})}{p_{k,l}(\vec{\bm{x}}) q_{s,t}(\vec{\bm{x}})} \frac{\partial}{\partial y_j} p_{k,l}(\vec{\bm{x}}) \frac{\partial}{\partial x_j} q_{s,t}(\vec{\bm{x}})\right).
	\end{align*}
	Observe that
	\begin{align*}
		\frac{\partial^2}{\partial x_j \partial y_j} p_{i,j}(\vec{\bm{x}}) = 2 = \frac{\partial^2}{\partial x_j \partial y_j} q_{j,i}(\vec{\bm{x}}),\\
		\frac{\partial^2}{\partial x_j \partial y_j} q_{i,j}(\vec{\bm{x}}) = -2 = \frac{\partial^2}{\partial x_j \partial y_j} p_{j,i}(\vec{\bm{x}}).
	\end{align*}
	Hence,
	\begin{align*}
		&\sum_{1 \le k < l \le m} \left(\frac{B(\vec{\bm{x}})}{p_{k,l}(\vec{\bm{x}})} 	\frac{\partial^2}{\partial x_j \partial y_j} p_{k,l}(\vec{\bm{x}}) + \frac{B(\vec{\bm{x}})}{q_{k,l}(\vec{\bm{x}})} 	\frac{\partial^2}{\partial x_j \partial y_j} q_{k,l}(\vec{\bm{x}}) \right)\\
		&\quad = 2 B(\vec{\bm{x}}) \sum_{1 \le i < j} \left(\frac{1}{p_{i,j}(\vec{\bm{x}})} - \frac{1}{q_{i,j}(\vec{\bm{x}})} \right) + 2 B(\vec{\bm{x}}) \sum_{j < i \le m} \left(\frac{1}{q_{j,i}(\vec{\bm{x}})} - \frac{1}{p_{j,i}(\vec{\bm{x}})} \right).
	\end{align*}
	Summing over $j = 1, \dots, m$, the total contribution of this part is zero:
	\[
		\sum_{j=1}^m \sum_{1 \le k < l \le m} \left(\frac{B(\vec{\bm{x}})}{p_{k,l}(\vec{\bm{x}})} 	\frac{\partial^2}{\partial x_j \partial y_j} p_{k,l}(\vec{\bm{x}}) + \frac{B(\vec{\bm{x}})}{q_{k,l}(\vec{\bm{x}})} 	\frac{\partial^2}{\partial x_j \partial y_j} q_{k,l}(\vec{\bm{x}}) \right) = 0.
	\]
	Next, we compute the remaining terms. 
The contribution from
	\begin{align*}
		\sum_{\substack{1 \le k < l \le m \\ 1 \le s < t \le m}} \frac{B(\vec{\bm{x}})}{p_{k,l}(\vec{\bm{x}}) q_{s,t}(\vec{\bm{x}})} \frac{\partial}{\partial x_j} p_{k,l}(\vec{\bm{x}}) \frac{\partial}{\partial y_j} q_{s,t}(\vec{\bm{x}})\\
	\end{align*}
	is equal to
	\begin{align*}
		&4 (x_j^2 - y_j^2) \sum_{j < l, t \le m} \frac{B(\vec{\bm{x}})}{p_{j,l}(\vec{\bm{x}}) q_{j,t}(\vec{\bm{x}})} + 4(x_j^2 - y_j^2) \sum_{1 \le k, s < j} \frac{B(\vec{\bm{x}})}{p_{k,j}(\vec{\bm{x}}) q_{s,j}(\vec{\bm{x}})}\\
		&\quad - 4(x_j^2 - y_j^2) \sum_{1 \le k < j < t \le m} \frac{B(\vec{\bm{x}})}{p_{k,j}(\vec{\bm{x}}) q_{j,t}(\vec{\bm{x}})} - 4(x_j^2 - y_j^2) \sum_{1 \le s < j < l \le m} \frac{B(\vec{\bm{x}})}{p_{j,l}(\vec{\bm{x}}) q_{s,j}(\vec{\bm{x}})}.
	\end{align*}
	The corresponding expression from the other term gives the same quantity with opposite sign, so they cancel.
	
	Finally, we evaluate the third and fourth terms of \eqref{eq:f3-spherical}. Since
	\[
		\frac{\partial}{\partial x_j} A(\vec{\bm{x}}) = \frac{A(\vec{\bm{x}})}{x_j + y_j} = \frac{\partial}{\partial y_j} A(\vec{\bm{x}})
	\]
	and
	\begin{align*}
		\frac{\partial}{\partial y_j} B(\vec{\bm{x}}) &= 2B(\vec{\bm{x}}) \sum_{1 \le i < j} \left(\frac{x_j - y_j}{p_{i,j}(\vec{\bm{x}})} - \frac{x_j + y_j}{q_{i,j}(\vec{\bm{x}})}\right) + 2B(\vec{\bm{x}}) \sum_{j < i \le m} \left(-\frac{x_j - y_j}{p_{j,i}(\vec{\bm{x}})} + \frac{x_j + y_j}{q_{j,i}(\vec{\bm{x}})}\right),\\
		\frac{\partial}{\partial x_j} B(\vec{\bm{x}}) &= 2B(\vec{\bm{x}}) \sum_{1 \le i < j} \left(-\frac{x_j - y_j}{p_{i,j}(\vec{\bm{x}})} - \frac{x_j + y_j}{q_{i,j}(\vec{\bm{x}})}\right) + 2B(\vec{\bm{x}}) \sum_{j < i \le m} \left(\frac{x_j - y_j}{p_{j,i}(\vec{\bm{x}})} + \frac{x_j + y_j}{q_{j,i}(\vec{\bm{x}})}\right),
	\end{align*}
	we have
	\begin{align*}
		&\sum_{j=1}^m \left( \frac{\partial}{\partial x_j} A(\vec{\bm{x}}) \cdot \frac{\partial}{\partial y_j} B(\vec{\bm{x}}) + \frac{\partial}{\partial y_j} A(\vec{\bm{x}}) \cdot \frac{\partial}{\partial x_j} B(\vec{\bm{x}})\right)\\
			&= 4A(\vec{\bm{x}}) B(\vec{\bm{x}}) \sum_{j=1}^m \left(-\sum_{1 \le i < j} \frac{1}{q_{i,j}(\vec{\bm{x}})} + \sum_{j < i \le m} \frac{1}{q_{j,i}(\vec{\bm{x}})} \right) = 0.
	\end{align*}
	Therefore, we conclude that $\Delta_m f_3(\vec{\bm{x}}) = 0$.
\end{proof}

By rewriting the left-hand side of the theorems in terms of $\theta_\triangle(\tau)$ and expressing the right-hand side as a sum over the entire lattice, the claims to be proven can be reformulated as follows.

\begin{theorem}[Equivalent to \cref{thm:Zagier-Q} and \cref{thm:Dmm-Eisen}]\label{thm:equivalent-3}
	For every integer $m \ge 1$, we have
	\begin{align*}
		\theta_\triangle(\tau)^{4m^2} &= \frac{(2m)!}{\prod_{j=1}^m (2j)!^2} \sum_{x_j, y_j \in 1/2+\Z} \prod_{j=1}^m \bigg(\sgn(x_j) + \sgn(y_j) \bigg) f_1(\vec{\bm{x}}) q^{Q_m(\vec{\bm{x}})},\\
		\theta_\triangle(\tau)^{4m(m+1)} &= \frac{1}{\prod_{j=1}^m (2j)!^2} \sum_{\substack{x_j \in \Z,\\ y_j \in 1/2+\Z}} \prod_{j=1}^m \bigg(\sgn(x_j) + \sgn(y_j) \bigg) f_2(\vec{\bm{x}}) q^{Q_m(\vec{\bm{x}})},\\
		\theta_\triangle(\tau)^{4m^2} &= \frac{(2m)!}{2^m \prod_{j=1}^m (2j)!^2} \sum_{x_j, y_j \in 1/2+\Z} \prod_{j=1}^m \bigg(\sgn(x_j) + \sgn(y_j) \bigg) f_3(\vec{\bm{x}}) q^{Q_m(\vec{\bm{x}})}.
	\end{align*}
\end{theorem}

\begin{proof}
	Since the strategy is the same, we verify only the first claim. In the first identity, after changing variables via $q \mapsto q^2$, the left-hand side becomes $q^{m^2/2} \triangle(q)^{4m^2}$. As for the right-hand side, because of the sum involving the sign function, only the terms where $x_j$ and $y_j$ have the same sign for each $j$ contribute. Moreover, noting that the polynomial $f_1(\vec{\bm{x}})$ changes sign under the replacement $\bm{x}_j \leftrightarrow -\bm{x}_j$, we obtain
	\[
		\sum_{x_j, y_j \in 1/2+\Z} \prod_{j=1}^m \bigg(\sgn(x_j) + \sgn(y_j) \bigg) f_1(\vec{\bm{x}}) q^{2Q_m(\vec{\bm{x}})} = 4^m \sum_{x_j, y_j \in 1/2 + \Z_{\ge 0}} f_1(\vec{\bm{x}}) q^{2Q_m(\vec{\bm{x}})}.
	\]
	This completes the derivation of the first identity in \cref{thm:Zagier-Q}.
\end{proof}

For $\vec{\bm{a}} = (\smat{1/2 \\ 1/2}, \dots, \smat{1/2 \\ 1/2})$, $\vec{\bm{a}}' = (\smat{0 \\ 1/2}, \dots, \smat{0 \\ 1/2})$, $\vec{\bm{c}}_0 = (\smat{0 \\ 1}, \dots, \smat{0 \\ 1}) \in \cS^m$, and $\vec{\bm{c}}_1 = (\smat{-1 \\ 0}, \dots, \smat{-1 \\ 0}) \in \cS^m$,  the series on the right-hand side can be expressed as
\begin{align}\label{eq:def:F123}
\begin{split}
	F_1(\tau) &\coloneqq \frac{(2m)!}{\prod_{j=1}^m (2j)!^2} \theta_{\vec{\bm{a}}, \vec{\bm{0}}}^{\vec{\bm{c}}_0, \vec{\bm{c}}_1}[f_1] (\tau),\\
	F_2(\tau) &\coloneqq \frac{1}{\prod_{j=1}^m (2j)!^2} \theta_{\vec{\bm{a}}', \vec{\bm{0}}}^{\vec{\bm{c}}_0, \vec{\bm{c}}_1}[f_2] (\tau),\\
	F_3(\tau) &\coloneqq \frac{(2m)!}{2^m \prod_{j=1}^m (2j)!^2} \theta_{\vec{\bm{a}}, \vec{\bm{0}}}^{\vec{\bm{c}}_0, \vec{\bm{c}}_1}[f_3] (\tau).
\end{split}
\end{align}
Note that $\smat{0 \\ 1/2} \not\in R(\smat{0 \\ 1})$, so the condition in \cref{def:indef-theta} is not satisfied for $F_2(\tau)$. Here, we write $F_2(\tau)$ to denote the right-hand side of the second equation in this theorem. The validity of this expression will be established in \cref{lem:F2-modified}. It remains to verify that each $F_j(\tau)$ satisfies the same properties as those on the left-hand side of \cref{lem:theta-triangle}.

% --------------------------------------------------------------------------
\subsection{Proof of the first identity (\cref{thm:Zagier-Q})}\label{sec:first_identity}
% --------------------------------------------------------------------------

\begin{lemma}\label{lem:Dmm-Gamma2}
	We have the following:
	\begin{align*}
		F_1(\tau+2) &= (-1)^m F_1(\tau),\\
		F_1 \left(\frac{\tau}{2\tau+1}\right) &= (2\tau + 1)^{2m^2} F_1(\tau).
	\end{align*}
\end{lemma}

\begin{proof}
	We first replace the given vectors $\vec{\bm{c}}_0, \vec{\bm{c}}_1 \in \cS^m$ with $\vec{\bm{c}}_0(t), \vec{\bm{c}}_1(t) \in \cC^m$, and apply \cref{thm:indef-theta-S}. By \cref{thm:shift+1}, we have
	\begin{align*}
		\theta_{\vec{\bm{a}}, \vec{\bm{0}}}^{\vec{\bm{c}}_0(t), \vec{\bm{c}}_1(t)}[f_1](\tau+2) = e^{-4\pi i Q_m(\vec{\bm{a}})} \theta_{\vec{\bm{a}}, 2\vec{\bm{a}}}^{\vec{\bm{c}}_0(t), \vec{\bm{c}}_1(t)}[f_1](\tau) = (-1)^m \theta_{\vec{\bm{a}}, \vec{\bm{0}}}^{\vec{\bm{c}}_0(t), \vec{\bm{c}}_1(t)}[f_1](\tau).
	\end{align*}
	Similarly, by \cref{thm:indefinite-theta} and \cref{thm:shift+1}, we obtain
	\begin{align*}
		\theta_{\vec{\bm{a}}, \vec{\bm{0}}}^{\vec{\bm{c}}_0(t), \vec{\bm{c}}_1(t)}[f_1] \left(\frac{\tau}{2\tau+1}\right) &= \left(2 + \frac{1}{\tau}\right)^{m+m(2m-1)} \theta_{\vec{\bm{0}}, \vec{\bm{a}}}^{\vec{\bm{c}}_0(t), \vec{\bm{c}}_1(t)}[f_1] \left(-2 -\frac{1}{\tau}\right)\\
		&= \left(2 + \frac{1}{\tau}\right)^{2m^2} \theta_{\vec{\bm{0}}, \vec{\bm{a}}}^{\vec{\bm{c}}_0(t), \vec{\bm{c}}_1(t)}[f_1] \left(-\frac{1}{\tau}\right)\\
		&= (2\tau + 1)^{2m^2} \theta_{-\vec{\bm{a}}, \vec{\bm{0}}}^{\vec{\bm{c}}_0(t), \vec{\bm{c}}_1(t)}[f_1] (\tau).
	\end{align*}
	The claims then follow by taking the limit as $t \to 0$ and multiplying the constant.
\end{proof}

\begin{lemma}\label{lem:Dmm-cusps}
	We have the following:
	\begin{align*}
		F_1(\tau) &= q^{m^2/4} + o(q^{m^2/4}),\\
		(-i \tau)^{-2m^2} F_1 \left(-\frac{1}{\tau}\right) &= O(1),\\
		(-i \tau)^{-2m^2} F_1 \left(\frac{\tau-1}{\tau}\right) &= O(q^{m^2/4}).
	\end{align*}
\end{lemma}

\begin{proof}
	The lowest degree in the $q$-series expansion of $\theta_{\vec{\bm{a}}, \vec{\bm{0}}}^{\vec{\bm{c}}_0, \vec{\bm{c}}_1}[f_1](\tau)$ is attained when $\vec{\bm{x}}_0 = (\smat{1/2 \\ 1/2}, \smat{3/2 \\ 1/2}, \dots, \smat{m-1/2 \\ 1/2})$, which gives
	\[
		Q_m(\vec{\bm{x}}_0) = \sum_{j=1}^m \frac{1}{2} \left(j - \frac{1}{2}\right) = \frac{m^2}{4}.
	\]
	The corresponding coefficient is computed by
	\[
		f_1(\vec{\bm{x}}_0) = \frac{(2m)!}{4^m m!} \prod_{j=1}^{m-1} (2j)!^2,
	\]
	(which can be shown by induction), and taking into account the $m!$ permutations of the entries of $\vec{\bm{x}}_0$ and $2^m$ sign changes, together with $\sgn(x_j) + \sgn(y_j) = \pm 2$, we obtain
	\[
		m! \cdot 2^m \cdot 2^m f_1(\vec{\bm{x}}_0) = \frac{\prod_{j=1}^m (2j)!^2}{(2m)!}.
	\]
	Hence, we have $F_1(\tau) = q^{m^2/4} + o(q^{m^2/4})$.
	
	To show the second formula, we consider $\vec{\bm{c}}_0(t), \vec{\bm{c}}_1(t)$, and replace $\vec{\bm{0}}$ by $\vec{\bm{\varepsilon}} = (\smat{\varepsilon \\ \varepsilon}, \dots, \smat{\varepsilon \\ \varepsilon})$ with $0 < \varepsilon < 1$. Applying \cref{thm:indefinite-theta}, we have
	\begin{align*}
		(-i\tau)^{-2m^2} \theta_{\vec{\bm{a}}, \vec{\bm{\varepsilon}}}^{\vec{\bm{c}}_0(t), \vec{\bm{c}}_1(t)}[f_1] \left(-\frac{1}{\tau}\right) = (-i)^{2m^2} e^{2\pi i B_m(\vec{\bm{a}}, \vec{\bm{\varepsilon}})} \theta_{-\vec{\bm{\varepsilon}}, \vec{\bm{a}}}^{\vec{\bm{c}}_0(t), \vec{\bm{c}}_1(t)}[f_1](\tau).
	\end{align*}
	As $t \to 0$ and $\varepsilon \to 0$, the left-hand side approaches $(-i\tau)^{-2m^2} F_1(-1/\tau)$, up to a constant multiple. On the other hand, since $-\smat{\varepsilon \\ \varepsilon} \in R(\smat{0 \\ -1}) \cap R(\smat{1 \\ 0})$, we may apply \cref{thm:indef-theta-S} to deduce that
	\begin{align*}
		\lim_{t \to 0} \theta_{-\vec{\bm{\varepsilon}}, \vec{\bm{a}}}^{\vec{\bm{c}}_0(t), \vec{\bm{c}}_1(t)}[f_1](\tau) = \sum_{\vec{\bm{x}} \in \Z^{2m}} \prod_{j=1}^m \bigg(\sgn(x_j - \varepsilon) + \sgn(y_j - \varepsilon) \bigg) f_1(\vec{\bm{x}} -\vec{\bm{\varepsilon}}) q^{Q_m(\vec{\bm{x}}-\vec{\bm{\varepsilon}})} e^{2\pi i B_m(\vec{\bm{x}}-\vec{\bm{\varepsilon}}, \vec{\bm{a}})}.
	\end{align*}
	Next, we expand $f_1(\vec{\bm{x}} - \vec{\bm{\varepsilon}})$ as
	\begin{align}\label{eq:f1-expand}
		f_1(\vec{\bm{x}} - \vec{\bm{\varepsilon}}) = \sum_{\vec{e} \in \Z_{\ge 0}^m} p_{\vec{e}}(\varepsilon) \prod_{j=1}^m x_j^{e_j}
	\end{align}
	for some polynomials $p_{\vec{e}}(\varepsilon) \in \Z[\varepsilon]$. Substituting this into the expression above, the right-hand side becomes
	\begin{align*}
		\sum_{\vec{e} \in \Z_{\ge 0}^m} p_{\vec{e}}(\varepsilon) \prod_{j=1}^m \left(\sum_{x_j, y_j \in \Z} \bigg(\sgn(x_j - \varepsilon) + \sgn(y_j - \varepsilon)\bigg) x_j^{e_j} q^{(x_j-\varepsilon)(y_j-\varepsilon)} e^{\pi i (x_j + y_j -2\varepsilon)}\right).
	\end{align*}
	To compute its limit as $\varepsilon \to 0$ for each fixed $\vec{e}$ and $j$, we separate the sum according to whether $x_j y_j = 0$ or not. 
	
	From the first case, we have
	\begin{align*}
		&\sum_{\substack{x_j, y_j \in \Z \\ x_j y_j = 0}} \bigg(\sgn(x_j - \varepsilon) + \sgn(y_j - \varepsilon)\bigg) x_j^{e_j} q^{(x_j-\varepsilon)(y_j-\varepsilon)} e^{\pi i (x_j + y_j -2\varepsilon)}\\
		&= -2\delta_{e_j, 0} \sum_{y_j \le 0} q^{(-\varepsilon)(y_j-\varepsilon)} e^{\pi i (y_j -2\varepsilon)} -2 \sum_{x_j \le 0} x_j^{e_j} q^{(x_j-\varepsilon)(-\varepsilon)} e^{\pi i (x_j -2\varepsilon)},
	\end{align*}
	where $\delta_{e_j, 0}$ is the Kronecker delta. Let $E_k(x) \in \Z[x]$ be the polynomial defined by
	\begin{align}\label{eq:def-Eulerian}
		\sum_{n=0}^\infty (-1)^n n^k x^n = \frac{E_k(x)}{(1+x)^{k+1}}.
	\end{align}
	Then this expression evaluates to
	\begin{align*}
		&= -2 \delta_{e_j, 0} \frac{q^{\varepsilon^2} e^{-2\pi i \varepsilon}}{1+q^{\varepsilon}} - 2(-1)^{e_j} \frac{q^{\varepsilon^2} e^{-2\pi i \varepsilon} E_{e_j}(q^\varepsilon)}{(1+q^\varepsilon)^{e_j+1}},
	\end{align*}
	which approaches to the constant
%	\[
%		- \delta_{e_j, 0} - (-1)^{e_j} \frac{E_{e_j}(1)}{2^{e_j}}
%	\]
	as $\varepsilon \to 0$.
	
	From the second case, we have
	\begin{align*}
		&\lim_{\varepsilon \to 0} \sum_{\substack{x_j, y_j \in \Z \\ x_j y_j \neq 0}} \bigg(\sgn(x_j - \varepsilon) + \sgn(y_j - \varepsilon)\bigg) x_j^{e_j} q^{(x_j-\varepsilon)(y_j-\varepsilon)} e^{\pi i (x_j + y_j -2\varepsilon)}\\
		&= \sum_{\substack{x_j, y_j \in \Z \\ x_j y_j \neq 0}} \bigg(\sgn(x_j) + \sgn(y_j)\bigg) x_j^{e_j} q^{x_j y_j} e^{\pi i (x_j + y_j)} = O(q).
	\end{align*}
	Here, the interchange of the order of limit and summation is justified. This completes the proof of the second formula.
	
	Finally, for the third result, by \cref{thm:indefinite-theta} and \cref{thm:shift+1}, we have
	\begin{align*}
		(-i\tau)^{-2m^2} \theta_{\vec{\bm{a}}, \vec{\bm{0}}}^{\vec{\bm{c}}_0(t), \vec{\bm{c}}_1(t)}[f_1] \left(\frac{\tau-1}{\tau}\right) &= (-i\tau)^{-2m^2} e^{-2\pi i Q_m(\vec{\bm{a}})} \theta_{\vec{\bm{a}}, \vec{\bm{a}}}^{\vec{\bm{c}}_0(t), \vec{\bm{c}}_1(t)}[f_1] \left(-\frac{1}{\tau}\right)\\
		&= (-i)^{2m^2} e^{2\pi i Q_m(\vec{\bm{a}})} \theta_{-\vec{\bm{a}}, \vec{\bm{a}}}^{\vec{\bm{c}}_0(t), \vec{\bm{c}}_1(t)}[f_1] (\tau).
	\end{align*}
	Letting $t \to 0$ and comparing with the first result, we arrive at $O(q^{m^2/4})$.
\end{proof}

By comparing \cref{lem:theta-triangle} with \cref{lem:Dmm-Gamma2}, we find that the expression
\[
	\frac{F_1(\tau)}{\theta_\triangle(\tau)^{4m^2}}
\]
is invariant under the action of $\Gamma(2)$. Since $\theta_\triangle(\tau)$ is holomorphic and nonvanishing on $\bbH$, it follows that this quotient defines a holomorphic function on the quotient space $\Gamma(2) \backslash \bbH$. Furthermore, combining the results of \cref{lem:Dmm-cusps} and \cref{lem:theta-triangle}, we find that this function is bounded at all three cusp $i\infty, 0$, and $1$, and attains the value $1$ at $i\infty$. As $\Gamma(2) \backslash \bbH$ has genus zero, Liouville's theorem implies that this function must be identically equal to $1$. In other words, we have $\theta_\triangle(\tau)^{4m^2} = F_1(\tau)$.

% --------------------------------------------------------------------------
\subsection{Proof of the second identity (\cref{thm:Zagier-Q})}\label{sec:second_identity}
% --------------------------------------------------------------------------

We take $\vec{\bm{a}}', \vec{\bm{c}}_0$, and $\vec{\bm{c}}_1$ to be those chosen in \eqref{eq:def:F123}. The argument follows the same outline as in \cref{sec:first_identity}, but as noted earlier, since $\smat{0 \\ 1/2} \not\in R(\smat{0 \\ 1})$, a slight modification is required. 

\begin{lemma}\label{lem:F2-modified}
	Let $\vec{\bm{b}}(\varepsilon) \in \R^{2m}$ be a function that is continuous at $\varepsilon = 0$. For $0 < \varepsilon < 1$, we set $\vec{\bm{\varepsilon}} = (\smat{\varepsilon \\ 0}, \dots, \smat{\varepsilon \\ 0})$. Then we have
	\begin{align*}
		\lim_{\varepsilon \to 0} \theta_{\vec{\bm{a}}'+ \vec{\bm{\varepsilon}}, \vec{\bm{b}}(\varepsilon)}^{\vec{\bm{c}}_0, \vec{\bm{c}}_1}[f_2](\tau) = \sum_{\vec{\bm{x}} \in \vec{\bm{a}}' + \Z^{2m}} \prod_{j=1}^m \bigg(\sgn(x_j) + \sgn(y_j)\bigg) f_2(\vec{\bm{x}}) q^{Q_m(\vec{\bm{x}})} e^{2\pi i B_m(\vec{\bm{x}}, \vec{\bm{b}}(0))}.
	\end{align*}
\end{lemma}

\begin{proof}
	Since $\smat{\varepsilon \\ 1/2} \in R(\smat{0 \\ 1}) \cap R(\smat{-1 \\ 0})$, the condition in \cref{def:indef-theta} is satisfied, and the theta function is defined as follows:
	\begin{align*}
		\theta_{\vec{\bm{a}}'+ \vec{\bm{\varepsilon}}, \vec{\bm{b}}(\varepsilon)}^{\vec{\bm{c}}_0, \vec{\bm{c}}_1}[f_2](\tau) = \sum_{\vec{\bm{x}} \in \vec{\bm{a}}' + \Z^{2m}} \prod_{j=1}^m \bigg(\sgn(x_j + \varepsilon) + \sgn(y_j)\bigg) f_2(\vec{\bm{x}} + \vec{\bm{\varepsilon}}) q^{Q_m(\vec{\bm{x}}+ \vec{\bm{\varepsilon}})} e^{2\pi i B_m(\vec{\bm{x}}+ \vec{\bm{\varepsilon}}, \vec{\bm{b}}(\varepsilon))}.
	\end{align*}
	As in the proof of the second identity in \cref{lem:Dmm-cusps}, we expand $f_2(\vec{\bm{x}} + \vec{\bm{\varepsilon}})$ as follows:
	\[
		f_2(\vec{\bm{x}} + \vec{\bm{\varepsilon}}) = \sum_{\vec{e} \in \Z_{\ge 0}^m} p_{\vec{e}}(\varepsilon) \prod_{j=1}^m x_j^{e_j}.
	\]
	A key observation here is that the polynomial $p_{\vec{e}}(\varepsilon) \in \Z[\varepsilon]$ is divisible by $\varepsilon^{3\#\{1 \le j \le m \mid e_j = 0\}}$. Then, the right-hand side becomes
	\begin{align*}
		\sum_{\vec{e} \in \Z_{\ge 0}^m} p_{\vec{e}}(\varepsilon) \prod_{j=1}^m \sum_{\substack{x_j \in \Z \\ y_j \in 1/2 + \Z}} \bigg(\sgn(x_j + \varepsilon) + \sgn(y_j)\bigg) x_j^{e_j} q^{(x_j+\varepsilon)y_j} e^{2\pi i ((x_j+\varepsilon)b_{j2} + y_j b_{j1})},
	\end{align*}
	where we set $\vec{\bm{b}}(\varepsilon) = (\smat{b_{11} \\ b_{12}}, \dots, \smat{b_{m1} \\ b_{m2}}) \in \R^{2m}$. For each $\vec{e}$ and $j$, we again separate the sum according to whether $x_j y_j = 0$ or not. 
	
	From the first case, we have
	\begin{align*}
		&\sum_{\substack{x_j \in \Z \\ y_j \in 1/2 + \Z \\ x_j y_j = 0}} \bigg(\sgn(x_j + \varepsilon) + \sgn(y_j)\bigg) x_j^{e_j} q^{(x_j+\varepsilon)y_j} e^{2\pi i ((x_j+\varepsilon)b_{j2} + y_j b_{j1})}\\
		&= 2\delta_{e_j, 0} \sum_{y_j \in 1/2 + \Z_{\ge 0}} q^{\varepsilon y_j} e^{2\pi i (\varepsilon b_{j2} + y_j b_{j1})} = 2\delta_{e_j, 0} \frac{q^{\varepsilon/2} e^{2\pi i (\varepsilon b_{j2} + b_{j1}/2)}}{1- q^{\varepsilon} e^{2\pi i b_{j1}}}.
	\end{align*}
	If $e_j > 0$, then it equals $0$. For $e_j = 0$, when $b_{j1} \in \Z$, the limit of this term alone as $\varepsilon \to 0$ diverges. Nevertheless, when combined with the factor $\varepsilon^3$ contained in $p_{\vec{e}}(\varepsilon)$, the product still converges to $0$ as $\varepsilon \to 0$, even in this case.
	
	Therefore, only the sum over terms with $x_j y_j \neq 0$ contributes. In this case, we can interchange the summation and limit, yielding
	\[
		\lim_{\varepsilon \to 0} \theta_{\vec{\bm{a}}'+ \vec{\bm{\varepsilon}}, \vec{\bm{b}}(\varepsilon)}^{\vec{\bm{c}}_0, \vec{\bm{c}}_1}[f_2](\tau) = \sum_{\substack{\vec{\bm{x}} \in \vec{\bm{a}}' + \Z^{2m} \\ x_j \neq 0\ (1 \le j \le m)}} \prod_{j=1}^m \bigg(\sgn(x_j) + \sgn(y_j)\bigg) f_2(\vec{\bm{x}}) q^{Q_m(\vec{\bm{x}})} e^{2\pi i B_m(\vec{\bm{x}}, \vec{\bm{b}}(0))}.
	\]
	Since $f_2(\vec{\bm{x}}) = 0$ whenever $x_j = 0$ for some $j$, the sum coincides with the one taken over all $\vec{\bm{x}} \in \vec{\bm{a}}' + \Z^{2m}$.
\end{proof}

The expression for $F_2(\tau)$ in \eqref{eq:def:F123} is justified in this sense.

\begin{lemma}\label{lem:Dmm+1-Gamma2}
	We have the following:
	\begin{align*}
		F_2(\tau+2) &= F_2(\tau),\\
		F_2 \left(\frac{\tau}{2\tau+1}\right) &= (2\tau + 1)^{2m(m+1)} F_2(\tau).
	\end{align*}
\end{lemma}

\begin{proof}
	We replace $\vec{\bm{a}}'$ with $\vec{\bm{a}}' + \vec{\bm{\varepsilon}}$, where $\vec{\bm{\varepsilon}} = (\smat{\varepsilon \\ 0}, \dots, \smat{\varepsilon \\ 0})$ with $0 < \varepsilon < 1$. Since $\smat{\varepsilon \\ 1/2} \in R(\smat{0 \\ 1}) \cap R(\smat{-1 \\ 0})$, we may apply \cref{thm:indef-theta-S}. By \cref{thm:shift+1}, we have
	\begin{align*}
		\theta_{\vec{\bm{a}}'+ \vec{\bm{\varepsilon}}, \vec{\bm{0}}}^{\vec{\bm{c}}_0, \vec{\bm{c}}_1}[f_2](\tau+2) = e^{-4\pi i Q_m(\vec{\bm{a}}' + \vec{\bm{\varepsilon}})} \theta_{\vec{\bm{a}}' + \vec{\bm{\varepsilon}}, 2\vec{\bm{a}}' + 2\vec{\bm{\varepsilon}}}^{\vec{\bm{c}}_0, \vec{\bm{c}}_1}[f_2](\tau).
	\end{align*}
	By \cref{lem:F2-modified}, taking the limit as $\varepsilon \to 0$ on both sides, followed by an applying an appropriate scaling, yields $F_2(\tau+2) = F_2(\tau)$.
	
	Similarly as in the proof of \cref{lem:Dmm-Gamma2}, we have
	\begin{align*}
		\theta_{\vec{\bm{a}}' + \vec{\bm{\varepsilon}}, \vec{\bm{0}}}^{\vec{\bm{c}}_0(t), \vec{\bm{c}}_1(t)}[f_2] \left(\frac{\tau}{2\tau+1}\right) 
%		&= \left(2 + \frac{1}{\tau}\right)^{m+m(2m-1)} \theta_{\vec{\bm{0}}, \vec{\bm{a}}}^{\vec{\bm{c}}_0(t), \vec{\bm{c}}_1(t)}[f_1] \left(-2 -\frac{1}{\tau}\right)\\
%		&= \left(2 + \frac{1}{\tau}\right)^{2m^2} \theta_{\vec{\bm{0}}, \vec{\bm{a}}}^{\vec{\bm{c}}_0(t), \vec{\bm{c}}_1(t)}[f_1] \left(-\frac{1}{\tau}\right)\\
		&= (2\tau + 1)^{2m(m+1)} \theta_{-\vec{\bm{a}} - \vec{\bm{\varepsilon}}, \vec{\bm{0}}}^{\vec{\bm{c}}_0(t), \vec{\bm{c}}_1(t)}[f_2] (\tau).
	\end{align*}
	Taking the limits $t \to 0$ and $\varepsilon \to 0$ on both sides completes the proof.
\end{proof}

\begin{lemma}\label{lem:Dmm+1-cusps}
	We have the following:
	\begin{align*}
		F_2(\tau) &= q^{m(m+1)/4} + o(q^{m(m+1)/4}),\\
		(-i \tau)^{-2m(m+1)} F_2 \left(-\frac{1}{\tau}\right) &= O(1),\\
		(-i \tau)^{-2m(m+1)} F_2 \left(\frac{\tau-1}{\tau}\right) &= O(q^{m(m+1)/4}).
	\end{align*}
\end{lemma}

\begin{proof}
We can prove the second and third formulas in a similar way by combining the arguments of \cref{lem:Dmm-cusps} and \cref{lem:Dmm+1-Gamma2},  so we omit the details. 
	
	To prove the first result, we begin by identifying the lowest-order term in the $q$-series expansion of $F_2(\tau)$. This term is attained at $\vec{\bm{x}}_0 = (\smat{1 \\ 1/2}, \smat{2 \\ 1/2}, \dots, \smat{m \\ 1/2})$, which yields the order
	\[
		Q_m(\vec{\bm{x}}_0) = \sum_{j=1}^m \frac{j}{2} = \frac{m(m+1)}{4}.
	\]
	To compute the coefficient of $q^{m(m+1)/4}$, we use the identity
	\[
		f_2(\vec{\bm{x}}_0) = m! \prod_{j=1}^m (2j-1)!^2,
	\]
	which can be verified by induction on $m$. By counting the permutations of the entries of $\vec{\bm{x}}_0$ (in $m!$ ways) and the sign changes (in $2^m$ ways), and noting that $\sgn(x_j) + \sgn(y_j) = \pm 2$, we find that the coefficient of $q^{m(m+1)/4}$ in $F_2(\tau)$ equals
	\[
		\frac{1}{\prod_{j=1}^m (2j)!^2} \cdot m! \cdot 2^m \cdot 2^m f_2(\vec{\bm{x}}_0) = 1
	\]
	as desired.
\end{proof}

By comparing \cref{lem:Dmm+1-Gamma2}, \cref{lem:Dmm+1-cusps}, and \cref{lem:theta-triangle}, we see that the identity $\theta_\triangle(\tau)^{4m(m+1)} = F_2(\tau)$ follows in the same way as in the previous subsection.

% --------------------------------------------------------------------------
\subsection{Proof of the third identity (\cref{thm:Dmm-Eisen})}\label{sec:third_identity}
% --------------------------------------------------------------------------

The same assertions as in \cref{lem:Dmm-Gamma2} and \cref{lem:Dmm-cusps} also hold for $F_3(\tau)$ as well. Since the proofs are completely the same except for the first and second statements of \cref{lem:Dmm-cusps}, we will prove only these two.

For the second statement of \cref{lem:Dmm-cusps}, the only change is that the expansion of $f_1(\vec{\bm{x}} - \vec{\bm{\varepsilon}})$ given in \eqref{eq:f1-expand} is replaced by 
    \[
	f_3(\vec{\bm{x}} - \vec{\bm{\varepsilon}}) = \sum_{\vec{e}, \vec{f} \in \Z_{\ge 0}^m} p_{\vec{e}, \vec{f}}(\varepsilon) \prod_{j=1}^m x_j^{e_j} y_j^{f_j},
    \]
which now involves the variables $y_j$ as well. However, since the rest of the argument proceeds in exactly the same way, we omit the detailed discussion.

For the first statement, we aim to identify the lowest-degree term in the $q$-series expansion of $F_3(\tau)$ and its coefficient. By changing variables as $X_i = x_i + y_i$ and $Y_i = x_i - y_i$, we obtain
\begin{align*}
	f_3(\vec{\bm{x}}) &= \prod_{j=1}^m X_j \cdot \prod_{1 \le i < j \le m} (X_i^2 - X_j^2) (Y_i^2 - Y_j^2),\\
	Q_m(\vec{\bm{x}}) &= \frac{1}{4} \sum_{j=1}^m (X_j^2 - Y_j^2),
\end{align*}
and
\begin{align*}
	\sgn(x_j) + \sgn(y_j) = \sgn(X_j+Y_j) + \sgn(X_j - Y_j).
\end{align*}
Therefore, the minimal degree of $F_3(\tau)$ is given by
\begin{align*}
	\min \left\{\frac{1}{4} \sum_{j=1}^m (X_j^2 - Y_j^2) \, \middle|\, \begin{array}{l} 
		X_j, Y_j \in \Z,\ X_j - Y_j \not\in 2\Z,\\
		X_i^2 \neq X_j^2,\ Y_i^2 \neq Y_j^2,\ (i \neq j),\\
		X_j \neq 0,\ X_j^2 > Y_j^2
 \end{array}
 \right\},
\end{align*}
which is attained at $\vec{\bm{X}}_0 = (\smat{1 \\ 0}, \smat{2 \\ 1}, \dots, \smat{m \\ m-1})$, that is, $\vec{\bm{x}}_0 = (\smat{1/2 \\ 1/2}, \smat{3/2 \\ 1/2}, \dots, \smat{m-1/2 \\ 1/2})$. Thus the lowest-degree is
    \[
	Q_m(\vec{\bm{x}}_0) = \sum_{j=1}^m \frac{1}{2} \left(j - \frac{1}{2}\right) = \frac{m^2}{4}.
   \]
By induction, one can verify that
    \[
	f_3(\vec{\bm{x}}_0) = \frac{2\cdot m!}{(2m)!} \prod_{j=1}^m (2j-1)!^2.
    \]
Taking into account the $m!$ permutations of the entries of $\vec{\bm{X}}_0$ and then $2^{2m-1}$ sign changes (excluding $Y_1 = 0$), and the fact that $\sgn(x_j) + \sgn(y_j) = \pm 2$, we find that the coefficient of $q^{m^2/4}$ in $F_3(\tau)$ equals
    \[
	\frac{(2m)!}{2^m \prod_{j=1}^m (2j)!^2} \cdot m! \cdot 2^{2m-1} \cdot 2^m \cdot f_3(\vec{\bm{x}}_0) = 1.
    \]
Thus, we conclude that $F_3(\tau) = q^{m^2/4} + o(q^{m^2/4})$. As in the previous cases, it follows from these results that $\theta_\triangle(\tau)^{4m^2} = F_3(\tau)$.

% --------------------------------------------------------------------------
\section{Affine Lie superalgebras}\label{sec:affile-Lie-super}
% --------------------------------------------------------------------------

In this section,  we prepare neccessary notations for affine Lie superalgebras and root systems. 
For the definition and basic concepts of Lie superalgebras,  we refer the readers to the book of Musson~\cite{Musson}.
Later we focus on the four cases $\widehat{\gl}(m,  m)$,  $\widehat{\sl}(m+1,  m)$,  $\widehat{\spo}(2m,  2m)$,  and $\widehat{\spo}(2m,  2m+2)$.

% --------------------------------------------------------------------------
\subsection{Lie superalgebras}
% --------------------------------------------------------------------------

First,  we recall the finite dimensional classical Lie superalgebras.
Let $M,  N$ be non-negative integers.
Let $\gl(M,  N)=\M_{M+N}(\C)$ be the $\C$-vector space with the $\Z/2\Z$-grading 
    \begin{align*}
    \gl(M,  N)_0 &= \left\{
        \begin{pmatrix}
        A &  \\
         & D
        \end{pmatrix}
    \,\middle|\,  A\in\M_M(\C),  \ D\in\M_N(\C) \right\},  \\
    \gl(M,  N)_1 &= \left\{
        \begin{pmatrix}
         & B \\
        C & 
        \end{pmatrix}
    \,\middle|\,  B\in\M_{M\times N}(\C),  \ C\in\M_{N\times M}(\C) \right\}.
    \end{align*}
The Lie bracket is a bilinear map $[\cdot,\cdot]\,\colon \gl(M,  N)\times\gl(M,  N)\to\gl(M,  N)$ given by 
    \[
    [X,  Y]=XY-(-1)^{ij}YX,  \qquad X\in\gl(M,N)_i,  \ Y\in\gl(M,N)_j.
    \]
The Lie superalgebra $\gl(M,  N)$ is called the \emph{general linear Lie superalgebra} of type $A(M-1,N-1)$.

The \emph{super-trace} is the linear form $\str\,\colon\gl(M,N)\to\C$ defined by 
\[
	\str \pmat{A & B \\ C & D} = \tr(A) - \tr(D),
\]
%$\str
%    \begin{pmatrix}
%    A & B \\
%    C & D
%    \end{pmatrix}
%=\tr(A)-\tr(D)$,  
where $\tr(\cdot)$ denotes the usual trace of matrices.
The subspace $\sl(M, N)\coloneqq\{X\in\gl(M,  N) \mid \str(X)=0\}$ is a subalgebra of $\gl(M,  N)$.
It is called the \emph{special linear Lie superalgebra} of type $A(M-1,N-1)$.

We define the non-degenerate bilinear form $(\cdot,  \cdot)\,\colon\gl(M,N)\times\gl(M,N)\to\C$ as $(X,  Y)=\str(XY)$,  where $XY$ is the usual product of matrices.
The linear map $\st\bullet\,\colon\gl(M,  N)\to\gl(M,  N)$ given by 
    \[
    \largest
        \begin{pmatrix}
        A & B \\
        C & D
        \end{pmatrix}
    =
        \begin{pmatrix}
        \t A & \t C \\
        -\t B & \t D
        \end{pmatrix},  \qquad 
        \begin{pmatrix}
        A & B \\
        C & D
        \end{pmatrix}
    \in\gl(M,  N)
    \]
is called the \emph{super-transpose}.

Suppose that $M=2m$ and $N=2n$ are even.
Let $\spo(2m,  2n)$ be the subalgebra of $\sl(2m,  2n)$ defined as
    \[
    \spo(2m,  2n) = \left\{X\in\sl(2m,  2n) \, \middle|\, \st X
         \begin{pmatrix}
        J_1 & \\
        & J_2
        \end{pmatrix}
    +
         \begin{pmatrix}
        J_1 & \\
        & J_2
        \end{pmatrix}
    X=0\right\}.
    \]
Here we set $J_1 = \smat{ & \1_m \\ -\1_m & }$, $J_2 = \smat{ & \1_n \\ \1_n & }$.
%$J_1=
%    \begin{pmatrix}
%     & \1_m \\
%    -\1_m & 
%    \end{pmatrix}$,  $J_2=
%    \begin{pmatrix}
%     & \1_n \\
%    \1_n & 
%    \end{pmatrix}$.
This is a subalgebra of $\sl(2m,  2n)$,  which is called the \emph{ortho-symplectic Lie superalgebra} of type $D(m, n)$.
It consists of $\smat{A & B \\ C & D} \in\gl(2m,  2n)$
%$
%    \begin{pmatrix}
%     A & B \\
%    C & D
%    \end{pmatrix}
%\in\gl(2m,  2n)$ 
such that
    \begin{align*}
    & A\in\sp(2m) \coloneqq \{X\in\M_{2m}(\C) \mid \t XJ_1+J_1X=0\},   \\
    & D\in\so(2n) \coloneqq \{X\in\M_{2n}(\C) \mid \t XJ_2+J_2X=0\},   \\
    & B\in\M_{2m\times 2n}(\C) ,  \quad C=J_2^{-1}\t BJ_1.
    \end{align*}  
Note that $\spo(2m,  2n)_0$ is a Lie algebra isomorphic to $\sp(2m)\oplus\so(2n)$.
Note that the restriction of the bilinear form $(\cdot,  \cdot)$ on $\gl(2m,  2n)$ to $\spo(2m,  2n)$ is still non-degenerate.

Let $\fg$ be one of the Lie superalgebras $\gl(M,  N)$,  $\sl(M,  N)$, and $\spo(2m,  2n)$.
In each case,  let $\fh$ be the \emph{Cartan subalgebra} of $\fg$ consisting of all diagonal elements.
A general element of $\fh$ is of the form
    \begin{equation}\label{Cartan-element}
    H = 
        \begin{cases}
        \diag(x_1,  \ldots,  x_M,  y_1,  \ldots,  y_N),  
        & \text{type $A(M-1,  N-1)$},   \\
        \diag(x_1,  \ldots,  x_m,  -x_1,  \ldots,  -x_m,  
        y_1,  \ldots,  y_n,  -y_1,  \ldots,  -y_n),  
        & \text{type $D(m,  n)$},  
        \end{cases}
    \end{equation}
where $x_i,  y_j\in\C$ and $\sum_{i=1}^Mx_i=\sum_{j=1}^Ny_j$ if $\fg=\sl(M,  N)$.

We define the \emph{affine Lie superalgebra} $\widehat{\fg}$ as follows.
As a vector space,  $\widehat{\fg}=\fg\otimes\C[t^{\pm1}]\oplus\C K\oplus\C d$  and the grading is given by $\widehat{\fg}_0=\fg_0\otimes\C[t^{\pm1}]\oplus\C K\oplus\C d$ and $\widehat{\fg}_1=\fg_1\otimes\C[t^{\pm1}]$.
For $X\in\fg$ and $i\in\Z$,  we simply write $X\otimes t^i$ as $Xt^i$.
The Lie bracket is determined by the following conditions:
\begin{itemize}
\item $[Xt^i,  Yt^j]=[X,Y]t^{i+j}+i\delta_{i,-j}(X,Y)K$ for $X,Y\in\fg$ and $i,  j\in\Z$.
Here,  $\delta_{k,l}$ denotes the Kronecker delta.
\item $[\widehat{\fg},  K]=0$.
\item $[d,  Xt^i]=iXt^i$ for $X\in\fg$ and $i\in \Z$.
\end{itemize}

Let $(\cdot,\cdot)\,\colon\widehat{\fg}\times\widehat{\fg}\to\C$ be the even super-symmetric pairing determined by the following conditions:
\begin{itemize}
\item $(Xt^i,  Yt^j)=\delta_{i,-j}(X,Y)$ for $X,Y\in\fg$ and $i,  j\in\Z$.
\item $(\fg\otimes\C[t^{\pm1}], K)=(\fg\otimes\C[t^{\pm1}], d)=0$.
\item $(K,K)=(d,d)=0$ and $(K,d)=1$.
\end{itemize}
One can see that $(\cdot,\cdot)$ is invariant and non-degenerate.

The Cartan subalgebra of $\widehat{\fg}$ is defined as $\widehat{\fh}=\fh\oplus\C K\oplus\C d$.
We regard each $\lambda\in\fh^\ast$ as an element of $\widehat{\fh}^\ast$ such that $\lambda(K)=\lambda(d)=0$.
Define $\delta, \gamma\in\widehat{\fh}^\ast$ by
    \[
    \delta(\fg\otimes\C[t^{\pm1}]) = \gamma(\fg\otimes\C[t^{\pm1}]) = 0,  \qquad
    \delta(K) = \gamma(d) = 0,  \qquad \delta(d) = \gamma(K) = 1.
    \]
Then $\widehat{\fh}^\ast=\fh^\ast\oplus\C\delta\oplus\C\gamma$.

The action of $W$ on $\fh^\ast$ naturally extend to $\widehat{\fh}^\ast$.
Since $\delta$ and $\gamma$ are orthogonal to roots in $\Phi$,  the action of $W$ on $\delta$ and $\gamma$ is trivial.

% --------------------------------------------------------------------------
\subsection{Root systems}
% --------------------------------------------------------------------------

Let $\fh^\ast$ denote the dual space of $\fh$.
We define $\varepsilon_i,  \delta_j\in\fh^\ast$ by $\varepsilon_i(H) = x_i$ and $\delta_j(H) = y_j$,  where $H\in\fh$ is of the form \eqref{Cartan-element}.
Then,  the root system of $\fg$ with respect to $\fh$ is given as $\Phi=\Phi_0\cup\Phi_1$,  where
    \begin{align*}
    \Phi_0 & \coloneqq 
        \begin{cases}
        \{\pm(\varepsilon_i-\varepsilon_j) \mid 1\leq i<j\leq M\}
        \cup\{\pm(\delta_i-\delta_j) \mid 1\leq i<j\leq N \},   
        & \text{type $A(M-1,  N-1)$},  \\[5pt]
         \begin{array}{r}
         \{\pm(\varepsilon_i-\varepsilon_j), \ \pm(\varepsilon_i+\varepsilon_j), 
        \ \pm2\varepsilon_p \mid 1\leq i<j\leq m, \ 1\leq p\leq m\}   \\
        \cup\{\pm(\delta_i-\delta_j), \ \pm(\delta_i+\delta_j) 
        \mid 1\leq i<j\leq n \},   
        \end{array}
        & \text{type $D(m,  n)$},
        \end{cases}
    \\
    \Phi_1 & \coloneqq
        \begin{cases}
         \{\pm(\varepsilon_i-\delta_j) \mid 1\leq i \leq M,  \ 1\leq j \leq N \},
         & \text{type $A(M-1,  N-1)$},  \\[5pt]
          \{\pm(\varepsilon_i-\delta_j), \ \pm(\varepsilon_r+\delta_s) 
          \mid 1\leq i,  r \leq m,  \ 1\leq j,  s \leq n \},
         & \text{type $D(m,  n)$}.
         \end{cases}
    \end{align*}
For $\alpha\in\Phi$, let $\fg_\alpha=\{X\in\fg \mid [H,  X]=\alpha(H)X \text{ for all $H\in\fh$}\}$ be the root subspace.
Then we have the root space decomposition $\fg_0 = \fh\oplus \bigoplus_{\alpha\in\Phi_0} \fg_\alpha$ and $\fg_1 = \bigoplus_{\alpha\in\Phi_1} \fg_\alpha$.

The \emph{Weyl group} $W$ of $\fg$ is defined to be the Weyl group of $\Phi_0$.
Note that $W$ is isomorphic to $\fS_M\times\fS_N$ if $\fg$ is of type $A(M-1,  N-1)$ and $(\fS_m\ltimes(\Z/2\Z)^m)\times(\fS_n\ltimes(\Z/2\Z)^{n-1})$ if $\fg$ is of type $D(m,  n)$.

The restriction of $(\cdot,  \cdot)$ defines a non-degenerate pairing on $\fh$ and it induces a linear isomorphism $\fh^\ast\xrightarrow{\sim}\fh$.
Hence we obtain the pairing on $\fh^\ast$ by pulling-back $(\cdot,  \cdot)$ on $\fh$ by this isomorphism,  which is also denoted as $(\cdot,  \cdot)$.
\begin{lemma}
We have 
    \[
    (\varepsilon_i,  \varepsilon_i) = -(\delta_j,  \delta_j) = 
        \begin{cases}
        1, & \text{type $A(M-1,  N-1)$},  \\
        1/2, & \text{type $D(m,  n)$}.
        \end{cases}
    \]
In particular,  we have
\begin{itemize}\setlength{\itemsep}{5pt}
\item $(\varepsilon_i-\varepsilon_j,  \varepsilon_i-\varepsilon_j)=-(\delta_r-\delta_s,  \delta_r-\delta_s)=
    \begin{cases}
    2, & \text{type $A(M-1,  N-1)$},    \\
    1, &  \text{type $D(m,  n)$},  
    \end{cases}$
\item $(\varepsilon_i\pm\delta_j,  \varepsilon_i\pm\delta_j)=0$.
\end{itemize}
\end{lemma}

In the following,  we consider the four cases $\fg=\gl(m,  m)$,  $\sl(m+1,  m)$,  $\spo(2m,  2m)$,  and $\spo(2m,  2m+2)$.
We introduce a fundamental set of roots $\Delta$ and the corresponding set of positive roots $\Phi_i^+$, $(i=0,1)$ by a case-by-case argument.
We define the \emph{denominator} of $\fg$ by
    \[
    R = \frac{\prod_{\alpha\in\Phi_0^+}(1-e^{-\alpha})}
    {\prod_{\alpha\in\Phi_1^+}(1+e^{-\alpha})}.
    \]

% --------------------------------------------------------------------------
\subsubsection{Case $\gl(m,  m)$}
% --------------------------------------------------------------------------

\begin{definition}
For $i=1,  \ldots,  2m-1$,  set
    \[
    \alpha_i=
        \begin{cases}
        \varepsilon_k-\delta_k & i=2k-1,  \quad (1\leq k\leq m) \\
        \delta_k-\varepsilon_{k+1} & i=2k,  \quad (1\leq k\leq m-1)
        \end{cases}
    \]
and $\Delta=\{ \alpha_i \mid i=1,  \ldots,  2m-1 \}$.
\end{definition}
Put $\Phi_i^+=(\sum_{\alpha\in\Delta}\Z_{\geq0}\cdot\alpha)\cap\Phi_i$,  ($i=0,  1$) and $\Phi^+=\Phi_0^+\cup\Phi_1^+$.
Then we have 
    \begin{align*}
    \Phi_0^+ & = \{\varepsilon_i-\varepsilon_j  \mid 1\leq i<j\leq m \} 
    \cup\{\delta_i-\delta_j \mid 1\leq i<j\leq m \},   \\
    \Phi_1^+ &= \{\varepsilon_i-\delta_j, \  \delta_k-\varepsilon_l
    \mid 1\leq i\leq j\leq m,  \ 1\le k<l\leq m  \}.
    \end{align*}
Set $\Phi_0^\sharp = \{\alpha\in\Phi_0 \mid (\alpha,  \alpha)>0 \}= \{\pm(\varepsilon_i-\varepsilon_j)  \mid 1\leq i<j\leq m \}$.
Note that this is a root system of type $A_{m-1}$.
Let $M^\sharp$ be the coroot lattice of $\Phi_0^\sharp$,  \emph{i.e.} the free $\Z$-module of rank $m-1$ spanned by $\alpha_{2i-1}+\alpha_{2i}=\varepsilon_i-\varepsilon_{i+1}$,   $i=1,  \ldots,  m-1$ and $W^\sharp\simeq\fS_m$ denote the Weyl group of $\Phi_0^\sharp$.

% --------------------------------------------------------------------------
\subsubsection{Case $\sl(m+1,  m)$}
% --------------------------------------------------------------------------

\begin{definition}
For $i=1,  \ldots,  2m$,  set
    \[
    \alpha_i=
        \begin{cases}
        \varepsilon_k-\delta_k & i=2k-1,   \\
        \delta_k-\varepsilon_{k+1} & i=2k,  
        \end{cases}
    \quad (1\leq k\leq m)
    \]
and $\Delta=\{\alpha_i \mid i=1,  \ldots,  2m\}$.
\end{definition}
Put $\Phi_i^+=(\sum_{\alpha\in\Delta}\Z_{\geq0}\cdot\alpha)\cap\Phi_i$,  ($i=0,  1$) and $\Phi^+=\Phi_0^+\cup\Phi_1^+$.
Then we have
     \begin{align*}
    \Phi_0^+ & = \{\varepsilon_i-\varepsilon_j  \mid 1\leq i<j\leq m+1 \} 
    \cup\{\delta_i-\delta_j \mid 1\leq i<j\leq m \},   \\
    \Phi_1^+ &= \{\varepsilon_i-\delta_j, \  \delta_k-\varepsilon_l
    \mid 1\leq i\leq j\leq m,  \ 1\le k<l\leq m+1  \}.
    \end{align*}
Set $\Phi_0^\sharp = \{\alpha\in\Phi_0 \mid (\alpha,  \alpha)>0 \}= \{\pm(\varepsilon_i-\varepsilon_j)  \mid 1\leq i<j\leq m+1 \}$.
Note that this is a root system of type $A_m$.
Let $M^\sharp$ be the coroot lattice of $\Phi_0^\sharp$,  \emph{i.e.} the free $\Z$-module of rank $m$ spanned by $\alpha_{2i-1}+\alpha_{2i}=\varepsilon_i-\varepsilon_{i+1}$,   $i=1,  \ldots,  m$ and $W^\sharp\simeq\fS_{m+1}$ denote the Weyl group of $\Phi_0^\sharp$. 

% --------------------------------------------------------------------------
\subsubsection{Case $\spo(2m,  2m)$}
% --------------------------------------------------------------------------

\begin{definition}
For $i=1,  \ldots,  2m$,  set
    \[
    \alpha_i=
        \begin{cases}
        \varepsilon_k-\delta_k & i=2k-1,  \ (1\leq k\leq m),  \\
        \delta_k-\varepsilon_{k+1} & i=2k,  \ (1\leq k\leq m-1),  \\
        \delta_m+\varepsilon_m & i=2m
        \end{cases}
    \]
and $\Delta=\{\alpha_i \mid i=1,  \ldots,  2m\}$.
\end{definition}

Put $\Phi_i^+=(\sum_{\alpha\in\Delta}\Z_{\geq0}\cdot\alpha)\cap\Phi_i$,  ($i=0,  1$) and $\Phi^+=\Phi_0^+\cup\Phi_1^+$.
Then we have
    \begin{align*}
    \Phi_0^+ & = \{\varepsilon_i-\varepsilon_j, \ \varepsilon_i+\varepsilon_j, 
    \ 2\varepsilon_p \mid 1\leq i<j\leq m, \ 1\leq p\leq m\} 
    \cup\{\delta_i-\delta_j, \ \delta_i+\delta_j \mid 1\leq i<j\leq m \},   \\
    \Phi_1^+ &= \{\varepsilon_i-\delta_j, \ 
    \delta_k-\varepsilon_l, \  \varepsilon_r+\delta_s
    \mid 1\leq i\leq j\leq m,  \ 1\le k<l\leq m,  \ 1\leq r,  s\leq m \}.
    \end{align*}
Set $\Phi_0^\sharp=\{\alpha\in\Phi_0 \mid (\alpha,  \alpha)>0 \}=\{\pm(\varepsilon_i-\varepsilon_j), \ \pm(\varepsilon_i+\varepsilon_j), \ \pm2\varepsilon_p \mid 1\leq i<j\leq m, \ 1\leq p\leq m\}$.
Note that this is a root system of type $C_m$.
Let $M^\sharp$ be the coroot lattice of $\Phi_0^\sharp$,  \emph{i.e.} the free $\Z$-module of rank $m$ spanned by $\{ 2\varepsilon_i \mid i=1,  \ldots,  m\}$ and $W^\sharp\simeq\fS_m\ltimes(\Z/2\Z)^m$ denote the Weyl group of $\Phi_0^\sharp$.

% --------------------------------------------------------------------------
\subsubsection{Case $\spo(2m,  2m+2)$}
% --------------------------------------------------------------------------

\begin{definition}
For $i=1,  \ldots,  2m+1$,  set
    \[
    \alpha_i=
        \begin{cases}
        \delta_k-\varepsilon_k & i=2k-1,  \ (1\leq k\leq m),     \\
        \varepsilon_k-\delta_{k+1} & i=2k,  \ (1\leq k\leq m),   \\
        \varepsilon_m+\delta_{m+1} & i=2m+1
        \end{cases}
    \]
and $\Delta=\{\alpha_i \mid i=1,  \ldots,  2m+1\}$.
\end{definition}
Put $\Phi_i^+=(\sum_{\alpha\in\Delta}\Z_{\geq0}\cdot\alpha)\cap\Phi_i$,  ($i=0,  1$) and $\Phi^+=\Phi_0^+\cup\Phi_1^+$.
Then we have
    \begin{align*}
    \Phi_0^+ & = \{\delta_i-\delta_j, \ \delta_i+\delta_j,  \mid 1\leq i<j\leq m+1 \} 
    \cup\{ \varepsilon_i-\varepsilon_j, \varepsilon_i+\varepsilon_j,  \ 
    2\varepsilon_p \mid 1\leq i<j\leq m,  \ 1\leq p\leq m \},   \\
    \Phi_1^+ &= \{ \delta_i-\varepsilon_j, \  \varepsilon_k-\delta_l, \  
    \varepsilon_r+\delta_s
    \mid 1\leq i\leq j\leq m,  \ 1\le k<l\leq m+1,  \ 1\leq r\leq m,  \ 1\leq s\leq m+1 \}.
    \end{align*}
Set $\Phi_0^\sharp=\{\alpha\in\Phi_0 \mid (\alpha,  \alpha)<0 \}= \{\pm(\delta_i-\delta_j), \ \pm(\delta_i+\delta_j) \mid 1\leq i<j\leq m+1 \}$.
This is a root system of type $D_{m+1}$.
We write the Weyl group of $\Phi_0^\sharp$ by $W^\sharp\simeq\fS_{m+1}\ltimes(\Z/2\Z)^m$.
Let $M^\sharp$ be the coroot lattice of $\{\alpha\in\Phi_0 \mid (\alpha,  \alpha)>0\}=\{\pm(\varepsilon_i-\varepsilon_j), \ \pm(\varepsilon_i+\varepsilon_j), \ \pm2\varepsilon_p \mid 1\leq i<j\leq m, \ 1\leq p\leq m\}$,  \emph{i.e.} the free $\Z$-module of rank $m$ spanned by $\{ 2\varepsilon_i \mid i=1,  \ldots,  m\}$.

% --------------------------------------------------------------------------
\subsection{Affine root systems}
% --------------------------------------------------------------------------

Let $\theta$ be the highest root in $\Phi$ with respect to $\Delta$ and set $\widehat{\Delta}=\Delta\cup\{\delta-\theta\}$.
This is a set of fundamental roots of the root system $\widehat{\Phi}=\widehat{\Phi}_0\cup\widehat{\Phi}_1$ of $\widehat{\fg}$.
Here,  
    \begin{align*}
    \widehat{\Phi}_0 & \coloneqq \{\alpha+n\delta \mid \alpha\in\Phi_0,  \ n\in\Z\} \cup \{n\delta \mid n\in\Z\setminus\{0\}\},  \\
    \widehat{\Phi}_1 & \coloneqq \{\alpha+n\delta \mid \alpha\in\Phi_1,  \ n\in\Z\}.
    \end{align*}
Put $\widehat{\Phi}_i^+=(\sum_{\alpha\in\widehat{\Delta}}\Z_{\geq0}\cdot\alpha)\cap\widehat{\Phi}_i$,  ($i=0,  1$) and $\widehat{\Phi}^+=\widehat{\Phi}_0^+\cup\widehat{\Phi}_1^+$.
A simple calculation shows that 
    \[
    \widehat{\Phi}_0^+  = \Phi_0^+\cup 
    \{ \alpha+n\delta \mid \alpha\in\Phi_0\cup\{0\},\ n\in\Z_{>0} \},  \qquad
    \widehat{\Phi}_1^+ = \Phi_1^+\cup 
    \{\alpha+n\delta \mid \alpha\in\Phi_1, \ n\in\Z_{>0} \}.
    \]
In particular,  we have $\widehat{\Phi}_i=\widehat{\Phi}_i^+\cup(-\widehat{\Phi}_i^+)$,  ($i=0,1$).
This  implies that $\widehat{\Delta}$ is a fundamental set of roots of $\widehat{\Phi}$.

For $\beta\in\widehat{\Phi}$,  the corresponding root subspace is $\widehat{\fg}_\beta=\{X\in\widehat{\fg} \mid [H,  X]=\beta(H)X \text{ for all $H\in\widehat{\fh}$}\}$.
One sees that $\widehat{\fg}_{\alpha+n\delta}=\fg_\alpha\otimes t^n$ and $\widehat{\fg}_{n\delta}=\fh\otimes t^n$ for $\alpha\in\Phi$,  $n\in\Z$.
In particular,  we have the root space decomposition $\widehat{\fg}_0=\widehat{\fh}\oplus\bigoplus_{\beta\in\widehat{\Phi}_0}\widehat{\fg}_\beta$ and $\widehat{\fg}_1=\bigoplus_{\beta\in\widehat{\Phi}_1}\widehat{\fg}_\beta$.

\begin{definition}
For $\alpha\in\fh^\ast$,  we define $t_\alpha\,\colon \widehat{\fh}^\ast\to\widehat{\fh}^\ast$ by
    \[
    t_\alpha(\lambda) = \lambda+\lambda(K)\alpha-
    ((\lambda,\alpha)+\tfrac12(\alpha,\alpha)\lambda(K))\delta.
    \]
\end{definition}

We define the \emph{affine Weyl group} as $\widehat{W}=W\ltimes M^\sharp$.
One can show that $t_\alpha\in\Aut(\widehat{\fh}^\ast)$, $t_\alpha t_\beta=t_{\alpha+\beta}$, and $wt_\alpha w^{-1}=t_{w(\alpha)}$ for $w\in W$ and $\alpha,  \beta\in\fh^\ast$.
Hence we obtain an action of $\widehat{W}$ on $\widehat{\fh}^\ast$.

\begin{definition}
The denominator $\widehat{R}$ of $\widehat{\fg}$ is defined as 
    \[
    \widehat{R} = \frac{\prod_{\beta\in\widehat{\Phi}_0^+}
    (1-e^{-\beta})^{\dim(\widehat{\fg}_\beta)}}
    {\prod_{\beta\in\widehat{\Phi}_1^+}
    (1+e^{-\beta})^{\dim(\widehat{\fg}_\beta)}}
    = R\prod_{n=1}^\infty\left( (1-e^{-n\delta})^k\cdot
    \frac{\prod_{\alpha\in\Phi_0^+}(1-e^{-\alpha-n\delta})(1-e^{\alpha-n\delta})}
    {\prod_{\alpha\in\Phi_1^+}(1+e^{-\alpha-n\delta})
    (1+e^{\alpha-n\delta})}\right),
    \]
where 
    \[
    k = \dim(\fh) =
        \begin{cases}
        2m+1 & \text{if $\widehat{\fg}=\widehat{\spo}(2m,  2m+2)$},  \\ 
        2m & \text{otherwise}.
        \end{cases}
    \]
\end{definition}

Let $\rho=\frac12\left(\sum_{\alpha\in\Phi_0^+}\alpha-\sum_{\alpha\in\Phi_1^+}\alpha\right)$ be the Weyl vector of $\Phi^+$.
Take $\widehat{\rho}\in\widehat{\fh}^\ast$ so that $(\widehat{\rho},  \alpha)=(\alpha,  \alpha)/2$ for all $\alpha\in\widehat{\Delta}$.
One explicit choice is:
    \[
    \widehat{\rho} = 
        \begin{cases}
        \rho & \text{if $\widehat{\fg}=\widehat{\gl}(m,  m)$,  $\widehat{\spo}(2m,  2m+2)$},   \\
        \gamma & \text{if $\widehat{\fg}=\widehat{\sl}(m+1,  m)$,  $\widehat{\spo}(2m,  2m)$}.  \\
        \end{cases}
    \]
Note that we have $\widehat{\rho}=-\frac12\str$ when $\widehat{\fg}=\widehat{\gl}(m,  m)$ and $\widehat{\rho}=0$ when $\widehat{\fg}=\widehat{\spo}(2m,  2m+2)$.
Set $S=\{ \alpha_{2j-1} \mid j=1,  2,  \ldots,  m \}$,  \emph{i.e.}
    \[
    S = 
        \begin{cases}
        \{ \delta_j-\varepsilon_j \mid j=1,  2,  \ldots,  m \} 
        & \text{if $\widehat{\fg}=\widehat{\spo}(2m,  2m+2)$},   \\
        \{ \varepsilon_j-\delta_j \mid j=1,  2,  \ldots,  m \} &  \text{otherwise}.
        \end{cases}
    \]
This is a maximal totally isotropic subset of $\Delta$.
The next theorem is called the \emph{denominator identity} for $\widehat{\fg}$.
\begin{theorem}[Gorelik \cite{Gorelik11},  Gorelik--Reif \cite{GorelikReif2012}]
We have
    \begin{equation}\label{affine-denominator}
    e^{\widehat{\rho}}f(q) \widehat{R} = \sum_{\alpha^\sharp\in M^\sharp} t_{\alpha^\sharp}(
    e^{\widehat{\rho}}R) = \frac{|W^\sharp|}{|W|} \sum_{\alpha^\sharp\in M^\sharp}
    \sum_{w\in W}\epsilon(w)wt_{\alpha^\sharp}\left(
    \frac{e^{\widehat{\rho}}}{\prod_{\beta\in S}(1+e^{-\beta})} \right).
    \end{equation}
Here,  $q\coloneqq e^{-\delta}$,  
    \[
    f(q) \coloneqq
        \begin{cases}
        \prod_{n=1}^\infty (1-(-1)^mq^ne^{\str})(1-(-1)^mq^ne^{-\str})(1-q^n)^{-2} 
        &  \widehat{\fg} = \widehat{\gl}(m,  m),   \\
        \prod_{n=1}^\infty(1-q^n)^{-1} 
        & \widehat{\fg} = \widehat{\spo}(2m,  2m+2),  \\
        1 & \text{otherwise},
        \end{cases}
    \]
and $\epsilon\,\colon W\to\{\pm1\}$ is the sign character. 
\end{theorem}

% --------------------------------------------------------------------------
\section{Denominator identities}\label{sec:denom-identities}
% --------------------------------------------------------------------------

We deduce from \eqref{affine-denominator}  the power series identities for the powers of $\triangle(q)$.

% --------------------------------------------------------------------------
\subsection{Case $\widehat{\gl}(m,  m)$}\label{sec:gl^(m_m)}
% --------------------------------------------------------------------------

Suppose that $m>1$.
Let $\xi=-\frac12\sum_{i=1}^{m}\varepsilon_i$ so that $(\alpha_j,  \xi)=(-1)^j/2$ for $j=1,  \ldots,  2m-1$ and $w(\xi)=\xi$ for all $w\in W$.
Applying $t_\xi$ to the both sides of \eqref{affine-denominator},  we obtain
    \begin{equation}\label{t_xiA2}
    t_\xi(e^{\widehat{\rho}}f(q)\widehat{R}) 
    = \frac{1}{m!}\sum_{\alpha^\sharp\in M^\sharp}\sum_{w\in W} 
    \epsilon(w)wt_{\xi+\alpha^\sharp}
    \left( \frac{e^{\widehat{\rho}}}{\prod_{\beta\in S}(1+e^{-\beta})} \right).
    \end{equation}

Write $\alpha^\sharp\in M^\sharp$ as $\alpha^\sharp=\sum_{i=1}^{m-1} (k_1+\cdots+k_i)(\alpha_{2i-1}+\alpha_{2i})=\sum_{i=1}^{m-1} k_i(\varepsilon_i-\varepsilon_m)$ with $\vec{k}=(k_1,  \ldots,  k_{m-1})\in\Z^{m-1}$.
For $\vec{k}=(k_1,  \ldots,  k_{m-1})\in\Z^{m-1}$,  set $|\vec{k}|=k_1+k_2+\cdots+k_{m-1}$.
We summarize some equations which we use to compute the right-hand side of \eqref{t_xiA2}.
Each of them is a consequence of an easy calculation.
\begin{lemma}
We have the following equations: 
\begin{itemize}
\item $t_{\xi+\alpha^\sharp}(\alpha_{2j-1})=\alpha_{2j-1}-
    \begin{cases}
    (-1/2+k_j)\delta & \text{if $j=1,  \ldots,  m-1$,}  \\
    (-1/2-|\vec{k}|)\delta & \text{if $j=m$.}
    \end{cases}$
\item $t_{\xi+\alpha^\sharp}(\widehat{\rho})=\widehat{\rho}-(m/4)\delta$.
\item $w(\delta)=\delta$ for $w\in W$.
\end{itemize}
\end{lemma}

For an index set $J\subset\{1,  \ldots,  m-1\}$,  set 
    \begin{equation}\label{eq:Z_J^m}
    \Z^{m-1}_J = \left\{\vec{k}=(k_1,  \ldots,  k_{m-1})\in\Z^{m-1} \, 
    \middle|\, J=\{j \mid k_j\leq 0\} \right\}
    \end{equation} 
and define the subset $\Z^{m-1}_{J,  \pm}$ of $\Z^{m-1}_J$ by
    \[
    \Z^{m-1}_{J,+} = \left\{\vec{k}=(k_1,  \ldots,  k_{m-1})\in\Z_J^{m-1} \, \middle|\,
     |\vec{k}|\geq0 \right\}
    \]
and 
    \[
    \Z^{m-1}_{J,-} = \left\{\vec{k}=(k_1,  \ldots,  k_{m-1})\in\Z_J^{m-1} \, \middle|\,
     |\vec{k}|<0 \right\}.
    \]
The right-hand side of \eqref{t_xiA2} becomes 
    \begin{align*}
    & \frac{q^{\frac{m}{4}}}{m!}
    \sum_{J\subset\{1,\ldots,m-1\}} \Biggl\{ \sum_{\vec{k}\in\Z^{m-1}_{J,+}}\sum_{w\in W}\epsilon(w) \\
    & \times w\left( e^{\widehat{\rho}} 
     \sum_{\vec{r}\in\Z_{\ge 0}^m} \left(-q^{\frac12+|\vec{k}|}e^{-\alpha_{2m-1}} \right)^{r_m}
    \prod_{j\in J} (-q^{-(k_j-\frac12)}e^{-\alpha_{2j-1}})^{r_j} \cdot
    \prod_{j\not\in J} \left(-(-q^{k_j-\frac12}e^{\alpha_{2j-1}})^{r_j+1}\right) \right) \\
    &\quad + \sum_{\vec{k}\in\Z^{m-1}_{J,-}}\sum_{w\in W}\epsilon(w) \\
    & \times w\left( e^{\widehat{\rho}} 
     \sum_{\vec{r}\in\Z_{\ge 0}^m} (-1)\left(-q^{-(\frac12+|\vec{k}|)}e^{\alpha_{2m-1}} \right)^{r_m+1}
    \prod_{j\in J} (-q^{-(k_j-\frac12)}e^{-\alpha_{2j-1}})^{r_j} \cdot
    \prod_{j\not\in J} \left(-(-q^{k_j-\frac12}e^{\alpha_{2j-1}})^{r_j+1}\right) \right) \Biggr\}.
    \end{align*}
Note that we have 
    \begin{align*}
    & \sum_{\vec{r}\in\Z_{\ge 0}^m} \left(-q^{\frac12+|\vec{k}|}e^{-\alpha_{2m-1}} \right)^{r_m}
    \prod_{j\in J} (-q^{-(k_j-\frac12)}e^{-\alpha_{2j-1}})^{r_j} \cdot
    \prod_{j\not\in J} \left(-(-q^{k_j-\frac12}e^{\alpha_{2j-1}})^{r_j+1}\right) \\
    & = (-1)^{m-1-|J|}\sum_{r_m\geq0}\left(-q^{\frac12+|\vec{k}|}e^{-\alpha_{2m-1}} \right)^{r_m}
    \sum_{-\vec{r}\in\Z_J^{m-1}} (-1)^{\sum_{j=1}^{m-1}r_j}
    q^{-\sum_{j=1}^{m-1}(k_j-\frac12)r_j}e^{-\sum_{j=1}^{m-1}r_j\alpha_{2j-1}} \\
    & = (-1)^{m-|J\sqcup\{m\}|}\sum_{-\vec{r}\in\Z_{J\sqcup\{m\}}^m}(-1)^{\sum_{j=1}^{m}r_j}
    q^{(\frac12+|\vec{k}|)r_m-\sum_{j=1}^{m-1}(k_j-\frac12)r_j}e^{-\sum_{j=1}^{m}r_j\alpha_{2j-1}}
    \end{align*}
and similarly
    \begin{align*}
    & \sum_{\vec{r}\in\Z_{\ge 0}^m} (-1)\left(-q^{-(\frac12+|\vec{k}|)} e^{\alpha_{2m-1}} \right)^{r_m+1}
    \prod_{j\in J} (-q^{-(k_j-\frac12)} e^{-\alpha_{2j-1}})^{r_j} \cdot
    \prod_{j\not\in J} \left(-(-q^{k_j-\frac12} e^{\alpha_{2j-1}})^{r_j+1}\right) \\
    & = (-1)^{m-|J|} \sum_{r_m<0} \left(-q^{\frac12+|\vec{k}|} e^{-\alpha_{2m-1}} \right)^{r_m}
    \sum_{-\vec{r}\in\Z_J^{m-1}} (-1)^{\sum_{j=1}^{m-1}r_j}
    q^{-\sum_{j=1}^{m-1}(k_j-\frac12)r_j}e^{-\sum_{j=1}^{m-1}r_j\alpha_{2j-1}} \\
    & = (-1)^{m-|J|}\sum_{-\vec{r}\in\Z_{J}^m} (-1)^{\sum_{j=1}^{m}r_j}
    q^{(\frac12+|\vec{k}|)r_m-\sum_{j=1}^{m-1}(k_j-\frac12)r_j} e^{-\sum_{j=1}^{m}r_j\alpha_{2j-1}}.
    \end{align*}
Hence the right-hand side of \eqref{t_xiA2} multiplied with $q^{-m/4}$ equals
    \begin{multline*}\label{heartA2}\tag{$\spadesuit$1}{}
     \frac{1}{m!}
    \sum_{J\subset\{1,\ldots,m\}} (-1)^{|J^c|} \sum_{\substack{\vec{k}\in\Z^{m-1} \\ (\vec{k},  -|\vec{k}|)\in\Z^m_J}}
    \sum_{-\vec{r}\in\Z^m_J} (-1)^{\sum_{j=1}^{m}r_j}
    q^{(\frac12+|\vec{k}|)r_m-\sum_{j=1}^{m-1}(k_j-\frac12)r_j}  \\
     \qquad \times \sum_{w\in W} \epsilon(w)w\left( e^{\widehat{\rho}-\sum_{j=1}^m r_j\alpha_{2j-1}} \right).
    \end{multline*}

On the other hand,  since $t_\xi(\widehat{\rho})=\widehat{\rho}-(m/4)\delta$ and $t_\xi(\str)=\str+(m/2)\delta$,  the left-hand side of \eqref{t_xiA2} multiplied with $q^{-m/4}$ equals
    \begin{multline*}\label{heartA3}\tag{$\spadesuit$2}{}
    e^{\widehat{\rho}}
    \left( \prod_{n=1}^\infty (1-q^n)^{2m-2} (1-(-1)^mq^{n-m/2}e^{\str})(1-(-1)^mq^{n+m/2}e^{-\str}) \right)  \\
    \times \frac{\prod_{\beta\in\Phi_0^+}
    \left(\prod_{n=0}^\infty(1-q^{n-(\beta,  \xi)}e^{-\beta})\cdot
    \prod_{n=1}^\infty(1-q^{n+(\beta,  \xi)}e^{\beta}) \right)}
    {\prod_{\beta\in\Phi_1^+}
    \left(\prod_{n=0}^\infty(1+q^{n-(\beta,  \xi)}e^{-\beta})\cdot
    \prod_{n=1}^\infty(1+q^{n+(\beta,  \xi)}e^{\beta} \right)}.
    \end{multline*}

Let $R_0=\prod_{\beta\in\Phi_0^+}(1-e^{-\beta})$ be the denominator of $\Phi_0$.
From now,  we divide the both \eqref{heartA2} and \eqref{heartA3} by $R_0$ and take the limit $e^{-\alpha_j}\to1$,  ($j=1,  \ldots,  2m-1$). 
 
First,  we deal with \eqref{heartA2}.  

\begin{lemma}\label{lem:RHS-A2}
We have 
    \[
    \lim_{\substack{e^{-\alpha_{j}}\to1 \\ j=1,  \ldots,  2m-1}}
    R_0^{-1}  \sum_{w\in W} \epsilon(w)w \left( e^
    {\widehat{\rho}-\sum_{j=1}^m r_j\alpha_{2j-1}} \right)
    = \frac{(-1)^{m(m-1)/2}}{\left( \prod_{i=1}^{m-1}i! \right)^2} 
    \left( \prod_{1\leq i<j\leq m}(r_j-r_i)^2 \right).
    \]
\end{lemma}

\begin{proof}
Let $\rho_0=\frac12\sum_{\beta\in\Phi_0^+}\beta$ be the Weyl vector of $\Phi_0^+$.  
Set 
    \[
    \lambda = -\rho_0+\widehat{\rho}-\sum_{j=1}^m r_j\alpha_{2j-1}
    = -\rho_0+\frac12\sum_{i=1}^m(\delta_i-\varepsilon_i)-\sum_{j=1}^m r_j\alpha_{2j-1}.
    \]
From the Weyl dimension formula (and its proof,  \cite[Corollary 24.3]{Humphreys}),  we obtain
    \[
    \lim_{\substack{e^{-\alpha_{j}}\to1 \\ j=1,  \ldots,  2m-1}}
    R_0^{-1} \sum_{w\in W} \epsilon(w)w(e^{\rho_0+\lambda}) \quad
    = \prod_{\beta\in\Phi_0^+}\frac{(\rho_0+\lambda,  \beta)}
    {(\rho_0,  \beta)}.
    \]

Since $(\rho_0+\lambda,  \varepsilon_i-\varepsilon_j)=-r_i+r_j=(\rho_0+\lambda,  \delta_i-\delta_j)$ for $1\leq i<j\leq m$,  we get
    \[
    \prod_{\beta\in\Phi_0^+}(\rho_0+\lambda,  \beta)
    = \prod_{1\leq i<j\leq m} (r_j-r_i)^2.
    \]
On the other hand,  since we have 
    \[
    \rho_0=\frac12\sum_{p=1}^m (m+1-2p)\varepsilon_p
    +\frac12\sum_{q=1}^m(m+1-2q)\delta_q,
    \]
the pairing $(\rho_0,  \beta)$ for $\beta\in\Phi_0^+$ is computed as 
\begin{itemize}
\item $(\rho_0,  \varepsilon_i-\varepsilon_j)=j-i$ and 
    \[
    \prod_{1\leq i<j\leq m}(\rho_0,  \varepsilon_i-\varepsilon_j) 
    = \prod_{1\leq i<m}(m-i)! = \prod_{i=1}^{m-1} i! .
    \]
\item $(\rho_0,  \delta_i-\delta_j)=i-j$ and 
    \[
    \prod_{1\leq i<j\leq m}(\rho_0,  \delta_i-\delta_j) 
    =(-1)^{m(m-1)/2}\prod_{i=1}^{m-1} i! .
    \]
\end{itemize}
Hence we have $\prod_{\beta\in\Phi_0^+}(\rho_0,  \beta)=(-1)^{m(m-1)/2}\left(\prod_{i=1}^{m-1}i!\right)^2$.
This completes the proof.
\end{proof}

Combining \cref{lem:RHS-A2} with \eqref{heartA2} and changing the variables $r_i\mapsto-r_i$,  $i=1,  \ldots,  m$,  we obtain the following equation.

\begin{corollary}\label{lem:RHS-A2}
We have
    \begin{align}\label{eq:RHS-A2}
    \begin{split}
    \lim_{\substack{e^{-\alpha_{j}}\to1 \\ j=1,  \ldots,  2m}} 
    \frac{\eqref{heartA2}}{R_0}  
     & = \frac{(-1)^{m(m-1)/2}}{m!\left( \prod_{i=1}^{m-1}i! \right)^2}
     \sum_{J\subset\{1,\ldots,m\}} (-1)^{|J^c|} \\
   & \times \sum_{\substack{\vec{k}\in\Z^{m-1} \\ (\vec{k},  -|\vec{k}|)\in\Z^m_J}}
    \sum_{\vec{r}\in\Z^m_J} (-1)^{\sum_{j=1}^{m}r_j}
    q^{\sum_{j=1}^{m-1}(k_j-\frac12)r_j-(\frac12+|\vec{k}|)r_m} 
    \left( \prod_{1\leq i<j\leq m}(r_i-r_j)^2 \right).
    \end{split}
    \end{align}
\end{corollary}

Next we deal with \eqref{heartA3}.
When $m$ is even,  the factor $(1-e^{\str})$ appears in the product $\prod_{n=1}^\infty(1-q^{n-m/2}e^{\str})$ and this factor vanishes under the limit $e^{-\alpha_j}\to1$,  ($j=1,  \ldots,  2m$).
Hence we focus on the case where $m$ is odd. 

\begin{lemma}\label{lem:LHS-A2}
Suppose that $m>1$ is odd.
Then we have 
    \begin{align}\label{eq:LHS-A2}
    \lim_{\substack{e^{-\alpha_{j}}\to1 \\ j=1,  \ldots,  2m}} 
    \frac{\eqref{heartA3}}{R_0} 
    & =  \frac{q^{\frac{m^2-1}{8}}}{1+q^{\frac{m}{2}}} \cdot 
    \left( \prod_{n=1}^\infty \frac{1-q^n}{1+q^{\frac12(2n-1)}} \right)^{2m^2-2}.
    \end{align}
\end{lemma}
\begin{proof}
Let $\beta\in\Phi^+$.
It is easy to see that 
    \[
    (\beta,  \xi) = 
        \begin{cases}
        -\frac12 & \text{$\beta\in\{\varepsilon_i-\delta_j,  \ (1\leq i\leq j\leq m) \}$},  \\
        \frac12 & \text{$\beta\in\{\delta_i-\varepsilon_j, \ (1\leq i< j\leq m)\}$},  \\
        0 & \text{otherwise}.
        \end{cases}
    \]
If $(\beta,  \xi)=0$,  then 
    \begin{align*}
    (1-e^{-\beta})^{-1}
    \prod_{n=0}^\infty(1-q^{n-(\beta,  \xi)}e^{-\beta})
    \prod_{n=1}^\infty(1-q^{n+(\beta,  \xi)}e^{\beta}) 
    & = \prod_{n=1}^\infty (1-q^ne^{-\beta})(1-q^ne^\beta) \\
    & \quad \underset{j=1,  \ldots,  2m-1}{\xrightarrow{e^{-\alpha_{j}}\to1}} \quad
    \left(\prod_{n=1}^\infty (1-q^n)\right)^2.
    \end{align*}
If $(\beta,  \xi)=-1/2$,  then 
    \begin{multline*}
    \prod_{n=0}^\infty(1+q^{n-(\beta,  \xi)}e^{-\beta})
    \prod_{n=1}^\infty(1+q^{n+(\beta,  \xi)}e^{\beta})
    = \prod_{n=1}^\infty (1+q^{n-\frac12}e^{-\beta})
    (1+q^{n-\frac12}e^{\beta}) \\
    \underset{j=1,  \ldots,  2m-1}{\xrightarrow{e^{-\alpha_{j}}\to1}} \
     \left(\prod_{n=1}^\infty(1+q^{\frac12(2n-1)})\right)^2.
    \end{multline*}
If $(\beta,  \xi)=1/2$,  then 
    \begin{multline*}
    \prod_{n=0}^\infty(1+q^{n-(\beta,  \xi)}e^{-\beta})
    \prod_{n=1}^\infty(1+q^{n+(\beta,  \xi)}e^{\beta})
    = (1+q^{-\frac12}e^{-\beta})\prod_{n=1}^\infty (1+q^{n-\frac12}e^{-\beta})
    (1+q^{n+\frac12}e^{\beta}) \\
    = q^{-\frac12}e^{-\beta}\cdot 
    \prod_{n=1}^\infty(1+q^{\frac12(2n-1)}e^{-\beta})
    (1+q^{\frac12(2n-1)}e^{\beta}) \
    \underset{j=1,  \ldots,  2m-1}{\xrightarrow{e^{-\alpha_{j}}\to1}} \
     q^{-\frac12}\left(\prod_{n=1}^\infty(1+q^{\frac12(2n-1)})\right)^2.
    \end{multline*}
    
On the other hand,  we see that
    \begin{multline*}
    \prod_{n=1}^\infty(1+q^{n-m/2}e^{\str})(1+q^{n+m/2}e^{-\str}) \\
    \underset{j=1,  \ldots,  2m-1}{\xrightarrow{e^{-\alpha_{j}}\to1}} \
    \prod_{n=1}^\infty(1+q^{n-m/2})(1+q^{n+m/2}) 
    = (1+q^{m/2})^{-1}q^{-\frac12\sum_{j=1}^{(m-1)/2}(2j-1)}  \cdot \prod_{n=1}^\infty(1+q^{\frac12(2n-1)})^2.
    \end{multline*}

Note that $\#\{\beta\in\Phi^+ \mid (\beta,  \xi)=1/2\}=m(m-1)/2$. 
Thus we obtain
    \begin{align*}
    \begin{split}
    \lim_{\substack{e^{-\alpha_{j}}\to1 \\ j=1,  \ldots,  2m}} 
    \frac{\eqref{heartA3}}{R_0} 
    & = (q^{-\frac12})^{-\frac{m(m-1)}{2}+\frac{(m-1)^2}{4}}(1+q^{\frac{m}{2}})^{-1}  \cdot
    \left( \prod_{n=1}^\infty \frac{1-q^n}{1+q^{\frac12(2n-1)}} \right)^{2m^2-2}  \\ 
    & = \frac{q^{\frac{m^2-1}{8}}}{1+q^{\frac{m}{2}}} \cdot 
    \left( \prod_{n=1}^\infty \frac{1-q^n}{1+q^{\frac12(2n-1)}} \right)^{2m^2-2}.
    \end{split}
    \end{align*}
\end{proof}

From \cref{lem:RHS-A2} and \cref{lem:LHS-A2} we obtain the next theorem.

\begin{theorem}\label{thm:denom-id-gl}
Suppose that $m>1$ is odd.
We have 
    \begin{align}\label{result-A2}
    \begin{split}
    \triangle(q)^{2m^2-2} & = \frac{(-1)^{m(m+1)/2}}{m!\left( \prod_{i=1}^{m-1}i! \right)^2} (1-q^m)
    \sum_{J\subset\{1,\ldots,m\}} (-1)^{|J|} \\
    & \times \sum_{\substack{\vec{k}\in\Z^{m-1} \\ (\vec{k},  -|\vec{k}|)\in\Z^m_J}}
    \sum_{\vec{r}\in\Z^m_J} q^{\sum_{j=1}^{m-1}(2k_j-1)r_j -(1+2|\vec{k}|)r_m -\frac{m^2-1}{4}} 
    \left( \prod_{1\leq i<j\leq m}(r_i-r_j)^2 \right).
    \end{split}
    \end{align}
Here, we set $\Z^m_J = \left\{\vec{k}=(k_1,  \ldots,  k_m)\in\Z^m \, \middle|\,
J=\{j \mid k_j\leq0\} \right\}$ and $|\vec{k}|=k_1+k_2+\cdots+k_{m-1}$.
\end{theorem}

\begin{proof}
From \eqref{eq:RHS-A2} and \eqref{eq:LHS-A2},  we obtain
    \begin{multline*}
    \left( \prod_{n=1}^\infty \frac{1-q^{\frac12\cdot 2n}}{1+q^{\frac12(2n-1)}} \right)^{2m^2-2} 
    = \frac{(-1)^{m(m-1)/2}}{m!\left( \prod_{i=1}^{m-1}i! \right)^2} (1+q^{\frac{m}{2}})
    \sum_{J\subset\{1,\ldots,m\}} (-1)^{|J^c|} \\
    \times \sum_{\substack{\vec{k}\in\Z^{m-1} \\ (\vec{k},  -|\vec{k}|)\in\Z^m_J}}
    \sum_{\vec{r}\in\Z^m_J} (-1)^{\sum_{j=1}^{m}r_j}
    q^{\sum_{j=1}^{m-1}(k_j-\frac12)r_j -(\frac12+|\vec{k}|)r_m -\frac{m^2-1}{8}} 
    \left( \prod_{1\leq i<j\leq m}(r_i-r_j)^2 \right).
    \end{multline*}
Replacing $q^{1/2}$ by $-q$  and applying the well-known identity 
    \begin{equation}\label{eq:triangle}
    \triangle(q) = \prod_{n=1}^\infty\frac{1-q^{2n}}{1-q^{2n-1}},
    \end{equation}
we obtain the equation \eqref{result-A2}.
\end{proof}

% --------------------------------------------------------------------------
\subsection{Case $\widehat{\sl}(m+1,  m)$}\label{sec:sl^(m+1_m)}
% --------------------------------------------------------------------------

Let $\xi=-\frac12\sum_{i=1}^{m+1}\varepsilon_i$ so that $(\alpha_j,  \xi)=(-1)^j/2$ for $j=1,  \ldots,  2m$ and $w(\xi)=\xi$ for all $w\in W$.
Applying $t_\xi$ to the both sides of \eqref{affine-denominator},  we obtain
    \begin{equation}\label{t_xiA}
    t_\xi(e^{\widehat{\rho}}\widehat{R}) 
    = \frac{1}{m!}\sum_{\alpha^\sharp\in M^\sharp}\sum_{w\in W} 
    \epsilon(w)wt_{\xi+\alpha^\sharp}
    \left( \frac{e^{\widehat{\rho}}}{\prod_{\beta\in S}(1+e^{-\beta})} \right).
    \end{equation}

Write $\alpha^\sharp\in M^\sharp$ as $\alpha^\sharp=\sum_{i=1}^m (k_1+\cdots+k_i)(\alpha_{2i-1}+\alpha_{2i})=\sum_{i=1}^mk_i(\varepsilon_i-\varepsilon_{m+1})$ with $\vec{k}=(k_1,  \ldots,  k_m)\in\Z^m$.
We summarize some equations which we use to compute the right-hand side of \eqref{t_xiA}.
Each of them is a consequence of an easy calculation.
\begin{lemma}
We have the following equations: 
\begin{itemize}
\item $t_{\xi+\alpha^\sharp}(\alpha_{2j-1})=\alpha_{2j-1}-(k_j-1/2)\delta$. 
\item $t_{\xi+\alpha^\sharp}(\widehat{\rho})=\gamma+\xi+\alpha^\sharp-((m+1)/8+(\alpha^\sharp,  \alpha^\sharp)/2)\cdot\delta$.
\item $w(\delta)=\delta$ and $w(\gamma)=\gamma$ for $w\in W$.
\item $(\alpha^\sharp,  \alpha^\sharp)=2\sum_{i=1}^m k_i^2+2\sum_{1\leq i<j\leq m}k_ik_j
=2\sum_{1\leq i\leq j\leq m}k_ik_j$.
\end{itemize}
\end{lemma}

For an index set $J\subset\{1,  \ldots,  m\}$,  define $\Z^m_J$ as \eqref{eq:Z_J^m}.
The right-hand side of \eqref{t_xiA} becomes
    \begin{align*}
    & \frac{e^{\gamma}q^{\frac{m+1}{8}}}{m!}
    \sum_{J\subset\{1,\ldots,m\}}\sum_{\vec{k}\in\Z^m_J}
    q^{\sum_{1\leq i\leq j\leq m}k_ik_j} \\
    & \quad \times \sum_{w\in W}\epsilon(w)w\left( 
    e^{\xi+\alpha^\sharp} 
     \sum_{\vec{r}\in\Z_{\ge 0}^m}
    \prod_{j\in J} (-q^{-(k_j-\frac12)}e^{-\alpha_{2j-1}})^{r_j} \cdot
    \prod_{j\not\in J} \left(q^{k_j-\frac12}e^{\alpha_{2j-1}}
    (-q^{k_j-\frac12}e^{\alpha_{2j-1}})^{r_j}\right) \right) \\
    & = \frac{e^{\gamma}q^{\frac{m+1}{8}}}{m!}
    \sum_{J\subset\{1,\ldots,m\}}(-1)^{|J^c|}
    \sum_{\vec{k},  -\vec{r}\in\Z^m_J} (-1)^{\sum_{j=1}^m r_j}
    q^{\sum_{1\leq i\leq j\leq m}k_ik_j -\sum_{j=1}^m (k_j-\frac12)r_j}   \\
    & \qquad \times \sum_{w\in W}\epsilon(w)w\left( 
    e^{\xi+\alpha^\sharp-\sum_{j=1}^m r_j\alpha_{2j-1}} \right).
    \end{align*}

On the other hand,  since $t_\xi(\widehat{\rho})=\gamma+\xi-(m+1)/8\cdot\delta$,  $t_\xi(\delta)=\delta$, and $t_\xi(\beta)=\beta-(\beta,  \xi)\delta$ for $\beta\in\Phi$,  the left-hand side of \eqref{t_xiA} becomes 
    \[
    q^{\frac{m+1}{8}}e^{\gamma+\xi}
    \left( \prod_{n=1}^\infty (1-q^n)^{2m} \right)
    \frac{\prod_{\beta\in\Phi_0^+}
    \left(\prod_{n=0}^\infty(1-q^{n-(\beta,  \xi)}e^{-\beta})\cdot
    \prod_{n=1}^\infty(1-q^{n+(\beta,  \xi)}e^{\beta}) \right)}
    {\prod_{\beta\in\Phi_1^+}
    \left(\prod_{n=0}^\infty(1+q^{n-(\beta,  \xi)}e^{-\beta})\cdot
    \prod_{n=1}^\infty(1+q^{n+(\beta,  \xi)}e^{\beta} \right)}.
    \]
Hence we obtain 
    \begin{equation*}\label{heartA}\tag{$\clubsuit$}{}
    \begin{split}
    e^{\xi}
    \left( \prod_{n=1}^\infty (1-q^n)^{2m} \right)
    \frac{\prod_{\beta\in\Phi_0^+}
    \left(\prod_{n=0}^\infty(1-q^{n-(\beta,  \xi)}e^{-\beta})\cdot
    \prod_{n=1}^\infty(1-q^{n+(\beta,  \xi)}e^{\beta}) \right)}
    {\prod_{\beta\in\Phi_1^+}
    \left(\prod_{n=0}^\infty(1+q^{n-(\beta,  \xi)}e^{-\beta})\cdot
    \prod_{n=1}^\infty(1+q^{n+(\beta,  \xi)}e^{\beta} \right)} \\
    = \frac{1}{m!}
    \sum_{J\subset\{1,\ldots,m\}}(-1)^{|J^c|}
    \sum_{\vec{k},  -\vec{r}\in\Z^m_J} (-1)^{\sum_{j=1}^m r_j}
    q^{\sum_{1\leq i\leq j\leq m}k_ik_j -\sum_{j=1}^m (k_j-\frac12)r_j}   \\
     \times \sum_{w\in W}\epsilon(w)w\left( 
    e^{\xi+\alpha^\sharp-\sum_{j=1}^m r_j\alpha_{2j-1}} \right).
    \end{split}
    \end{equation*}

Let $R_0=\prod_{\beta\in\Phi_0^+}(1-e^{-\beta})$ be the denominator of $\Phi_0$.
From now,  we divide the both sides of \eqref{heartA} by $R_0$ and take the limit $e^{-\alpha_j}\to1$,  ($j=1,  \ldots,  2m$).

First,  we deal with the right-hand side.  

\begin{lemma}\label{lem:RHS-A}
We have 
    \begin{multline*}
    \lim_{\substack{e^{-\alpha_{j}}\to1 \\ j=1,  \ldots,  2m}}
    R_0^{-1}  \sum_{w\in W} \epsilon(w)w \left( e^
    {\xi+\alpha^\sharp-\sum_{j=1}^m r_j\alpha_{2j-1}} \right) \\
    \quad  = \frac{(-1)^{m(m-1)/2}m!}{\left( \prod_{i=1}^{m}i! \right)^2} 
    \left( \prod_{i=1}^m (k_i-r_i+|\vec{k}|) \right) 
    \left( \prod_{1\leq i<j\leq m}(k_i-r_i-k_j+r_j)(r_j-r_i) \right).
    \end{multline*}
\end{lemma}

\begin{proof}
Let $\rho_0=\frac12\sum_{\beta\in\Phi_0^+}\beta$ be the Weyl vector of $\Phi_0^+$.  
Set 
    \[
    \lambda = -\rho_0
    +\xi+\alpha^\sharp-\sum_{j=1}^m r_j\alpha_{2j-1}.
    \]
From the Weyl dimension formula (and its proof,  \cite[Corollary 24.3]{Humphreys}),  we obtain
    \[
    \lim_{\substack{e^{-\alpha_{j}}\to1 \\ j=1,  \ldots,  2m}}
    R_0^{-1} \sum_{w\in W} \epsilon(w)w(e^{\rho_0+\lambda}) \quad
    = \prod_{\beta\in\Phi_0^+}\frac{(\rho_0+\lambda,  \beta)}
    {(\rho_0,  \beta)}.
    \]

First,  $(\rho_0+\lambda,  \beta)$ for $\beta\in\Phi_0^+$ is computed as follows:
\begin{itemize}
\item $(\rho_0+\lambda,  \varepsilon_i-\varepsilon_j)=k_i-r_i-k_j+r_j$ for $1\leq i<j\leq m$.
\item $(\rho_0+\lambda,  \varepsilon_i-\varepsilon_{m+1})=k_i-r_i+\sum_{j=1}^mk_j$ for $1\leq i\leq m$.
\item $(\rho_0+\lambda,  \delta_i-\delta_j)=r_j-r_i$ for $1\leq i<j\leq m$.
\end{itemize}
Hence we get
    \[
    \prod_{\beta\in\Phi_0^+}(\rho_0+\lambda,  \beta)
    =\left(\prod_{i=1}^m (k_i-r_i+|\vec{k}|) \right) 
    \left(\prod_{1\leq i<j\leq m} (k_i-r_i-k_j+r_j)(r_j-r_i)\right).
    \]

On the other hand,  since we have 
    \[
    \rho_0=\frac12\sum_{p=1}^{m+1} (m+2-2p)\varepsilon_p
    +\frac12\sum_{q=1}^m(m+1-2q)\delta_q,
    \]
the pairing $(\rho_0,  \beta)$ for $\beta\in\Phi_0^+$ is computed as 
\begin{itemize}
\item $(\rho_0,  \varepsilon_i-\varepsilon_j)=j-i$ and 
    \[
    \prod_{1\leq i<j\leq m+1}(\rho_0,  \varepsilon_i-\varepsilon_j) 
    = \prod_{1\leq i<m+1}(m+1-i)! = \prod_{i=1}^{m} i! .
    \]
\item $(\rho_0,  \delta_i-\delta_j)=i-j$ and 
    \[
    \prod_{1\leq i<j\leq m}(\rho_0,  \delta_i-\delta_j) 
    =(-1)^{m(m-1)/2}\prod_{i=1}^{m-1} i! .
    \]
\end{itemize}
Hence we have $\prod_{\beta\in\Phi_0^+}(\rho_0,  \beta)=(-1)^{m(m-1)/2}(m!)^{-1}\left(\prod_{i=1}^{m}i!\right)^2$.
This completes the proof.
\end{proof}

Combining \cref{lem:RHS-A} with \eqref{heartA} and changing the variables $r_i\mapsto-r_i$,  $i=1,  \ldots,  m$,  we obtain the following equation.

\begin{corollary}\label{lem:RHS-A}
We have
    \begin{multline}\label{eq:RHS-A}
    \begin{split}
    \lim_{\substack{e^{-\alpha_{j}}\to1 \\ j=1,  \ldots,  2m}} 
    \frac{\mathrm{RHS}\text{ of }\eqref{heartA}}{R_0}  
     = \frac{(-1)^{m(m-1)/2}}{\left( \prod_{i=1}^{m}i! \right)^2}
    \sum_{J\subset\{1,  \ldots,  m\}}(-1)^{|J^c|}\sum_{\vec{k},  \vec{r}\in\Z_J^m}
    (-1)^{\sum_{j=1}^m r_j}
    q^{\sum_{1\leq i\leq j\leq m}k_ik_j +\sum_{j=1}^m (k_j-\frac12)r_j} \\
     \times \left( \prod_{i=1}^m (k_i+r_i+|\vec{k}|) \right) 
    \left( \prod_{1\leq i<j\leq m}(k_i+r_i-k_j-r_j)(r_i-r_j) \right).
    \end{split}
    \end{multline}
\end{corollary}

Next we deal with the left-hand side of \eqref{heartA}.
\begin{lemma}\label{lem:LHS-A}
    \begin{align}\label{eq:LHS-A}
    \lim_{\substack{e^{-\alpha_{j}}\to1 \\ j=1,  \ldots,  2m}} 
    \frac{\mathrm{LHS}\text{ of }\eqref{heartA}}{R_0} 
    & = q^{\frac{m(m+1)}{4}} \cdot 
    \left( \prod_{n=1}^\infty \frac{1-q^{n}}{1+q^{\frac12(2n-1)}} \right)^{2m(m+1)}.
    \end{align}
\end{lemma}
\begin{proof}
Let $\beta\in\Phi^+$.
It is easy to see that 
    \[
    (\beta,  \xi) = 
        \begin{cases}
        -\frac12 & \text{$\beta\in\{\varepsilon_i-\delta_j,  \ (1\leq i\leq j\leq m) \}$},  \\
        \frac12 & \text{$\beta\in\{\delta_i-\varepsilon_j, \ (1\leq i< j\leq m+1)\}$},  \\
        0 & \text{otherwise}.
        \end{cases}
    \]
%If $(\beta,  \xi)=0$,  then 
%    \begin{align*}
%    (1-e^{-\beta})^{-1}
%    \prod_{n=0}^\infty(1-q^{n-(\beta,  \xi)}e^{-\beta})
%    \prod_{n=1}^\infty(1-q^{n+(\beta,  \xi)}e^{\beta}) 
%    & = \prod_{n=1}^\infty (1-q^ne^{-\beta})(1-q^ne^\beta) \\
%    & \quad \underset{j=1,  \ldots,  2m}{\xrightarrow{e^{-\alpha_{j}}\to1}} \quad
%    \left(\prod_{n=1}^\infty (1-q^n)\right)^2.
%    \end{align*}
%If $(\beta,  \xi)=-1/2$,  then 
%    \begin{multline*}
%    \prod_{n=0}^\infty(1+q^{n-(\beta,  \xi)}e^{-\beta})
%    \prod_{n=1}^\infty(1+q^{n+(\beta,  \xi)}e^{\beta})
%    = \prod_{n=1}^\infty (1+q^{n-\frac12}e^{-\beta})
%    (1+q^{n-\frac12}e^{\beta}) \\
%    \underset{j=1,  \ldots,  2m}{\xrightarrow{e^{-\alpha_{j}}\to1}} \
%     \left(\prod_{n=1}^\infty(1+q^{\frac12(2n-1)})\right)^2.
%    \end{multline*}
%If $(\beta,  \xi)=1/2$,  then 
%    \begin{multline*}
%    \prod_{n=0}^\infty(1+q^{n-(\beta,  \xi)}e^{-\beta})
%    \prod_{n=1}^\infty(1+q^{n+(\beta,  \xi)}e^{\beta})
%    = (1+q^{-\frac12}e^{-\beta})\prod_{n=1}^\infty (1+q^{n-\frac12}e^{-\beta})
%    (1+q^{n+\frac12}e^{\beta}) \\
%    = q^{-\frac12}e^{-\beta}\cdot 
%    \prod_{n=1}^\infty(1+q^{\frac12(2n-1)}e^{-\beta})
%    (1+q^{\frac12(2n-1)}e^{\beta}) \
%    \underset{j=1,  \ldots,  2m}{\xrightarrow{e^{-\alpha_{j}}\to1}} \
%     q^{-\frac12}\left(\prod_{n=1}^\infty(1+q^{\frac12(2n-1)})\right)^2.
%    \end{multline*}

Note that $\#\{\beta\in\Phi^+ \mid (\beta,  \xi)=1/2\}=m(m+1)/2$.
Similarly as in the proof of \cref{lem:LHS-A2},  we obtain
    \[
    \lim_{\substack{e^{-\alpha_{j}}\to1 \\ j=1,  \ldots,  2m}} 
    \frac{\mathrm{LHS}\text{ of }\eqref{heartA}}{R_0} 
    = q^{\frac{m(m+1)}{4}} \cdot 
    \left( \prod_{n=1}^\infty \frac{1-q^n}{1+q^{\frac12(2n-1)}} \right)^{2m(m+1)}.
    \]
\end{proof}

From \cref{lem:RHS-A} and \cref{lem:LHS-A}  we obtain the next theorem.

\begin{theorem}\label{thm:denom-id-sl}
For $m\in\Z_{>0}$,  we have
    \begin{equation}\label{result-A}
    \begin{split}
    \triangle(q)^{2m(m+1)} = \frac{1}{\left( \prod_{i=1}^{m}i! \right)^2}
    \sum_{J\subset\{1,  \ldots,  m\}}(-1)^{|J|}\sum_{\vec{k},  \vec{r}\in\Z_J^m}
    q^{2\sum_{1\leq i\leq j\leq m}k_ik_j +\sum_{j=1}^m (2k_j-1)r_j
    -\frac{m(m+1)}{2}} \\
     \times \left( \prod_{i=1}^m (k_i+r_i+|\vec{k}|) \right) 
    \left( \prod_{1\leq i<j\leq m}(k_i+r_i-k_j-r_j)(r_i-r_j) \right).
    \end{split}
    \end{equation}
Here, we set $\Z^m_J = \left\{\vec{k}=(k_1,  \ldots,  k_m)\in\Z^m \, \middle|\,
J=\{j \mid k_j\leq 0\} \right\}$.
\end{theorem}

\begin{proof}
From \eqref{eq:RHS-A} and \eqref{eq:LHS-A},  we obtain
    \begin{multline*}
    \left( \prod_{n=1}^\infty \frac{1-q^{\frac12\cdot 2n}}{1+q^{\frac12(2n-1)}} \right)^{2m(m+1)} \\
    = \frac{(-1)^{m(m-1)/2}}{\left( \prod_{i=1}^{m}i! \right)^2}
    \sum_{J\subset\{1,  \ldots,  m\}}(-1)^{|J^c|}\sum_{\vec{k},  \vec{r}\in\Z_J^m}
    (-1)^{\sum_{j=1}^m r_j}
    q^{\frac12\left(\sum_{1\leq i\leq j\leq m}2k_ik_j +\sum_{j=1}^m (2k_j-1)r_j
    -\frac{m(m+1)}{2}\right)} \\
     \times \left( \prod_{i=1}^m (k_i+r_i+|\vec{k}|) \right) 
    \left( \prod_{1\leq i<j\leq m}(k_i+r_i-k_j-r_j)(r_i-r_j) \right).
    \end{multline*}
Replacing $q^{1/2}$ by $-q$ and applying \eqref{eq:triangle},  the equation \eqref{result-A} follows.
\end{proof}

% --------------------------------------------------------------------------
\subsection{Case $\widehat{\spo}(2m,  2m)$}\label{sec:spo^(2m_2m)}
% --------------------------------------------------------------------------

Let $\xi=-\sum_{i=1}^m\varepsilon_i$ so that $(\alpha_{2j-1},  \xi)=-1/2$ for $j=1,  \ldots,  m$ and $w(\xi)-\xi\in M^\sharp$ for all $w\in W$. 
Applying $t_\xi$ to the both sides of \eqref{affine-denominator},  we obtain
    \begin{equation}\label{t_xiD}
    t_\xi(e^{\widehat{\rho}}\widehat{R}) 
    = \frac{1}{2^{m-1}m!}\sum_{\alpha^\sharp\in M^\sharp}\sum_{w\in W} 
    \epsilon(w)wt_{\xi+\alpha^\sharp}
    \left( \frac{e^{\widehat{\rho}}}{\prod_{\beta\in S}(1+e^{-\beta})} \right).
    \end{equation}

Write $\alpha^\sharp\in M^\sharp$ as $\alpha^\sharp=\sum_{i=1}^mk_i\cdot 2\varepsilon_i$ with $\vec{k}=(k_1,  \ldots,  k_m)\in\Z^m$. 
We summarize some equations which we use to compute the right-hand side of \eqref{t_xiD}.
Each of them is a consequence of an easy calculation.
\begin{lemma}
We have the following equations: 
\begin{itemize}
\item $t_{\xi+\alpha^\sharp}(\alpha_{2j-1})=\alpha_{2j-1}-(k_j-1/2)\delta$. 
\item $t_{\xi+\alpha^\sharp}(\widehat{\rho})=\gamma+\xi+\alpha^\sharp-(m/4-(k_1+\cdots+k_m)+(\alpha^\sharp,  \alpha^\sharp)/2)\cdot\delta$.
\item $w(\delta)=\delta$ and $w(\gamma)=\gamma$ for $w\in W$.
\item $(\alpha^\sharp,  \alpha^\sharp)=2\sum_{i=1}^m k_i^2$.
\end{itemize}
\end{lemma}

The right-hand side of \eqref{t_xiD} becomes 
    \begin{align*}
    & \frac{e^{\gamma}q^{\frac{m}{4}}}{2^{m-1}m!}
    \sum_{J\subset\{1,\ldots,m\}}\sum_{\vec{k}\in\Z^m_J}
    q^{\sum_{j=1}^mk_j^2-k_j} \\
    & \times \sum_{w\in W}\epsilon(w)w\left( 
    e^{\xi+\alpha^\sharp} 
     \sum_{\vec{r}\in\Z_{\ge 0}^m}
    \prod_{j\in J} (-q^{-(k_j-\frac12)}e^{-\alpha_{2j-1}})^{r_j} \cdot
    \prod_{j\not\in J} \left(q^{k_j-\frac12}e^{\alpha_{2j-1}}
    (-q^{k_j-\frac12}e^{\alpha_{2j-1}})^{r_j}\right) \right) \\
    & = \frac{e^{\gamma}q^{\frac{m}{4}}}{2^{m-1}m!}
    \sum_{J\subset\{1,\ldots,m\}}(-1)^{|J^c|}
    \sum_{\substack{ \vec{k}\in\Z^m_J,   \\
    -\vec{r}\in\Z^m_{J}}} (-1)^{\sum_{j=1}^m r_j}
    q^{\sum_{j=1}^m k_j^2-k_j-k_jr_j+\frac12r_j}  
    \sum_{w\in W}\epsilon(w)w\left( 
    e^{\xi+\alpha^\sharp
    -\sum_{j=1}^m r_j\alpha_{2j-1}} \right).
    \end{align*}
Here,  $\Z_J^m$ is defined as \eqref{eq:Z_J^m}.

On the other hand,  since $t_\xi(\widehat{\rho})=\gamma+\xi-m/4\cdot\delta$,  $t_\xi(\delta)=\delta$, and $t_\xi(\beta)=\beta-(\beta,  \xi)\delta$ for $\beta\in\Phi$,  the left-hand side of \eqref{t_xiD} becomes 
    \[
    q^{\frac{m}{4}}e^{\gamma+\xi}
    \left( \prod_{n=1}^\infty (1-q^n)^{2m} \right)
    \frac{\prod_{\beta\in\Phi_0^+}
    \left(\prod_{n=0}^\infty(1-q^{n-(\beta,  \xi)}e^{-\beta})\cdot
    \prod_{n=1}^\infty(1-q^{n+(\beta,  \xi)}e^{\beta}) \right)}
    {\prod_{\beta\in\Phi_1^+}
    \left(\prod_{n=0}^\infty(1+q^{n-(\beta,  \xi)}e^{-\beta})\cdot
    \prod_{n=1}^\infty(1+q^{n+(\beta,  \xi)}e^{\beta} \right)}.
    \]
Hence we obtain 
    \begin{multline*}\label{heartD}\tag{$\diamondsuit$}{}
    \begin{split}
    e^{\xi}
    \left( \prod_{n=1}^\infty (1-q^n)^{2m} \right)
    \frac{\prod_{\beta\in\Phi_0^+}
    \left(\prod_{n=0}^\infty(1-q^{n-(\beta,  \xi)}e^{-\beta})\cdot
    \prod_{n=1}^\infty(1-q^{n+(\beta,  \xi)}e^{\beta}) \right)}
    {\prod_{\beta\in\Phi_1^+}
    \left(\prod_{n=0}^\infty(1+q^{n-(\beta,  \xi)}e^{-\beta})\cdot
    \prod_{n=1}^\infty(1+q^{n+(\beta,  \xi)}e^{\beta} \right)} \\
    = \frac{1}{2^{m-1}m!}
    \sum_{J\subset\{1,\ldots,m\}}(-1)^{|J^c|}
    \sum_{\substack{\vec{k},\in\Z^m_J  \\ -\vec{r}\in\Z^m_J}} 
    (-1)^{\sum_{j=1}^m r_j} q^{\sum_{j=1}^m k_j^2-k_j-k_jr_j+\frac12r_j}   \\
     \times \sum_{w\in W}\epsilon(w)w\left( 
    e^{\xi+\alpha^\sharp-\sum_{j=1}^m r_j\alpha_{2j-1}} \right).
    \end{split}
    \end{multline*}

Let $R_0=\prod_{\beta\in\Phi_0^+}(1-e^{-\beta})$ be the denominator of $\Phi_0$.
From now,  we divide the both sides of \eqref{heartD} by $R_0$ and take the limit $e^{-\alpha_j}\to1$,  ($j=1,  \ldots,  2m$).

First,  we deal with the right-hand side.  

\begin{lemma}\label{lem:RHS-D}
We have
    \begin{multline*}
    \lim_{\substack{e^{-\alpha_{j}}\to1 \\ j=1,  \ldots,  2m}}
    R_0^{-1}  \sum_{w\in W} \epsilon(w)w \left( e^
    {\xi+\alpha^\sharp-\sum_{j=1}^m r_j\alpha_{2j-1}}\right) \\
    \quad  = \frac{(2m-1)!}{(m-1)!\left(\prod_{i=1}^m(2i-1)!\right)^2} 
    \left(\prod_{i=1}^m (2k_i-r_i-1)\right) 
    \left(\prod_{1\leq i<j\leq m} (r_i-r_j)(r_i+r_j)\right) \\
    \times  
    \left(\prod_{1\leq i<j\leq m}(2k_i-r_i-2k_j+r_j)(2k_i-r_i+2k_j-r_j-2) \right).
    \end{multline*}
\end{lemma}

\begin{proof}
Let $\rho_0=\frac12\sum_{\beta\in\Phi_0^+}\beta$ be the Weyl vector of $\Phi_0^+$.  
Set 
    \[
    \lambda = -\rho_0
    +\xi+\alpha^\sharp-\sum_{j=1}^m r_j\alpha_{2j-1}
    = -\rho_0+\sum_{i=1}^m(2k_i-1)\varepsilon_i
    -\sum_{j=1}^m r_j\alpha_{2j-1}.
    \]
From the Weyl dimension formula (and its proof,  \cite[Corollary 24.3]{Humphreys}),  we obtain
    \[
    \lim_{\substack{e^{-\alpha_{j}}\to1 \\ j=1,  \ldots,  2m}}
    R_0^{-1} \sum_{w\in W} \epsilon(w)w(e^{\rho_0+\lambda}) \quad
    = \prod_{\beta\in\Phi_0^+}\frac{(\rho_0+\lambda,  \beta)}
    {(\rho_0,  \beta)}.
    \]

First,  $(\rho_0+\lambda,  \beta)$ for $\beta\in\Phi_0^+$ is computed as follows:
\begin{itemize}
\item $(\rho_0+\lambda,  \varepsilon_i-\varepsilon_j)=(2k_i-r_i-2k_j+r_j)/2$ for $1\leq i<j\leq m$.
\item $(\rho_0+\lambda,  \varepsilon_i+\varepsilon_j)=(2k_i-r_i+2k_j-r_j-2)/2$ for $1\leq i<j\leq m$.
\item $(\rho_0+\lambda,  2\varepsilon_p)=2k_p-r_p-1$ for $1\leq p\leq m$.
\item $(\rho_0+\lambda,  \delta_i-\delta_j)=(r_j-r_i)/2$ for $1\leq i<j\leq m$.
\item $(\rho_0+\lambda,  \delta_i+\delta_j)=-(r_i+r_j)/2$ for $1\leq i<j\leq m$.
\end{itemize}
Hence we get
    \begin{align*}
    \prod_{\beta\in\Phi_0^+}(\rho_0+\lambda,  \beta)
    =2^{-2m(m-1)} \left(\prod_{i=1}^m (2k_i-r_i-1)\right) 
    \left(\prod_{1\leq i<j\leq m} (r_i-r_j)(r_i+r_j)\right) \\
    \times  
    \left(\prod_{1\leq i<j\leq m}(2k_i-r_i-2k_j+r_j)(2k_i-r_i+2k_j-r_j-2) \right).
    \end{align*}

On the other hand,  since we have $\rho_0=\sum_{p=1}^m (m+1-p)\varepsilon_p+\sum_{q=1}^m(m-q)\delta_q$,  $(\rho_0,  \beta)$ for $\beta\in\Phi_0^+$ is computed as 
\begin{itemize}
\item $(\rho_0,  \varepsilon_i-\varepsilon_j)=(j-i)/2$ and 
    \[
    \prod_{1\leq i<j\leq m}(\rho_0,  \varepsilon_i-\varepsilon_j)
    = 2^{-m(m-1)/2}\prod_{1\leq i<m}(m-i)! = 2^{-m(m-1)/2}\prod_{i=1}^{m-1}i!.
    \]
\item $(\rho_0,  \varepsilon_i+\varepsilon_j)=(2m+2-i-j)/2$ and
    \[
    \prod_{1\leq i<j\leq m}(\rho_0,  \varepsilon_i+\varepsilon_j)
    = 2^{-m(m-1)/2}\prod_{1\leq i<m}\frac{(2m-2i+1)!}{(m+1-i)!}
    = 2^{-m(m-1)/2}\prod_{i=1}^{m-1}\frac{(2i+1)!}{(i+1)!}.
    \]
\item $(\rho_0,  2\varepsilon_p)=m+1-p$ and $\prod_{p=1}^m(\rho_0,  2\varepsilon_p)=m!$.

\item $(\rho_0,  \delta_i-\delta_j)=(i-j)/2$ and $\prod_{1\leq i<j\leq m}(\rho_0,  \delta_i-\delta_j)=(-2)^{-m(m-1)/2}\prod_{i=1}^{m-1}i!$.
\item $(\rho_0,  \delta_i+\delta_j)=(i+j-2m)/2$ and
    \[
    \prod_{1\leq i<j\leq m}(\rho_0,  \delta_i+\delta_j) 
    = (-2)^{-m(m-1)/2}\prod_{1\leq i<m}\frac{(2m-2i-1)!}{(m-i-1)!}
    = (-2)^{-m(m-1)/2}\prod_{i=1}^{m-1}\frac{(2i-1)!}{(i-1)!}.
    \]
\end{itemize}
Hence we have $\prod_{\beta\in\Phi_0^+}(\rho_0,  \beta)=2^{-2m(m-1)}(m-1)!((2m-1)!)^{-1}\left(\prod_{i=1}^{m}(2i-1)!\right)^2$.
This completes the proof.
\end{proof}

Combining \cref{lem:RHS-D} with \eqref{heartD} and changing the variables $r_i\mapsto-r_i$,  $i=1,  \ldots,  m$,  we obtain the following equation.

\begin{corollary}
We have
    \begin{align}\label{eq:RHS-D}
    & \lim_{\substack{e^{-\alpha_{j}}\to1 \\ j=1,  \ldots,  2m}} 
    \frac{\mathrm{RHS}\text{ of }\eqref{heartD}}{R_0}  \\ \nonumber 
    & \quad = \frac{(2m-1)!}{2^{m-1}m!(m-1)!\left(\prod_{i=1}^m(2i-1)!\right)^2} 
    \sum_{J\subset\{1,  \ldots,  m\}} (-1)^{|J^c|}
    \sum_{\vec{k},  \vec{r}\in\Z^m_J} (-1)^{\sum_{j=1}^m r_j}
    q^{\sum_{j=1}^m k_j^2-k_j+k_jr_j-\frac12r_j}   \\ \nonumber
    & \hspace{20pt} \times  \left(\prod_{i=1}^m (r_i+2k_i-1)\right)
    \left(\prod_{1\leq i<j\leq m}(r_i-r_j)(r_i+r_j)(r_i+2k_i-r_j-2k_j)(r_i+2k_i+r_j+2k_j-2) \right).
    \end{align}
\end{corollary}

Next we deal with the left-hand side of \eqref{heartD}.

\begin{lemma}\label{lem:LHS-D}
We have
    \begin{equation}\label{eq:LHS-D}
    \lim_{\substack{e^{-\alpha_{j}}\to1 \\ j=1,  \ldots,  2m}} 
    \frac{\text{LHS of }\eqref{heartD}}{R_0} 
    = (-1)^{\frac{m(m+1)}{2}} q^{\frac{m(m-1)}{4}} \cdot 
    \left(\frac{\prod_{n=1}^\infty(1-q^n)}
    {\prod_{n=1}^\infty\left(1+q^{\frac12(2n-1)}\right)}\right)^{4m^2}.
    \end{equation}
\end{lemma}

\begin{proof}
Let $\beta\in\Phi^+$.
It is easy to see that 
    \[
    (\beta,  \xi) = 
        \begin{cases}
        -1 & \text{$\beta\in\{\varepsilon_i+\varepsilon_j, \ (1\leq i<j\leq m),  \
         2\varepsilon_p,  \ (1\leq p\leq m)\}$},  \\
        -\frac12 & \text{$\beta\in\{\varepsilon_i-\delta_j,  \ (1\leq i\leq j\leq m),  \ 
        \varepsilon_i+\delta_j, \ (1\leq i,  j\leq m) \}$},  \\
        \frac12 & \text{$\beta\in\{\delta_i-\varepsilon_j, \ (1\leq i< j\leq m)\}$},  \\
        0 & \text{otherwise}.
        \end{cases}
    \]
%If $(\beta,  \xi)=0$,  then 
%    \begin{align*}
%    (1-e^{-\beta})^{-1}
%    \prod_{n=0}^\infty(1-q^{n-(\beta,  \xi)}e^{-\beta})
%    \prod_{n=1}^\infty(1-q^{n+(\beta,  \xi)}e^{\beta}) 
%    & = \prod_{n=1}^\infty (1-q^ne^{-\beta})(1-q^ne^\beta) \\
%    & \quad \underset{j=1,  \ldots,  2m}{\xrightarrow{e^{\alpha_{j}}\to1}} \quad
%    \left(\prod_{n=1}^\infty (1-q^n)\right)^2.
%    \end{align*}
If $(\beta,  \xi)=-1$,  then 
    \begin{multline*}
    (1-e^{-\beta})^{-1}
    \prod_{n=0}^\infty(1-q^{n-(\beta,  \xi)}e^{-\beta})
    \prod_{n=1}^\infty(1-q^{n+(\beta,  \xi)}e^{\beta}) 
    = -e^{\beta} \prod_{n=1}^\infty 
    (1-q^ne^{-\beta})(1-q^ne^{\beta}) \\
    \underset{j=1,  \ldots,  2m}{\xrightarrow{e^{-\alpha_{j}}\to1}} \
    - \left(\prod_{n=1}^\infty(1-q^n)\right)^2.
    \end{multline*}
%If $(\beta,  \xi)=-\frac12$,  then 
%    \begin{multline*}
%    \prod_{n=0}^\infty(1+q^{n-(\beta,  \xi)}e^{-\beta})
%    \prod_{n=1}^\infty(1+q^{n+(\beta,  \xi)}e^{\beta})
%    = \prod_{n=1}^\infty (1+q^{n-\frac12}e^{-\beta})
%    (1+q^{n-\frac12}e^{\beta}) \\
%    \underset{j=1,  \ldots,  2m}{\xrightarrow{e^{\alpha_{j}}\to1}} \
%     \left(\prod_{n=1}^\infty(1+q^{\frac12(2n-1)})\right)^2.
%    \end{multline*}
%If $(\beta,  \xi)=\frac12$,  then 
%    \begin{multline*}
%    \prod_{n=0}^\infty(1+q^{n-(\beta,  \xi)}e^{-\beta})
%    \prod_{n=1}^\infty(1+q^{n+(\beta,  \xi)}e^{\beta})
%    = (1+q^{-\frac12}e^{-\beta})\prod_{n=1}^\infty (1+q^{n-\frac12}e^{-\beta})
%    (1+q^{n+\frac12}e^{\beta}) \\
%    = q^{-\frac12}e^{-\beta}\cdot 
%    \prod_{n=1}^\infty(1+q^{\frac12(2n-1)}e^{-\beta})
%    (1+q^{\frac12(2n-1)}e^{\beta}) \
%    \underset{j=1,  \ldots,  2m}{\xrightarrow{e^{\alpha_{j}}\to1}} \
%     q^{-\frac12}\left(\prod_{n=1}^\infty(1+q^{\frac12(2n-1)})\right)^2.
%    \end{multline*}

Note that $\#\{\beta\in\Phi^+ \mid (\beta,  \xi)=-1\}=\frac{m(m+1)}{2}$ and $\#\{\beta\in\Phi^+ \mid (\beta,  \xi)=\frac12\}=\frac{m(m-1)}{2}$. 
Similarly as in the proof of \cref{lem:LHS-A2},  we obtain
    \[
    \lim_{\substack{e^{-\alpha_{j}}\to1 \\ j=1,  \ldots,  2m}} 
    \frac{\text{LHS of }\eqref{heartD}}{R_0} 
    = (-1)^{\frac{m(m+1)}{2}} q^{\frac{m(m-1)}{4}} \cdot 
    \left(\frac{\prod_{n=1}^\infty(1-q^n)}
    {\prod_{n=1}^\infty\left(1+q^{\frac12(2n-1)}\right)}\right)^{4m^2}.
    \]
\end{proof}

Replacing $q^{\frac12}$ by $-q$ in \eqref{eq:RHS-D} and \eqref{eq:LHS-D} and applying \eqref{eq:triangle},  we obtain the next theorem.

\begin{theorem}
For $m\in\Z_{>0}$,  we have
    \begin{align*}
    \triangle(q)^{4m^2} & = \frac{(2m-1)!}{2^{m-1}m!(m-1)!\left(\prod_{i=1}^m(2i-1)!\right)^2} \\
    & \quad \times \sum_{J\subset\{1,  \ldots,  m\}} (-1)^{|J|}
    \sum_{\vec{k},  \vec{r}\in\Z_J^m} 
    q^{\sum_{j=1}^m (2(k_j^2+k_jr_j-k_j)-r_j) -\frac{m(m-1)}{2}}  
    \left(\prod_{i=1}^m (r_i+2k_i-1)\right) \\
    & \quad \times  
    \left(\prod_{1\leq i<j\leq m}(r_i-r_j)(r_i+r_j)(r_i+2k_i-r_j-2k_j)(r_i+2k_i+r_j+2k_j-2) \right).
    \end{align*}
Here, we set $\Z^m_J = \left\{\vec{k}=(k_1,  \ldots,  k_m)\in\Z^m \, \middle|\,
J=\{j \mid k_j\leq0\} \right\}$.
\end{theorem}

By changing variables as $k_j = x_j + 1/2$ and $r_j = y_j - x_j$, and using the symmetry, we obtain \cref{thm:Dmm-Eisen}.

% --------------------------------------------------------------------------
\subsection{Case $\widehat{\spo}(2m,  2m+2)$}\label{sec:spo^(2m_2m+2)}
% --------------------------------------------------------------------------

Let $\xi=\sum_{i=1}^m\varepsilon_i$ so that $(\alpha_{2j-1},  \xi)=-1/2$ for $j=1,  \ldots,  m$ and $w(\xi)-\xi\in M^\sharp$ for all $w\in W$.
Applying $t_\xi$ to the both sides of \eqref{affine-denominator},  we obtain
    \begin{equation}\label{t_xiD2}
    t_\xi(\widehat{R}) \prod_{n=1}^\infty(1-q^n)^{-1} 
    = \frac{1}{2^m m!}\sum_{\alpha^\sharp\in M^\sharp}\sum_{w\in W} 
    \epsilon(w)wt_{\xi+\alpha^\sharp}
    \left( \frac{1}{\prod_{\beta\in S}(1+e^{-\beta})} \right).
    \end{equation}

Write $\alpha^\sharp\in M^\sharp$ as $\alpha^\sharp=\sum_{i=1}^mk_i\cdot 2\varepsilon_i$ with $\vec{k}=(k_1,  \ldots,  k_m)\in\Z^m$.
We summarize some equations which we use to compute the right-hand side of \eqref{t_xiD2}.
Each of them is a consequence of an easy calculation.
\begin{lemma}
We have the following equations: 
\begin{itemize}
\item $t_{\xi+\alpha^\sharp}(\alpha_{2j-1})=\alpha_{2j-1}+(k_j+1/2)\delta$. 
\item $w(\delta)=\delta$ and $w(\gamma)=\gamma$ for $w\in W$.
\end{itemize}
\end{lemma}

The right-hand side of \eqref{t_xiD2} becomes 
    \begin{align*}
    & \frac{1}{2^m m!}
    \sum_{J\subset\{1,\ldots,m\}}\sum_{-\vec{k}\in\Z^m_J} \\
    & \times \sum_{w\in W}\epsilon(w)w\left( 
     \sum_{\vec{r}\in\Z_{\ge 0}^m}
    \prod_{j\in J} (-q^{k_j+\frac12}e^{-\alpha_{2j-1}})^{r_j} \cdot
    \prod_{j\not\in J} \left(q^{-(k_j+\frac12)}e^{\alpha_{2j-1}}
    (-q^{-(k_j+\frac12)}e^{\alpha_{2j-1}})^{r_j}\right) \right) \\
    & = \frac{1}{2^m m!}
    \sum_{J\subset\{1,\ldots,m\}}(-1)^{|J^c|}
    \sum_{-\vec{k}, -\vec{r} \in\Z^m_J} (-1)^{\sum_{j=1}^m r_j}
    q^{\sum_{j=1}^m k_jr_j + \frac12r_j}  
    \sum_{w\in W}\epsilon(w)w\left( 
    e^{-\sum_{j=1}^m r_j\alpha_{2j-1}} \right).
    \end{align*}
Here,  $\Z_J^m$ is defined as \eqref{eq:Z_J^m}.

On the other hand,  since $t_\xi(\delta)=\delta$ and $t_\xi(\beta)=\beta-(\beta,  \xi)\delta$ for $\beta\in\Phi$,  the left-hand side of \eqref{t_xiD2} becomes 
    \[
    \left( \prod_{n=1}^\infty (1-q^n)^{2m} \right)
    \frac{\prod_{\beta\in\Phi_0^+}
    \left(\prod_{n=0}^\infty(1-q^{n-(\beta,  \xi)}e^{-\beta})\cdot
    \prod_{n=1}^\infty(1-q^{n+(\beta,  \xi)}e^{\beta}) \right)}
    {\prod_{\beta\in\Phi_1^+}
    \left(\prod_{n=0}^\infty(1+q^{n-(\beta,  \xi)}e^{-\beta})\cdot
    \prod_{n=1}^\infty(1+q^{n+(\beta,  \xi)}e^{\beta} \right)}.
    \]
Hence we obtain 
    \begin{multline*}\label{heartD2}\tag{$\heartsuit$}{}
    \begin{split}
    \left( \prod_{n=1}^\infty (1-q^n)^{2m} \right)
    \frac{\prod_{\beta\in\Phi_0^+}
    \left(\prod_{n=0}^\infty(1-q^{n-(\beta,  \xi)}e^{-\beta})\cdot
    \prod_{n=1}^\infty(1-q^{n+(\beta,  \xi)}e^{\beta}) \right)}
    {\prod_{\beta\in\Phi_1^+}
    \left(\prod_{n=0}^\infty(1+q^{n-(\beta,  \xi)}e^{-\beta})\cdot
    \prod_{n=1}^\infty(1+q^{n+(\beta,  \xi)}e^{\beta} \right)} \\
    = \frac{1}{2^m m!}
    \sum_{J\subset\{1,\ldots,m\}}(-1)^{|J^c|}
    \sum_{-\vec{k}, -\vec{r} \in\Z^m_J} 
    (-1)^{\sum_{j=1}^m r_j} q^{\sum_{j=1}^m k_jr_j + \frac12r_j}   \\
     \times \sum_{w\in W}\epsilon(w)w\left( 
    e^{-\sum_{j=1}^m r_j\alpha_{2j-1}} \right).
    \end{split}
    \end{multline*}

Let $R_0=\prod_{\beta\in\Phi_0^+}(1-e^{-\beta})$ be the denominator of $\Phi_0$.
From now,  we divide the both sides of \eqref{heartD2} by $R_0$ and take the limit $e^{-\alpha_j}\to1$,  ($j=1,  \ldots,  2m$).
 
First,  we deal with the right-hand side.  

\begin{lemma}\label{lem:RHS-D2}
We have
    \begin{multline*}
    \lim_{\substack{e^{-\alpha_{j}}\to1 \\ j=1,  \ldots,  2m}}
    R_0^{-1}  \sum_{w\in W} \epsilon(w)w \left( e^
    {-\sum_{j=1}^m r_j\alpha_{2j-1}}\right)  \\
    \quad  = \frac{1}{m!\left(\prod_{i=1}^m(2i-1)!\right)^2} 
    \left(\prod_{i=1}^m r_i^3 \right) 
    \left(\prod_{1\leq i<j\leq m} (r_i-r_j)^2(r_i+r_j)^2\right).
    \end{multline*}
\end{lemma}

\begin{proof}
Let $\rho_0=\frac12\sum_{\beta\in\Phi_0^+}\beta$ be the Weyl vector of $\Phi_0^+$.  
Set 
    \[
    \lambda = -\rho_0 -\sum_{j=1}^m r_j\alpha_{2j-1}.
    \]
From the Weyl dimension formula (and its proof,  \cite[Corollary 24.3]{Humphreys}),  we obtain
    \[
    \lim_{\substack{e^{-\alpha_{j}}\to1 \\ j=1,  \ldots,  2m}}
    R_0^{-1} \sum_{w\in W} \epsilon(w)w(e^{\rho_0+\lambda}) \quad
    = \prod_{\beta\in\Phi_0^+}\frac{(\rho_0+\lambda,  \beta)}
    {(\rho_0,  \beta)}.
    \]

First,  $(\rho_0+\lambda,  \beta)$ for $\beta\in\Phi_0^+$ is computed as follows:
\begin{itemize}
\item $(\rho_0+\lambda,  \varepsilon_i-\varepsilon_j)=(r_i-r_j)/2$ for $1\leq i<j\leq m$.
\item $(\rho_0+\lambda,  \varepsilon_i+\varepsilon_j)=(r_i+r_j)/2$ for $1\leq i<j\leq m$.
\item $(\rho_0+\lambda,  2\varepsilon_p)=r_p$ for $1\leq p\leq m$.
\item $(\rho_0+\lambda,  \delta_i-\delta_j)=
    \begin{cases}
    (r_i-r_j)/2 & \text{$1\leq i<j\leq m$,} \\
    r_i/2 & \text{$1\leq i\leq m$,  $j=m+1$.}
    \end{cases}$
\item $(\rho_0+\lambda,  \delta_i+\delta_j)=
    \begin{cases}
    (r_i+r_j)/2 & \text{$1\leq i<j\leq m$,} \\
    r_i/2 & \text{$1\leq i\leq m$,  $j=m+1$.}
    \end{cases}$
\end{itemize}
Hence we get
    \[
    \prod_{\beta\in\Phi_0^+}(\rho_0+\lambda,  \beta)
    = 2^{-2m^2} \left( \prod_{i=1}^m r_i^3 \right)
    \left(\prod_{1\leq i<j\leq m}(r_i-r_j)^2(r_i+r_j)^2 \right).
    \]

On the other hand,  since we have $\rho_0=\sum_{p=1}^m (m+1-p)\varepsilon_p+\sum_{q=1}^{m+1}(m+1-q)\delta_q$,  $(\rho_0,  \beta)$ for $\beta\in\Phi_0^+$ is computed as 
\begin{itemize}
\item $(\rho_0,  \varepsilon_i-\varepsilon_j)=(j-i)/2$ and 
    \[
    \prod_{1\leq i<j\leq m}(\rho_0,  \varepsilon_i-\varepsilon_j)
    = 2^{-m(m-1)/2}\prod_{1\leq i<m}(m-i)! = 2^{-m(m-1)/2}\prod_{i=1}^{m-1}i!.
    \]
\item $(\rho_0,  \varepsilon_i+\varepsilon_j)=(2m+2-i-j)/2$ and
    \[
    \prod_{1\leq i<j\leq m}(\rho_0,  \varepsilon_i-\varepsilon_j)
    = 2^{-m(m-1)/2}\prod_{1\leq i<m}\frac{(2m-2i+1)!}{(m+1-i)!}
    = 2^{-m(m-1)/2}\prod_{i=1}^{m-1}\frac{(2i+1)!}{(i+1)!}.
    \]
\item $(\rho_0,  2\varepsilon_p)=m+1-p$ and $\prod_{p=1}^m(\rho_0,  2\varepsilon_p)=m!$.

\item $(\rho_0,  \delta_i-\delta_j)=(i-j)/2$ and 
    \[
    \prod_{1\leq i<j\leq m+1}(\rho_0,  \delta_i-\delta_j)
    = (-2)^{-m(m+1)/2}\prod_{1\leq i<m+1} (m+1-i)!
    = (-2)^{-m(m+1)/2}\prod_{i=1}^m i!.
    \]
\item $(\rho_0,  \delta_i+\delta_j)=(i+j-2m-2)/2$ and
    \[
    \prod_{1\leq i<j\leq m+1}(\rho_0,  \delta_i+\delta_j) 
    = (-2)^{-m(m+1)/2}\prod_{1\leq i<m+1}\frac{(2m-2i+1)!}{(m-i)!}
    = (-2)^{-m(m+1)/2}\prod_{i=1}^m\frac{(2i-1)!}{(i-1)!}.
    \]
\end{itemize}
Hence we have $\prod_{\beta\in\Phi_0^+}(\rho_0,  \beta)=2^{-2m^2}m!\left(\prod_{i=1}^{m}(2i-1)!\right)^2$.
This completes the proof.
\end{proof}

Combining \cref{lem:RHS-D2} with \eqref{heartD2} and changing the variables $r_i\mapsto-r_i$, $k_i \mapsto -k_i$, $i=1,  \ldots,  m$,  we obtain the following equation.

\begin{corollary}
We have
    \begin{align}\label{eq:RHS-D2}
    & \lim_{\substack{e^{-\alpha_{j}}\to1 \\ j=1,  \ldots,  2m}} 
    \frac{\mathrm{RHS}\text{ of }\eqref{heartD2}}{R_0}  \\ \nonumber 
    & \quad = \frac{1}{2^m(m!)^2\left(\prod_{i=1}^m(2i-1)!\right)^2} 
    \sum_{J\subset\{1,  \ldots,  m\}} (-1)^{|J|}
    \sum_{\vec{k},  \vec{r}\in\Z^m_J} (-1)^{\sum_{j=1}^m r_j}
    q^{\sum_{j=1}^m k_jr_j-\frac12r_j}   \\ \nonumber
    & \hspace{100pt} \times \left( \prod_{i=1}^m r_i^3 \right)
    \left(\prod_{1\leq i<j\leq m}(r_i-r_j)^2(r_i+r_j)^2 \right).
    \end{align}
\end{corollary}

Next we deal with the left-hand side of \eqref{heartD2}.

\begin{lemma}
We have
    \begin{equation}\label{eq:LHS-D2}
    \lim_{\substack{e^{-\alpha_{j}}\to1 \\ j=1,  \ldots,  2m}} 
    \frac{\text{LHS of }\eqref{heartD2}}{R_0} 
    = (-1)^{\frac{m(m+1)}{2}} q^{\frac{m(m+1)}{4}} \cdot 
    \left(\frac{\prod_{n=1}^\infty(1-q^n)}
    {\prod_{n=1}^\infty\left(1+q^{\frac12(2n-1)}\right)}\right)^{4m^2+4m}.
    \end{equation}
\end{lemma}

\begin{proof}
Let $\beta\in\Phi^+$.
It is easy to see that 
    \[
    (\beta,  \xi) = 
        \begin{cases}
        1 & \text{$\beta\in\{\varepsilon_i+\varepsilon_j, \ (1\leq i<j\leq m),  \
         2\varepsilon_p,  \ (1\leq p\leq m)\}$},  \\
        \frac12 & \text{$\beta\in\{\varepsilon_i-\delta_j,  \ (1\leq i< j\leq m+1),  \ 
        \varepsilon_i+\delta_j, \ (1\leq i\leq m,  \ 1\leq j\leq m+1) \}$},  \\
        -\frac12 & \text{$\beta\in\{\delta_i-\varepsilon_j, \ (1\leq i\leq j\leq m)\}$},  \\
        0 & \text{otherwise}.
        \end{cases}
    \]

Note that $\#\{\beta\in\Phi^+ \mid (\beta,  \xi)=1\}=\frac{m(m+1)}{2}$ and $\#\{\beta\in\Phi^+ \mid (\beta,  \xi)=\frac12\}=\frac{3m(m+1)}{2}$.
Similarly as in the proof of \cref{lem:LHS-A2} and \cref{lem:LHS-D},  we obtain
    \[
    \lim_{\substack{e^{-\alpha_{j}}\to1 \\ j=1,  \ldots,  2m}} 
    \frac{\text{LHS of }\eqref{heartD2}}{R_0} 
    = (-1)^{\frac{m(m+1)}{2}} q^{\frac{m(m+1)}{4}} \cdot 
    \left(\frac{\prod_{n=1}^\infty(1-q^n)}
    {\prod_{n=1}^\infty\left(1+q^{\frac12(2n-1)}\right)}\right)^{4m^2+4m}.
    \]
\end{proof}

Replacing $q^{\frac12}$ by $-q$ in \eqref{eq:RHS-D2} and \eqref{eq:LHS-D2} and applying \eqref{eq:triangle},  we obtain the next theorem.

\begin{theorem}
For $m\in\Z_{>0}$,  we have
    \begin{align*}
    \triangle(q)^{4m^2+4m} & = \frac{1}{2^m(m!)^2\left(\prod_{i=1}^m(2i-1)!\right)^2} \\
    & \quad \times \sum_{J\subset\{1,  \ldots,  m\}} (-1)^{|J|}
    \sum_{\vec{k},  \vec{r}\in\Z_J^m} 
    q^{\sum_{j=1}^m (2k_jr_j-r_j) -\frac{m(m+1)}{2}}  
    \left(\prod_{i=1}^m r_i^3 \right) 
    \left(\prod_{1\leq i<j\leq m}(r_i-r_j)^2(r_i+r_j)^2 \right).
    \end{align*}
Here,  we set $\Z^m_J = \left\{\vec{k}=(k_1,  \ldots,  k_m)\in\Z^m \, \middle|\,
J=\{j \mid k_j\leq0\} \right\}$.
\end{theorem}

By changing variables as $x_j = r_j$ and $y_j = k_j - 1/2$, and using the symmetry, we obtain the second formula in \cref{thm:Zagier-Q}.

% --------------------------------------------------------------------------
\subsection{Remark on \cref{thm:denom-id-gl} and \cref{thm:denom-id-sl}}\label{sec:Final-remark}
% --------------------------------------------------------------------------

To summarize, the second identity in \cref{thm:Zagier-Q} can be derived from the denominator identity for $\widehat{\mathfrak{spo}}(2m,2m+2)$, while \cref{thm:Dmm-Eisen} follows from that for $\widehat{\mathfrak{spo}}(2m,2m)$. Furthermore, although not discussed here, the first and second identities in \cref{thm:Zagier-Q} arise from the denominator identities for $\widehat{\mathfrak{q}}(2m-1)$ and $\widehat{\mathfrak{q}}(2m)$, respectively.
We expect that we can prove the two $q$-series identities obtained from $\widehat{\gl}(m, m)$ and $\widehat{\sl}(m+1,m)$ using the theory of indefinite theta functions, as in the other cases.
%in \cref{sec:proof_via_theta}. 
However,  these two cases differ from other cases at some points outlined below and we have not yet been able to obtain such proofs.

We begin by considering \cref{thm:denom-id-sl}. 
The identity is equivalent to
\begin{align*}
	&\theta_\triangle(\tau)^{2m(m+1)}\\
	&= \frac{1}{2^m \prod_{j=1}^m j!^2} \sum_{\vec{k}, \vec{r} \in \Z^m} \prod_{j=1}^m \bigg(\sgn(k_j) + \sgn(r_j)\bigg) g(\vec{k}, \vec{r}) q^{\sum_{1 \le i \le j \le m} k_i k_j + \sum_{j=1}^m (k_j-1/2)r_j - \frac{m(m+1)}{8}},
\end{align*}
where
    \[
	g(\vec{k}, \vec{r}) \coloneqq \prod_{j=1}^m (k_j + r_j + |\vec{k}|) \cdot \prod_{1 \le i < j \le m} (k_i + r_i - k_j - r_j)(r_i-r_j).
    \]
By changing variables as $k_j = x_j + 1/2$ and $r_j = y_j - \sum_{j \le i \le m} x_i - \frac{m+1}{2}$, the exponent of $q$ simplifies as
    \[
	\sum_{1 \le i \le j \le m} k_i k_j + \sum_{j=1}^m (k_j-1/2)r_j - \frac{m(m+1)}{8} = \sum_{j=1}^m x_j y_j.
    \]
For instance, when $m=2$, by taking some symmetries into account, we observe that
\begin{align*}
	\theta_\triangle(\tau)^{12} &= \frac{1}{32} \sum_{\substack{x_1, x_2, y_1 \in 1/2 + \Z \\ y_2 \in \Z}} \bigg(\sgn(x_1) + \sgn(y_1)\bigg)\bigg(\sgn(x_2) + \sgn(y_2 - x_2 ) \bigg) g_0(\vec{\bm{x}}) q^{x_1 y_1 + x_2 y_2},
\end{align*}
where
\[
	g_0(\vec{\bm{x}}) \coloneqq  (x_1 + y_1)(x_1 + x_2 + y_2)(x_1 - y_1 + y_2) (x_2 - y_1 + y_2).
\]
For $\vec{\bm{a}} = (\smat{1/2 \\ 1/2}, \smat{1/2 \\ 0}), \vec{\bm{c}}_0 = (\smat{0 \\ 1}, \smat{0 \\ 1}) \in \cS^2$, and $\vec{\bm{c}}_1 = (\smat{-1 \\ 0}, \smat{-1 \\ 1}) \in \cS \times \cC$, the right-hand side is related to the indefinite that function $\theta_{\vec{\bm{a}}, \vec{\bm{0}}}^{\vec{\bm{c}}_0, \vec{\bm{c}}_1}[g_0](\tau)$. 
However, since $\smat{-1 \\ 1} \in \cC$, the theta function is, a priori, a non-holomorphic modular form involving error functions. 
To establish an alternative proof in the framework of indefinite theta functions, one must verify that all non-holomorphic contributions, such as the $\beta$-terms arising from the decomposition of the error function as in~\eqref{eq:Error-beta}, and the parts arising from derivatives of the error functions, cancel out completely. 
The resulting theta function must be actually a holomorphic modular form.

The situation is more subtle in the case of \cref{thm:denom-id-gl}. 
At first glance, the identity seems to lie outside the framework developed in this article, as the summation variables have mismatched dimensions, namely, $\vec{k} \in \Z^{m-1}$ and $\vec{r} \in \Z^m$. 
A more careful discussion may be necessary to address these cases.

\bibliographystyle{amsalpha}
\bibliography{References}

\end{document}